\documentclass{amsart}
\usepackage{amssymb}
\usepackage{latexsym}
\usepackage{amsthm}
\usepackage[dvips]{graphicx}
\usepackage{epsf}

\theoremstyle{plain}

\newtheorem{thm}{Theorem}[section]
\newtheorem{cor}[thm]{Corollary}
\newtheorem{lem}[thm]{Lemma}
\newtheorem{prop}[thm]{Proposition}
\newtheorem{conj}[thm]{Conjecture}

\theoremstyle{remark}

\newtheorem{ex}[thm]{Example}

\newtheorem{defn}{Definition}[section]
\newtheorem{rem}{Remark}[section]

\newtheorem{prob}{Problem}
\newtheorem{obs}{Observation}

\numberwithin{equation}{section}

\newcommand{\ra}{\rightarrow}

\newcommand{\Gal}{\mbox{Gal}}
\newcommand{\fS}{\mathfrak{S}}

\newcommand{\ZZ}{\mathbb Z}
\newcommand{\CC}{\mathbb C}
\newcommand{\EEE}{\mathbb E}
\newcommand{\QQ}{\mathbb Q}

\newcommand{\NN}{\mathbb N}
\newcommand{\PP}{\mathbb P}

\newcommand{\FF}{\mathbb F}
\newcommand{\LL}{\mathbb L}
\newcommand{\KK}{\mathbb K}
\newcommand{\bA}{\mathbb A}

\newcommand{\GQ}{\Gal(\overline{\QQ}/\QQ)}
\newcommand{\GK}{\Gal(\overline{\QQ}/K)}
\newcommand{\MM}{\mathcal M}
\newcommand{\OO}{\mathcal O}
\newcommand{\EE}{\mathcal E}
\newcommand{\Ql}{\QQ_{\ell}}

\newcommand{\ba}{\mathbf a}

\newcommand{\fA}{\mathfrak{A}}
\newcommand{\fB}{\mathfrak{B}}
\newcommand{\fC}{\mathfrak{C}}

\newcommand{\fp}{\mathfrak{p}}

\newcommand{\lra}{\longrightarrow}

\newcommand{\rank}{{\rm rank}\, }
\newcommand{\pic}{{\rm Pic}\, }
\newcommand{\lcm}{{\rm lcm}\, }

\begin{document}

\title[Automorphy of Calabi--Yau threefolds
of Borcea--Voisin type over $\QQ$]{Automorphy of  
Calabi--Yau threefolds of Borcea--Voisin type over $\QQ$}
\author{Yasuhiro Goto}
\address{Department of Mathematics, Hokkaido University
of Education, 1-2 Hachiman-cho, Hakodate 040-8567 Japan}
\email{goto.yasuhiro@h.hokkyodai.ac.jp}
\author{Ron Livn\'e}
\address{Institute of Mathematics, Hebrew University
of Jerusalem Givat-Ram, Jerusalem 919004, Israel}
\email{rlivne@math.huji.ac.il}
\author{Noriko Yui}
\address{Department of Mathematics and Statistics, Queen's University,
Kingston, Ontario Canada K7L 3N6}
\email{yui@mast.queensu.ca}
\date{\today}
\begin{abstract}
We consider certain Calabi--Yau threefolds
of Borcea--Voisin type defined over $\QQ$. We will discuss the
automorphy of the Galois representations associated to these 
Calabi--Yau threefolds. We construct such Calabi--Yau 
threefolds as the quotients of products of $K3$ surfaces $S$ 
and elliptic curves by a specific involution. 
We choose $K3$ surfaces $S$ over $\QQ$ 
with non-symplectic involution $\sigma$ acting by $-1$ on 
$H^{2,0}(S)$. We fish out $K3$ surfaces with the involution 
$\sigma$ from the famous $95$ families of $K3$ surfaces in 
the list of Reid \cite{R79}, and of Yonemura \cite{Yo89}, where
Yonemura described hypersurfaces defining these $K3$ surfaces 
in weighted projective $3$-spaces.

Our first result is that for all but few (in fact, nine) of the
$95$ families of $K3$ surfaces $S$ over $\QQ$ in Reid--Yonemura 
list, there are subsets of equations defining quasi-smooth 
hypersurfaces which are of Delsarte or Fermat type and endowed with 
non-symplectic involution $\sigma$. One implication of this 
result is that with this choice of defining equation,  
$(S,\sigma)$ becomes of CM type.

Let $E$ be an elliptic curve
over $\QQ$ with the standard involution $\iota$, and let
$X$ be a standard (crepant) resolution, defined over $\QQ$,
of the quotient threefold $E\times S/\iota\times\sigma$,
where $(S,\sigma)$ is one of the above $K3$ surfaces 
over $\QQ$ of CM type.  One of our main results is the automorphy of 
the $L$-series of $X$.

The moduli spaces of these Calabi--Yau threefolds are
Shimura varieties.  Our result shows the existence of a 
CM point in the moduli space.

We also consider the $L$-series of mirror pairs of 
Calabi--Yau threefolds of Borcea--Voisin type, and study
how $L$-series behave under mirror symmetry.
\end{abstract}

\maketitle
\setcounter{tocdepth}{1}
\tableofcontents

\section{Introduction}\label{sect1}

We will address the automorphy
of the Galois representations associated to certain
Calabi-Yau threefolds of Borcea--Voisin type over $\QQ$.
Here by the automorphy, we mean the Langlands reciprocity
conjecture which claims that the $L$-series (of 
the %middle-dimensional 
$\ell$-adic \'etale cohomology group) of the Calabi-Yau 
threefolds over $\QQ$ come from automorphic 
representations.  % in $GL_n(\bA_{\QQ})$ for some $n$. 
For our Calabi--Yau threefolds, we will show that these 
representations arise as induced automorphic cuspidal
representations of $GL_2(K)$ of some abelian number 
fields $K$. 
%using a generalization of the work of Arthur--Clozel 
%\cite{AC90}.

Our Calabi--Yau threefolds were previously considered by
Voisin \cite{V93} and also by Borcea \cite{B94} from the point 
of view of geometry and also towards physics (mirror symmetry) 
applications.

We now describe briefly the Borcea--Voisin construction of 
Calabi--Yau threefolds over $\CC$.  Let $E$ be any elliptic 
curve with involution $\iota$, 
%$\iota:x\mapsto -x$, 
and let $S$ be 
a $K3$ surface with involution $\sigma$ acting by $-1$ on 
$H^{2,0}(S)$. 
%The Hodge structure of the product $E\times S$
%is given by the tensor product of the Hodge structures of 
%the components. 
The quotient threefold $E\times S/\iota\times \sigma$ is singular, 
but the singularities are all cyclic quotient singularities, and
there is an explicit crepant resolution, which yields a 
smooth Calabi--Yau threefold $X$. %The Hodge numbers 
%$h^{1,1}(X)$ and $h^{2,1}(X)$ are determined by the fixed 
%locus $S^{\sigma}$ of $\sigma$ on $S$.

To find our $K3$ surfaces, we use the famous $95$ families
of $K3$ surfaces which can be given by weighted homogeneous 
equations in weighted projective $3$-spaces. They are classified by 
M. Reid~\cite{R79} (see also Iano--Fletcher~\cite{IF20}),
and also in Yonemura~\cite{Yo89}. Yonemura 
gave explicit equations for these surfaces as weighted 
hypersurfaces $h(x_0,x_1,x_2,x_3)=0$ using toric 
methods, and we will use Yonemura's list throughout this article. 

We first fish out, from the list of Yonemura, %subfamilies $S$ of 
$K3$ surfaces $S$ having the required involutions $\sigma$ 
acting on the holomorphic $2$-forms of the surfaces as 
multiplication by $-1$.  Earlier, Borcea \cite{B94} 
found $48$ such pairs $(S, \sigma)$.% in Yonemura's list.  
We will find additional $41+(3)$ such pairs $(S,\sigma)$ % in Yonemura's list 
(our involutions may have a different formula from Borcea's 
examples), bringing the total to $92$ pairs $(S,\sigma)$.  

Nikulin \cite{Ni80} classified all $K3$ surfaces $(S,\sigma)$ over 
$\CC$ with non-symplectic involution $\sigma$ by triplets of integers 
$(r, a,\delta)$, and found that there are $75$ triplets up to deformation. 
In this paper, we calculate only the invariants $r$ and $a$ for our 
$92$ examples and realize at least $40$ Nikulin triplets $(r, a, \delta)$. 
As the task of calculating $\delta$ is more involved, especially 
because we often need a 
$\ZZ$-basis for $\pic (S)$, we leave the determination of the
invariant $\delta$ to a future publication(s).  
Since $\delta\in\{0,1\}$, the number of triplets realized may 
increase somewhat.

For $86$ of our $92$ pairs of $(S,\sigma)$ above, we find a 
representative hypersurface defining equation 
for $S$ of Delsarte type over $\QQ$, that is, the equation
consists exactly of four monomials with rational coefficients.  
%(the other terms are chosen as $0$). 
Since $S$ needs to be 
quasi-smooth, we put a condition on the defining equation 
(see Subsection \ref{sect2.2}). Then our new $S$ has the same 
singularity configuration as the original hypersurface.
%Toric methods show that 
%the chosen equation defines a $K3$ surface if it satisfies the 
%following conditions:
%(i) for any $i, \, 0\leq i\leq 3$, it must contain terms
%of the form $x_i^n$ or $x_i^nx_j\,(i\neq j)$ with non-zero
%coefficients; (ii) it must be quasi-smooth (cf. Subsection \ref{sect2.2}) 
%and (iii) it has the same singularity configuration as
%the original hypersurface. 
(We should call attention why we only have $86$ pairs: What happens to
the remaining $6$ pairs? This is because for the six weights, $K3$ surfaces
have involution but cannot be realized as quasi-smooth
hypersurfaces in four monomials.)

%Once these conditions are satisfied, 
Thus, we obtain $K3$ surfaces $S$ of 
Delsarte type. Recall that a cohomology group of a variety is of CM 
type if its Hodge group is commutative; 
and a variety is of CM type if all its cohomology groups are of 
CM type (see Zarhin, \cite{Za83}). In general the computation 
of Hodge groups is notoriously difficult, and this is definitely
not the direction we will pursue. Instead, we will follow the argument
similar to the one in Livn\'e-Sch\"utt--Yui \cite{LSY10}: 
a Delsarte surface $S$ can be realized as a quotient of a Fermat 
surface by some finite group. Since we know that Fermat (hyper)surfaces 
are of CM type, it follows that a Delsarte surface is also of CM type.

It is known \cite{B94}, \cite{R09} that over $\CC$ the moduli 
spaces of Nikulin's $K3$ families are  Shimura varieties. 
Recently, the rationality of the moduli spaces of all but two
out of the $75$ Nikulin's $K3$ families has been
established by Ma \cite{M11, M12}, combined with the results of
Kondo \cite{K94}, and Dolgachev--Kondo \cite{DK12}.

Our results give explicit CM points in these moduli spaces 
defined over $\QQ$; we do not know what their fields of 
definition (or moduli) are in the Shimura variety.

Next we take a product $E\times S$, where $E$ is an elliptic curve 
over $\QQ$ with the $-1$-involution $\iota$, and $S$ is a $K3$
surface of CM type over $\QQ$ with involution $\sigma$ as above. 
Take the quotient $E\times S/\iota\times\sigma$.
Let $X$ be a crepant resolution of the quotient  
threefold $E\times S/\iota\times\sigma$. Then $X$ is a smooth
Calabi--Yau threefold. We first show that $X$ has a model 
defined over $\QQ$. Then we will establish
the automorphy of the Galois representations associated to $X$,
in support of the Langlands reciprocity conjecture. 
We show that $X$ is of CM type if and only if $E$ also has complex 
multiplication. 

This generalizes the work by Livn\'e and Yui \cite{LY05} on the
modularity of the non-rigid %rank $4$ motive associated to the
Calabi--Yau threefold over $\QQ$ obtained from the quotient 
$E\times S/\iota\times\sigma$, where $S$ is a singular $K3$ surface 
with involution $\sigma$ (and hence of CM type).

We also construct mirror partners $X^{\vee}$ (if they exist)
of our Calabi--Yau threefolds using the Borcea--Voisin
construction. In fact, $57$  of the $95$ families of
$K3$ surfaces $S$ of Reid and Yonemura have mirror partners
$S^{\vee}$ within the list.  We show that all these $57$
families have subfamilies with involution $\sigma$ and a 
CM point rational over $\QQ$.  Then the quotients of
the products $E\times S^{\vee}/\iota\times\sigma^{\vee}$ give rise to
mirror partners of $E\times S/\iota\times\sigma$.

 From the point of view of mirror symmetry computations, our 
results supply particularly convenient base points in both 
the moduli space and in the mirror moduli space: they are 
defined over $\QQ$, and their $\ell$-adic \'etale 
cohomological Galois representations 
are attached to some automorphic forms whose $L$-series 
are known.

\section{$K3$ surfaces}\label{sect2}

\subsection{$K3$ surfaces with involution}\label{sect2.1}

Let $S$ be a $K3$ surface over $\CC$. Then $H^2(S,\ZZ)$ is 
torsion-free and the intersection pairing gives it the structure 
of a lattice, even and unimodular, of rank $22$ and signature 
$(3,19)$. By the classification theorem of such 
lattices, up to isometry,
$$H^2(S,\ZZ)\simeq U^3\oplus (-E_8)^2$$
where $U$ is the usual hyperbolic lattice of rank $2$ and 
$E_8$ is the unique even unimodular lattice of rank $8$.

Let $\pic (S)$ be the Picard lattice of $S$. It is torsion free 
and finitely generated, and together with the intersection 
pairing it can be identified as the sublattice 
$\pic (S)=H^2(S,\ZZ)\cap H^{1,1}(S)$ of $H^2(S,\ZZ)$.
We define the transcendental lattice of $S$, denoted by $T(S)$,
to be the orthogonal complement of $\pic (S)$ in $H^2(S,\ZZ)$,
i.e., $T(S):=\pic (S)^{\perp}$ in $H^2(S,\ZZ)$, with respect
to the intersection pairing.
\smallskip

Consider now a pair $(S, \sigma)$, where $S$ is a $K3$ surface 
and $\sigma$ is an involution of $S$ acting by $-1$ on 
$H^{2,0}(S)$.  Let $\pic (S)^{\sigma}$ denote the sublattice of 
$\pic (S)$ fixed by $\sigma$. %(In our case, $\pic (S)=
%\pic (S)^{\sigma=1}\oplus \pic (S)^{\sigma =-1}$.) 
%$\pic(S)^{\sigma}=\pic(S)$.)  
Let $(\pic (S)^{\sigma})^* :=\rm{Hom}(\pic (S)^{\sigma},\ZZ)$ be 
the dual lattice of $\pic (S)^{\sigma}$.
Let $T(S)_0=(\pic (S)^{\sigma})^{\perp}$ be the orthogonal 
complement of $\pic (S)^{\sigma}$ in $H^2(S,\ZZ)$, and let
$T(S)_0^*$ be the dual lattice of $T(S)_0$.
 From the assumption that $\sigma$ acts as $-1$ on the 
holomorphic $2$-forms of $S$, one can show that it acts by 
$-1$ on $T(S)_0$ (and by $1$ on $\pic (S)^{\sigma}$).
\smallskip

Consider the quotient groups
$(\pic (S)^{\sigma})^*/\pic (S)^{\sigma}$ and $T(S)_0^*/T(S)_0$.
Since $H^2(S,\ZZ)$ is unimodular, % and $\pic(S)$ is torsion-free,
the two quotient abelian groups are canonically isomorphic:
$$(\pic (S)^{\sigma})^*/\pic (S)^{\sigma}\simeq T(S)_0^*/T(S)_0.$$
On $(\pic (S)^{\sigma})^*/\pic (S)^{\sigma}$, $\sigma$ acts by 
$1$, while on $T(S)_0^*/T(S)_0$ it acts by $-1$. Then the only finite
abelian groups where $+1$ is $-1$ are the
$(\ZZ/2\ZZ)^a$ for some $a$. This shows that
$$(\pic (S)^{\sigma})^*/\pic (S)^{\sigma}\simeq (\ZZ/2\ZZ)^a
\quad\mbox{for some non-negative integer $a$}.$$
%\vspace{.5\baselineskip} 

Nikulin \cite{Ni80, Ni86} has classified such pairs $(S, \sigma)$.

\begin{thm} [Nikulin] {\sl The pair $(S, \sigma)$ of $K3$
surfaces with non-symplectic involution $\sigma$ is determined,
up to deformation, by a triplet $(r, a, \delta)$, where
$r=\rank \, \pic (S)^{\sigma}$,
$(\pic (S)^{\sigma})^*/\pic (S)^{\sigma}\simeq (\ZZ/2\ZZ)^a,$ and
$\delta=0$ if $(x^*)^2\in\ZZ$ for any
$x^*\in (\pic (S)^{\sigma})^*$, and $1$ otherwise.

There are in total $75$ triplets $(r,a,\delta)$, as shown in {\em Figure 
\ref{pyramid}}. 
 
The moduli space of $(S,\sigma)$ with given triplet $(r,a,\delta)$
is a bounded symmetric domain of type IV having dimension $20-r$.} 
\end{thm}

\begin{figure} 
\centering 
\includegraphics[width=12cm]{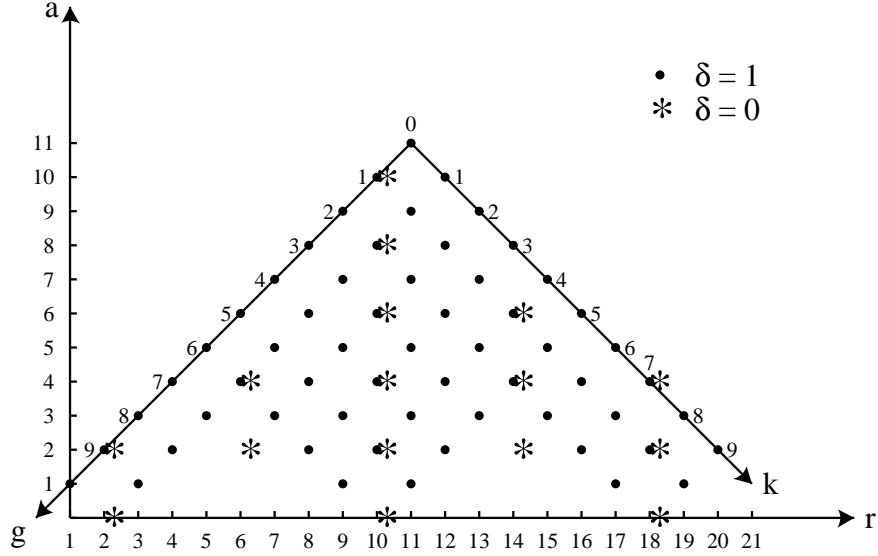}  %
\caption{Nikulin's pyramid} 
\label{pyramid} 
\end{figure} 

For a given pair $(S, \sigma)$ of a $K3$ surface $S$ with 
involution $\sigma$, we now consider the geometric structure 
of the fixed part $S^{\sigma}$ of $S$ under $\sigma$ (i.e., 
the part where $\sigma$ acts as identity). We follow Voisin 
\cite{V93} for this exposition.
\smallskip

\begin{prop} \label{r-and-a} {\sl
There are three types for $S^{\sigma}$:
\smallskip

(I) For $(r,a,\delta)\neq (10,10,0),\, (10,8,0)$,
$S^{\sigma}$ is a disjoint union of a smooth curve $C_g$ of 
genus $g$ and $k$ rational curves $L_i$:
$$S^{\sigma}=C_g\cup L_1\cup\cdots \cup L_k.$$
%\smallskip

(II) For $(r,a,\delta)=(10,10,0)$,  $S^{\sigma}=\emptyset$.
\smallskip

(III) For $(r,a,\delta)=(10,8,0)$, $S^{\sigma}$ is a disjoint 
union of two elliptic curves $C_1$ and $\bar C_1$:
$$S^{\sigma}=C_1\cup \bar{C}_1.$$

Furthermore, in the case (I), the genus $g$ and the number $k$ of
rational curves can be determined in terms of the triplet 
$(r, a, \delta)$ as follows:

$$g=\frac{1}{2}(22-r-a),$$
and
$$k=\frac{1}{2}(r-a).$$

Equivalently, $(r,a)$ and $(g,k)$ are related by the
identities:
$$r=11-g+k,\quad a=11-g-k.$$}
\end{prop}

\begin{rem}
Since $\sigma$ is a non-symplectic involution, the quotient 
$S/\sigma$ is either a rational surface or Enriques surface. 
It is an Enriques surface if and only if 
$S^{\sigma}=\emptyset$, i.e. $(r,a,\delta)=(10,10,0)$.
Note that if $\sigma$ is a symplectic involution, 
then $\sigma$ has eight fixed points and the minimal resolution 
of $S/\sigma$ is again a $K3$ surface.
\end{rem}

\subsection{Realization of $K3$ surfaces as hypersurfaces over $\QQ$}
\label{sect2.2}

We are interested in finding defining equations over $\QQ$
for pairs $(S,\sigma)$ of $K3$ surfaces $S$ with non-symplectic
involution $\sigma$.  For this, we appeal to the famous $95$
families of $K3$ surfaces of M. Reid \cite{R79} (see also
Iano--Fletcher \cite{IF20}) and of Yonemura \cite{Yo89}. All these
$95$ families of $K3$ surfaces are realized in weighted
projective $3$-spaces $\PP^3(w_0,w_1,w_2,w_3)$. Reid determined
$95$ possible weights $(w_0,w_1,w_2,w_3)$, and singularities
as they are all determined by the weights. Then Yonemura 
described concrete families of 
hypersurfaces defining them, using toric constructions.

We first recall a result of Borcea \cite{B94}. 
% and of Yonemura \cite{Yo89}. 
Here we say that $Q=(w_0,w_1,w_2,w_3)$ is
{\it normalized} if $\gcd (w_i,w_j,w_k)=1$ for every distinct
$i, j, k$. Also, we assume that $w_i$'s are ordered in such a
way that $w_0\geq w_1\geq w_2\geq w_3$.

\begin{prop} [Borcea] \label{prop-BY}
{\sl Assume that $Q=(w_0,w_1,w_2,w_3)$ is normalized and 
$w_0=w_1+w_2+w_3$. Then there are in total $48$ weights
$(w_0,w_1,w_2,w_3)$ giving rise to pairs $(S,\sigma)$ of $K3$
surfaces $S$ with involution $\sigma$ acting by $-1$ on $H^{2,0}(S)$.
More precisely, if $w_0$ is odd, there are $29$ weights, and if
$w_0$ is even, there are $19$ weights. 

$S$ may be realized as the minimal resolution
of a hypersurface $S_0$ of degree $2w_0$ in $\PP^3(w_0,w_1,w_2,w_3)$
of the form
$$x_0^2=f(x_1,x_2,x_3)$$
where $\mbox{deg}(x_i)=w_i$ for $0\leq i\leq 3$. A non-symplectic involution
$\sigma$ on $S_0$ is defined by $\sigma(x_0)=-x_0$, and $f$ is a
homogeneous polynomial in the variables $x_1,x_2,x_3$ of degree
$2w_0$. }
\end{prop}

By abuse of notation, we often write $\sigma$ for the involution on $S_0$ 
as well as that induced on $S$ by desingularization. 

Our first result is to extend the list of Borcea by
adding more weights that yield $K3$ surfaces $(S,\sigma)$ with
involution $\sigma$.

\begin{thm} \label{thm-newS}
{\sl There are in total $92=48+44$ normalized weights $(w_0,w_1,w_2,w_3)$
giving rise to pairs $(S,\sigma)$ of $K3$ surfaces $S$ with
non-symplectic involution $\sigma$ defined by $\sigma (x_i)=-x_i$
for some single variable $x_i$. In other words, we have $44$
new weights (i.e., not in the list of Borcea) yielding $K3$
surfaces with involution $\sigma$.

We divide the $92$ cases into two groups:

\noindent {\em (i)}
The $48$ weights of Borcea and defining equations for quasi-smooth
$K3$ surfaces $S_0$ are tabulated in {\em Tables \ref{table1},
\ref{table2}} and {\em \ref{table3}} in {\em Section \ref{tables}}.

\noindent {\em (ii)}
The additional $44$ weights and defining equations $F(x_0,x_1,x_2,x_3)
=0$ for quasi-smooth $K3$ surfaces $S_0$ are tabulated in {\em Tables
\ref{table-y2}, \ref{table-y4}, \ref{table-y6}} and {\em \ref{table-y5}}
in {\em Section \ref{tables}}.
}
\end{thm}

\begin{proof}
The proof of Theorem \ref{thm-newS} is done by case by case
analysis. Yonemura \cite{Yo89} determined hypersurface 
equations defining Reid's $95$ families of $K3$ surfaces 
using toric geometry. We use his list of
equations to find $K3$ surfaces with non-symplectic involutions.

If the defining equation contains the term $x_0^2$ or $x_0^2x_i$,
then we can define involution $\sigma$ by $\sigma (x_0)=-x_0$
just as in Borcea's 48 cases. If the defining equation contains
the term $x_0^2x_i +x_0x_j^m$ (say), then we remove $x_0x_j^m$ to
define the involution $\sigma (x_0)=-x_0$ (see Tables \ref{table-y2}
and \ref{table-y5}).

For the equations in Table \ref{table-y4}, we change $x_0^3$ to
$x_0^2x_1$ to define an involution by $\sigma (x_0)=-x_0$.

For the equations in Table \ref{table-y6}, we choose
variables other than $x_0$ (and remove several terms if necessary)
to define an involution.

Note that in each of the $92$ cases, the quotient $S/\sigma$ is a
rational or Enriques surface. Hence $\sigma$ is a non-symplectic
involution.
\end{proof}

\begin{rem}
Among the $95$ $K3$ weights of Reid, there are three cases $\#15, 
\#53, \#54$ where we find no obvious involution; that is, there
is no involution $\sigma$ on $S$ acting as $\sigma (x_i)=-x_i$
for some variable $x_i$. These cases are tabulated in Table
\ref{table-y1} in Section \ref{tables}.
\end{rem}

For our arithmetic purposes, it is useful that we can compute 
the zeta-functions of $S$ explicitly. One of such classes of 
varieties are those defined by equations of {\it Delsarte
type} (i.e. equations consisting of the same number of monomials as
the variables) named after Delsarte (see \cite{LSY10}, Section 4).
Hypersurfaces of Delsarte type are finite quotients of Fermat varieties. 
Our next task is to find the subset of the $92$ cases of Theorem 
\ref{thm-newS} which can be defined by equations of Delsarte type, namely 

\begin{enumerate}
\item for each $S$ of \cite{Yo89}, find an equation $h(x_0,x_1,x_2,x_3)$ 
consisting exactly of four monomials, and 
\item make sure that the hypersurface obtained in (1) is quasi-smooth.
\end{enumerate}

Conditions (1) and (2) give a restriction on the form of $S$, but 
many of its geometric properties are unchanged. For instance, the 
types of singularities on $S$ remain the same as the original hypersurfaces
$h(x_0,x_1,x_2,x_3)=0$ of \cite{Yo89}.

\begin{thm} \label{thm-DelsarteS}
{\sl There are $86$ weights $(w_0,w_1,w_2,w_3)$ for which
there exist a quasi-smooth $K3$ surface $S_0$ in
$\PP^3(w_0,w_1,w_2,w_3)$ defined by a Delsarte equation
with an involution. Moreover, this involution defines
a non-symplectic involution $\sigma$ on the
minimal resolution $S$ of $S_0$.
%that give rise to
%pairs $(S,\sigma)$ of $K3$ surfaces $S$ of Delsarte type with
%non-symplectic involution $\sigma$. $S$ is the minimal resolution
%of a quasi-smooth $K3$ surface $S_0$ defined by a Delsarte equation
%in $\PP^3(w_0,w_1,w_2,w_3)$.

{\em (a)} If $(S,\sigma)$ is one of the $48$ pairs determined in
{\em Proposition \ref{prop-BY}} other than $\#90, \#91, \#93$, then
$S_0$ can be defined by an equation over $\QQ$ of four monomials
$$x_0^2=f(x_1,x_2,x_3)\quad \subset \PP^3(w_0,w_1,w_2,w_3).$$
The equation is obtained by removing several terms from the equation
of Yonemura \cite{Yo89}, where $f$ is a homogeneous polynomial over
$\QQ$ of degree $w_0+w_1+w_2+w_3$ (cf. {\em Tables \ref{table1},
\ref{table2}, \ref{table3}}).

{\em (b)} Let $(S,\sigma)$ be one of the additional $38$ pairs
determined in {\em Theorem \ref{thm-newS} (ii)} other than $\#85$, 
$\#90$, $\#91$, $\#93$, $\#94$ and $\#95$. Then $S_0$ can be defined 
by an equation $F(x_0,x_1,x_2,x_3)=0$ over $\QQ$ consisting of four 
monomials of degree $w_0+w_1+w_2+w_3$. In most cases, $F(x_0,x_1,x_2,x_3)$ 
can be chosen as
$$F(x_0,x_1,x_2,x_3)=x_0^2x_i+f(x_1,x_2,x_3).$$
The weights and equations are listed in {\em Tables
\ref{table-y2}, \ref{table-y4}} and {\em \ref{table-y6}}
in {\em Section \ref{tables}}.
 }
\end{thm}

\begin{proof}
This can be proved by case by case checking of the list of
equations of Yonemura \cite{Yo89}. In both (a) and (b), we
transform $S$ into a Delsarte type by removing several terms of its
original defining equation. In doing so, we make sure that
the condition (2) above is satisfied so that the new surface is also
quasi-smooth.

Condition (2) is met if for each variable $x_i\ (0\leq i\leq 3)$,
the set of monomials containing $x_i$ takes one of the following forms:
$$x_i^n,\ x_i^nx_j,\ x_i^n+x_ix_j^m,\ x_i^nx_k+x_ix_j^m$$
for some $j$ and $k$ $(j\neq k)$ different from $i$.

This choice for the defining hypersurface preserves the
configuration of singularities on $S$ (i.e., types and the number
of singularities) as the original hypersurfaces. The point
is that for each of the $86$ families, we can find a 
defining equation which consists of four monomials
containing $x_i^n$ or $x_i^nx_j $ ($i\neq j$) with
nonzero coefficients.
\end{proof}
\medskip

\begin{rem}
Our list of defining equations for $S$ does not cover all possible
equations of Delsarte type. For instance, in case $\#19$ of
weight $(3,2,2,1)$ in Table \ref{table-y2}, we may also choose 
an equation $x_0^2x_1+x_1^3x_2+x_2^4+x_3^8=0$.
\end{rem}

\begin{rem}
For the $95$ families of quasi-smooth weighted $K3$ hypersurfaces,
Yonemura \cite{Yo89} described the number of parameters (i.e., the
number of monomials) for their defining equations. The minimum
number was four, but often equations contain more than four
monomials.  Our result shows that except for the six cases
$\#85, \#94, \#95$ of Table \ref{table-y5} and $\#90, \#91, \#93$
of Table \ref{table1}, the minimum number
of parameters is attained.
\end{rem}

\begin{ex}
$\bullet$ Consider $\#42$ in Yonemura =$\#3$ in  Borcea.
The weight is $(5,3,1,1)$ and a hypersurface is given by
$$x_0^2=f(x_1,x_2,x_3)=x_1^3x_3+x_1^3x_3+x_2^{10}+x_3^{10}$$
of degree $10$.  We can remove the second monomial $x_1^3x_3$.
The singularity is of type $A_2$.
\smallskip

$\bullet$ Consider $\#78$ in Yonemura = $\#10$ in Borcea. The weight
is $(11,6,4,1)$ and a hypersurface is given by
$$x_0^2=f(x_1,x_2,x_3)=
x_1^3x_2+x_1^3x_3^4+x_1x_2^4+x_2^5x_3^2+x_3^{22}$$
of degree $22$. We can remove the second and the fourth
monomials $x_1^3x_3^4$ and $x_2^6x_3^2$.
The singularity is of type $A_1+A_3+A_5$.
\smallskip

$\bullet$ Consider $\#19$ in Yonemura. The weight is $(3,2,2,1)$
and a hypersurface is given by
$$F(x_0,x_1,x_2,x_3)
=x_0^2x_1+x_0^2x_2+x_0^2x_3^2+x_1^4+x_2^4+x_3^8$$
of degree $8$. We can remove the first or the second monomials
$x_0^2x_1$ or $x_0^2x_2$, the third monomial $x_0^2x_3^2$.
The involution is given by $x_0\to -x_0$.
The singularity is of type $4A_1+A_2$.
\end{ex}

\begin{rem} \label{rmk-table8}
The cases $\#85$, $\#90$, $\#91$, $\#93$, $\#94$ and $\#95$ of
Yonemura cannot be realized as quasi-smooth hypersurfaces in four
monomials with involution $\sigma$.

For instance, consider the case $\#85$ of weight $(5,4,3,2)$
and degree $14$. All the possible monomials of degree $14$ are
$$x_0^2x_1,\, x_0^2x_3^2,\, x_0x_1x_2x_3,\, x_1x_2^3,\, x_0x_2x_3^3,\,
x_1^3x_3,\, x_1^2x_2^2,\,$$
$$x_1^2x_3^3,\, x_1x_2^2x_3^2,\, x_1x_3^5,\, 
x_2^4x_3,\, x_2^2x_3^4,\, x_3^7.$$
To make the polynomial quasi-smooth and defined by four monomials,
we remove the monomials 
$$x_0^2x_3^2,\, x_0x_1x_2x_3,\, x_0x_2x_3^3,\, x_1^2x_2^2,\, 
x_1^2x_3^3,\, x_1x_2^2x_3^2,\, x_2^2x_3^4.$$ 
Then we obtain monomials 
$$x_0^2x_1,\ x_1x_2^3,\ x_1^3x_3,\ x_1x_3^5,\ x_2^4x_3,\ x_3^7.$$
There are no four monomials from this set such that their sum 
defines a quasi-smooth polynomial.

Note that if we allow more than four monomials, we can
define a non-symplectic involution on this surface. For
example, the surface defined by
$$x_0^2x_1+x_0^2x_3^2+x_1^3x_3+x_1^2x_3^2+x_2x_3^5
+x_2^4x_3+x_3^7=0$$
is quasi-smooth and endowed with an involution $\sigma (x_0) =-x_0$.
\end{rem}

\subsection{$K3$ surfaces of CM type}\label{sect2.3}

Recall a definition of a CM type variety.

\begin{defn}
{\rm A cohomology group of a variety is said to be {\it of CM type} if
its Hodge group is commutative, and a variety is said to
be {\it of CM type} if all its cohomology groups are of CM 
type (\cite{Za83}).}
\end{defn}

\begin{thm}\label{thm2.5}
{\sl Let $(S,\sigma)$ be one of the $86$ pairs of
$K3$ surfaces $S$ with involution $\sigma$. Then
$(S,\sigma)$ is defined over $\QQ$ and it is
of CM type.}
\end{thm}

\begin{proof}
%By Weil, one knows that Fermat (hyper)surfaces
%are of CM type. 
We use Shioda's result (see, for instance, \cite{SK79}).
Let $X_n^m: x_0^m+x_1^m+\cdots+x_{n+1}^m=0\subset{\mathbb{P}}^{n+1}$ 
be the Fermat variety of degree $m$ and dimension $n$. Let ${\mathbf{\mu}}_m$
denote the group of $m$-th roots of unity (in $\CC$). 
Then the eigenspaces of the action of $({\mathbf{\mu}}_m)^{n+2}$ 
on the middle cohomology group of $X_n^m$ are one-dimensional, 
and this 
action commutes with the Hodge group. Hence the Hodge group is
commutative, and so the Fermat 
(hyper)surfaces are of CM type.  Since a pair $(S,\sigma)$ is a finite
quotient of a Fermat surface, it is of CM type.
\end{proof}

\subsection{Computations of Nikulin's invariants for $K3$ surfaces 
of Borcea type} \label{sect2.4}

Recall that a $K3$ surface $S$ with involution $\sigma$ is determined 
up to deformation by a triplet $(r,a,\delta )$, where $r$ is the rank of 
$\pic (S)^{\sigma}$. In this section, we compute $r$ and $a$
for $K3$ surfaces of Borcea type. By Proposition \ref{r-and-a}, $r$ 
and $a$ can be computed through the fixed locus 
$S^{\sigma}$:
$$r=11-g+k,\quad a=11-g-k.$$

We note that the direct computation of $r$ 
often requires a basis for $\pic (S)$ or at least for $\pic (S)\otimes \QQ$. 
Since $\pic (S)$ is usually difficult to determine, the fixed locus 
$S^{\sigma}$ is often easier to handle than the Picard group. 

In what follows, first we explain a general algorithm of computing $g$ and $k$
(and hence $r$ and $a$). To describe the algorithm in detail, we choose $K3$ 
surfaces of Borcea type, namely those defined by $x_0^2=f(x_1,x_2,x_3)$. 
After that, we explain how to compute $r$ directly by looking at the 
$\sigma$-action on $\pic (S)$ (see Theorem \ref{r-and-rq}). 
It has a merit that we can obtain a closed formula for $r$. 

\medskip

\subsubsection{\bf Algorithm for the computation of $g$ and $k$ }
\medskip

We explain how to compute $g$ and $k$ (and then $r$) for our $K3$ surfaces $S$. 
In this section and the next, we choose $S_0$ to be a surface in 
$\PP^3(w_0,w_1,w_2,w_3)$ defined by the equation 
$$x_0^2=f(x_1,x_2,x_3) \quad \mbox{ or }\quad x_0^2x_i=f(x_1,x_2,x_3)$$  
where $\sigma$ acts on it by $\sigma (x_0)=-x_0$. We assume that $Q=
(w_0,w_1,w_2,w_3)$ is normalized. The algorithm described in this section 
also works for more general $K3$ surfaces in $\PP^3(Q)$ with non-symplectic 
involution. 

Since $S_0$ is singular, $S$ is chosen to be the minimal resolution of $S_0$ 
as in
$$\begin{array}{ccc}
\PP^3(w_0,w_1,w_2,w_3) & \longleftarrow & \tilde{\PP^3}(w_0,w_1,w_2,w_3)\\
\cup & & \cup \\
S_0 & \longleftarrow & S
\end{array}$$
where $\tilde{\PP^3}(w_0,w_1,w_2,w_3)$ is a partial resolution of
$\PP^3(w_0,w_1,w_2,w_3)$ so that $S$ is non-singular. The curve $C_g$
in Proposition \ref{r-and-a} is the strict transform of the curve 
defined by $x_0=0$ on $S_0$ which
is isomorphic to the curve $f(x_1,x_2,x_3)=0$ in $\PP^2(w_1,w_2,w_3)$.
The rational curves $L_i$ in Proposition \ref{r-and-a} are among those 
defined by letting either 
$x_1$, $x_2$ or $x_3$ be zero, or from the exceptional divisors
arising in the desingularization. The procedure is as follows.
\medskip

(i)\ The curve $\{ x_0=0\}$ on $S_0$ is fixed by $\sigma$. Since
$f(x_1,x_2,x_3)=0$ is quasi-smooth (and hence smooth) in $\PP^2(w_1,w_2,w_3)$,
its strict transform $C_g$ is also fixed by $\sigma$. The genus $g$ can
be calculated from $d:=\deg f$ and weight $(w_1,w_2,w_3)$ once it is
normalized (see examples below). For instance, one can use the formula,
which can be found in, e.g., Iano-Fletcher \cite{IF20}:
$$g=\frac{1}{2}\biggl( \frac{d^2}{w_1w_2w_3} -d\sum_{i>j}
\frac{\gcd (w_i,w_j)}{w_iw_j}+\sum_{i=1}^3 \frac{\gcd (d,w_i)}{w_i} -1\biggr).$$
\medskip

(ii)\ The locus $\{ x_1=0\}$, $\{ x_2=0\}$ or $\{ x_3=0\}$ may be fixed
by $\sigma$ depending on $(w_0,w_1,w_2,w_3)$. For instance, consider the
case where $w_0$ is odd. If $w_2$ and $w_3$ are even (and $w_3$ is
necessarily odd), then the locus $\{ x_3=0\}$ is fixed by $\sigma$. In
this case, the strict transform of the locus $\{ x_3=0\}$ is also 
point-wise fixed by $\sigma$. It can be shown that this
locus is a rational curve on $S$ and contributes to a curve 
$L_i$ of Proposition \ref{r-and-a}.
\medskip

(iii) The rest of the curves $L_i$ (for $S^{\sigma}$) are 
exceptional divisors 
in the minimal resolution $S_0\leftarrow S$. To find the divisors where 
$\sigma$ acts as identity, we look carefully at the $\sigma$-action 
around the singularities of $S_0$ fixed by $\sigma$. Every singularity 
$P=(x_0:x_1:x_2:x_3)$ on $S_0$ 
has at least two coordinates zero. Hence $P$ is fixed by $\sigma$ if 
and only if

$\bullet$ it has three zero coordinates, or

$\bullet$ it has exactly two zero coordinates with $x_0=0$, or

$\bullet$ it has exactly two zero coordinates and two non-zero
coordinates $x_0$, $x_i$, and if $d=\gcd (w_0,w_i)\geq 2$, then 
$w_0/d$ is odd and $w_i/d$ is even.  

Combining information obtained in (ii) and (iii), we can calculate the 
number $k$ and hence the invariants $r$ and $a$ by the formula given in 
Proposition \ref{r-and-a}. 

\medskip 

\subsubsection{\bf Computation of $g$ and $k$ for surfaces 
$x_0^2=f(x_1,x_2,x_3)$} 
\medskip 

We consider surfaces $x_0^2=f(x_1,x_2,x_3)$ in detail. Since the 
minimal resolution $S$ is a $K3$ surface, $P$ is a cyclic quotient singularity 
of type $A_{n+1,n}$ (or simply $A_n$) for some positive integer $n$. 
%We obtain $n$ exceptional divisors by resolving $P$. Each divisor on $S$ is 
%isomorphic to $\PP^1$, and $\sigma$ induces on $S$ either the identity or 
%a non-trivial involution on $\PP^1$, which depends on the singularity and 
%the divisor. 
We obtain $n$ exceptional divisors (i.e., irreducible components in
the exceptional locus) by resolving $P$. Each exceptional divisor is
isomorphic to $\PP^1$ and $\sigma$ acts on it either as identity or
as a non-trivial involution, which depends on the singularity at $P$.
In the first case, we say that the divisor is {\it ramified}. 
%As $P$ is fixed by $\sigma$, $\sigma$ sends every exceptional 
%divisor to itself. However, the $\sigma$ action 
%on each divisor is either $+1$ or $-1$, which depends on the singularity. 
The detail is explained in the following lemmas.  

\begin{lem} \label{sing-case1} 
{\sl Let $S_0:\ x_0^2=f(x_1,x_2,x_3)$ be one of the $48$ $K3$ surfaces 
defined in {\rm Proposition \ref{prop-BY}}. If $w_1\geq 2$ and $P=(0,1,0,0)$ 
is on $S_0$, then $P$ is a cyclic quotient singularity of type 
$A_{w_1,w_1-1}\, (=A_{w_1-1})$. Among the exceptional divisors 
arising from $P$, ramified divisors appear alternately and 
there are $\displaystyle 
\biggl[ \frac{w_1-1}{2} \biggr]$ of them, where $[x]$ denotes the 
integer part of $x$. The same assertion holds 
for singularities $(0,0,1,0)$ and $(0,0,0,1)$. 
} 
\end{lem} 

\begin{proof} 
%When $P=(0,1,0,0)$ is a singularity, depending on $f(x_1,x_2,x_3)$, 
%$(x_0,x_2)$ or $(x_0,x_3)$ gives a pair of local coordinates around $P$. 
%Assume that $(x_0,x_2)$ is a coordinate system. Since $P$ is of type 
%$A_{w_1,w_1-1}$, 
When $P=(0,1,0,0)$ is a singularity, $(x_0,x_2)$ or $(x_0,x_3)$ gives
a pair of local coordinates above $P$, depending on the polynomial 
$f(x_1,x_2,x_3)$. Assume that $(x_0,x_2)$ is a coordinate system 
above $P$. Since $P$ is a cyclic quotient singularity of type 
$A_{w_1,w_1-1}$, the $\mu_{w_1}$-action above $P$ can be written as 
$$(x_0,x_2)\longmapsto (\zeta^{w_1-1} x_0, \zeta x_2)$$ 
with $\zeta \in \mu_{w_1}$. (In other words, $P$ is the singularity of 
the quotient $\bA^2 /\mu_{w_1}$ at the origin.) By the quotient map 
$\pi_0: S_0\lra S_0/\sigma$, where $S_0/\sigma$ is in $\PP^3(2w_0,w_1,w_2,w_3)$, 
$P$ is mapped to $\pi_0 (P) \in S_0/\sigma$. It is a singularity locally  
described by the group action 
$$(y_0,x_2)\longmapsto (\zeta^{2(w_1-1)} y_0, \zeta x_2)$$
with $y_0=x_0^2$. Hence $\pi_0 (P)$ is a singularity of type $A_{w_1,2(w_1-1)}$ 
(or precisely, $A_{w_1,w_1-2}$). 

In order to see how $\sigma$ acts on the exceptional divisors, let $E_1+E_2+
E_3+\cdots +E_{w_1-1}$ be the exceptional divisors on $S$ arising 
from $P$. Here we set $E_1$ to be the divisor intersecting with 
(the strict transforms) of the curves passing through $P$. Consider 
the continued fractional expansions 
$$\frac{w_1}{w_1-1}=2-\cfrac{1}{2-\cfrac{1}{2-\cfrac{1}{2-\cfrac{1}{\cdots}}}} 
\quad \mbox{and} \quad 
\frac{w_1}{2(w_1-1)}=
1-\cfrac{1}{4-\cfrac{1}{1-\cfrac{1}{4-\cfrac{1}{\cdots}}}}\,\,\,.$$ 
If $\pi: S\lra S/\sigma$ denotes the quotient map, the continued fractions 
show  that 
$$\pi (E_i)^2 =\begin{cases} 
-1 & \mbox{ if $i$ is odd } \\ 
-4 & \mbox{ if $i$ is even.} 
\end{cases}$$ 
By the projection formula, we see that $\sigma$ acts as identity 
(resp. $-1$) on $E_i$ when $\pi (E_i)^2=-4$ (resp. when $\pi (E_i)^2=-1$). 
Therefore the ramified exceptional divisors are those $E_i$'s with even $i$ 
and there are in total $[ \frac{w_1-1}{2} ]$ of them. 
\end{proof} 

Next, we discuss singularities with exactly two zero coordinates, one of 
which is $x_0=0$. 

\begin{lem} \label{sing-case2} 
{\sl Let $S_0:\ x_0^2=f(x_1,x_2,x_3)$ be one of the $48$ $K3$ surfaces 
defined in {\rm Proposition \ref{prop-BY}}. If $\gcd (w_1,w_2)\geq 2$, 
then $P=(0,x_1,x_2,0)$ with $x_1x_2\neq 0$ is a singularity of type $A_{2,1}$ 
and fixed by $\sigma$. There arises one exceptional divisor from $P$ and 
it is not ramified under the quotient $S\lra S/\sigma$. The same assertion 
holds for singularities of the form $(0,x_1,0,x_3)$ and $(0,0,x_2,x_3)$.  } 
\end{lem} 

\begin{proof} 
Write $d:=\gcd (w_1,w_2)$. When $P=(0,x_1,x_2,0)$ is a singularity, $(x_0,x_3)$ 
gives a local coordinate system above $P$; that is, $P$ is a singularity of 
the quotient of an affine plane by the group action $\zeta \in \mu_d$ 
defined by 
$$(x_0,x_3)\longmapsto (\zeta^{w_0} x_0, \zeta^{w_3} x_3)$$ 
Since $f$ is quasi-smooth, $f$ contains a term without 
$x_3$. Comparing the degree of monomials in $x_0^2=f(x_1,x_2,x_3)$, one finds 
that $2w_0$ is divisible by $d$. Since the weight is normalized, we see that 
$\gcd (d,w_0)=\gcd (d,w_3)=1$ and hence $d\mid 2$. If $d\geq 2$, then $d$ must 
be $2$. 

Since $d=2$, the relation $w_0=w_1+w_2+w_3$ implies $w_0\equiv w_3\pmod{2}$  
and the $\mu_2$ action above can be written as 
$$(x_0,x_3)\longmapsto (-x_0, -x_3).$$
This means that $P$ is a singularity of type $A_{2,1}$. If 
$\pi_0: S_0\lra S_0/\sigma$ denotes the quotient map, then $\pi_0(P)$ 
is the singularity associated with the group action 
$$(y_0,x_3)\longmapsto (y_0, -x_3)$$
where $y_0=x_0^2$. This shows in fact that $\pi_0 (P)\in S_0/\sigma$ is not 
a singularity. Hence if $E$ is the exceptional divisor arising from $P$ and 
$\pi : S\lra S/\sigma$ is the quotient map, $\pi (E)$ should be a $(-1)$-curve. 
Therefore $E$ is not ramified in $\pi$. 
\end{proof} 

Lastly, we look at the singularity with two zero coordinates and $x_0\neq 0$. 

\begin{lem} \label{sing-case3} 
{\sl Let $S_0:\ x_0^2=f(x_1,x_2,x_3)$ be one of the $48$ $K3$ surfaces 
defined in {\rm Proposition \ref{prop-BY}}. If $d:=\gcd (w_0,w_1)\geq 2$, 
then $P=(x_0,x_1,0,0)$ with $x_0x_1\neq 0$ is a singularity of $S_0$. It 
is fixed by $\sigma$ if and only if $w_0/d$ is odd and $w_1/d$ is even. 
Let $E_1+\cdots +E_{d-1}$ denote the exceptional divisors arising from $P$.  
Then $\sigma$ acts on $E_i$ as identity (resp. by $-1$) if $i$ is odd (resp. 
even). There are in total $\displaystyle \biggl[ \frac{d}{2} \biggr]$ 
divisors with odd $i$. The same assertion holds for singularities 
$(x_0,0,x_2,0)$ and $(x_0,0,0,x_3)$.  } 
\end{lem} 

\begin{proof} 
Since $S_0$ is quasi-smooth, one knows that $P=(x_0,x_1,0,0)$ with $x_0x_1
\neq 0$ is a singularity if and only if $d=\gcd (w_0,w_1)\geq 2$. Let 
$w_0=du_0$ and $w_1=du_1$ with $\gcd (u_0,u_1)=1$. To see if $P$ is fixed 
by $\sigma$, there are three cases to consider: $(u_0,u_1)=(\mbox{even}, 
\mbox{odd}),\, (\mbox{odd}, \mbox{odd}),\, (\mbox{odd}, \mbox{even})$. Suppose 
that there is a $t$ such that $t^{w_0}=-1$ and $t^{w_1}=1$. If $u_0$ is even 
and $u_1$ is odd, then $t^{u_0u_1d}=(t^{u_0d})^{u_1}=(-1)^{u_1}=-1$ and 
$t^{u_0u_1d}=(t^{u_1d})^{u_0}=1$, which is absurd. By the same reason as above, 
$(u_0,u_1)=(\mbox{odd},\, \mbox{odd})$ cannot happen either. If $u_0$ is odd 
and $u_1$ is even, then let $\zeta$ be a primitive $2d$-th root of unity. 
We have ${\zeta}^{w_0}=({\zeta}^d)^{u_0}=(-1)^{u_0}=-1$ and ${\zeta}^{w_1}=
({\zeta}^d)^{u_1}=(-1)^{u_1}=1$. Hence there does exist 
a $t$ satisfying $t^{w_0}=-1$ 
and $t^{w_1}=1$, and $P$ is fixed by $\sigma$ in this case. 

Choose $(x_2,x_3)$ as a local coordinate system above $P$. $P$ is 
isomorphic to the 
singularity of the quotient by the group action $\mu_{d}$ defined by 
$$(x_2,x_3)\longmapsto (\zeta^{w_2} x_2, \zeta^{w_3} x_3)$$ 
where $\zeta$ runs through $\mu_d$. Since the weight is normalized, 
$\gcd (d,w_2)=\gcd (d,w_3)=1$ and the $\mu_d$ action above can be written as 
$$(x_2,x_3)\longmapsto (\zeta^{d-1} x_2, \zeta x_3).$$
$P$ is mapped to $(x_0,x_1,0,0)\in S_0/\sigma$ and the group action around 
this point is 
$$(x_2,x_3)\longmapsto (\zeta^{d-1} x_2, \zeta x_3)$$
as above. But, since $\gcd (2w_0,w_1)=\gcd (2du_0,du_1)=2d\gcd (u_0,u_1/2)=2d$, 
$\zeta$ now runs through $\mu_{2d}$. Hence this singularity on $S_0/\sigma$ 
is of type $A_{2d,d-1}$. 

Consider two continued fractions 
$$\frac{d}{d-1}=2-\cfrac{1}{2-\cfrac{1}{2-\cfrac{1}{2-\cfrac{1}{\cdots}}}} 
\quad \mbox{and} \quad 
\frac{2d}{d-1}=4-\cfrac{1}{1-\cfrac{1}{4-\cfrac{1}{1-\cfrac{1}{\cdots}}}}\,\,\,.$$ 
This shows that if $E_1+E_2+E_3+\cdots$ are exceptional divisor arising from 
$P$, then only the exceptional divisors $E_i$ with odd $i$ are fixed by 
$\sigma$. Therefore ramified exceptional divisors appear alternately and 
there are $[ \frac{d}{2} ]$ such divisors. 
\end{proof} 

Recall that $r$ is the rank of $\pic (S)^{\sigma}$. If we have good 
knowledge of $\pic (S)$ and the $\sigma$ action on it, then we can compute 
$r$ without knowing $S^{\sigma}$. The following theorem shows that we can 
in fact find a closed formula for $r$ by taking this approach. 

\begin{thm} \label{r-and-rq}
{\sl Let $(S,\sigma)$ be one of the $48$ $K3$ surfaces defined in
{\em Proposition \ref{prop-BY}} as the minimal resolution of a hypersurface
$$S_0:\ x_0^2=f(x_1,x_2,x_3)\quad \subset \PP^3(Q)$$
where $Q=(w_0,w_1,w_2,w_3)$ and $f$ is a homogeneous polynomial
of degree $w_0+w_1+w_2+w_3$. Recall $r=\rank \pic (S)^{\sigma}$. Let $r(Q)$
denote the number of exceptional divisors in the resolution $S\lra S_0$.
Assume that $\rank \pic (S_0)^{\sigma} =1$.

{\rm (a)} If $w_0$ is odd, then there exists at most one odd weight $w_i$
such that $\gcd (w_0,w_i)\geq 2$ and in such a case, $\gcd (w_0,w_i)=w_i$.
We have
$$r=\begin{cases}
r(Q)-w_i+2 & \mbox{ if $\gcd (w_0,w_i)\geq 2$ for some odd weight $w_i$
($1\leq i\leq 3$) } \\
r(Q)+1 & \mbox{ otherwise.}
\end{cases}$$

{\rm (b)} If $w_0$ is even, then let $d_i =\gcd (w_0,w_i)$. We have
$$r=r(Q)+1-\sum_{i=1}^3 (d_i-1) \biggl( \frac{2d_i}{w_i}-1 \biggr).$$
}
\end{thm}

\begin{proof}
Since $\rank \pic (S_0)^{\sigma} =1$, $\pic (S_0)\otimes \QQ$ is generated by 
the hyperplane section $\{ x_0=0\}$, which is fixed by $\sigma$. Its strict
transform on $S$ is also fixed by $\sigma$ and gives a one-dimensional
subspace of $\pic (S)^{\sigma}\otimes \QQ$. The rest of it is generated by
exceptional divisors.

Possible singularities for $S_0$ are either of the form $(0:x_1:x_2:x_3)$
with one or two zero coordinates or of the form $(x_0:x_1:x_2:x_3)$ with
$x_0\neq 0$ and exactly two zero coordinates. Since the points with $x_0=0$
are fixed by $\sigma$, every exceptional divisor $E$ arising from a singularity
$(0:x_1:x_2:x_3)$ satisfies $\sigma (E)=E$. ($\sigma$ acts on $E$ as $\pm 1$.)
Such exceptional divisors form part of a basis for $\pic (S)^{\sigma}$.

Consider a singularity $(x_0:x_1:x_2:x_3)$ with $x_0x_i\neq 0$, $\gcd (w_0,w_i)
\geq 2$ and other coordinates zero. It is fixed by $\sigma$ if and only if
$t^{w_0}=-1$ and $t^{w_i}=1$ for some $t$. As $x_0x_i\neq 0$ and other
coordinates are zero, $f(x_1,x_2,x_3)$ contains a monomial solely in $x_i$.
Hence $w_i\mid 2w_0$, where $2w_0 =\deg f$. Now we divide the proof according
as the parity of $w_0$.
\vspace{.5\baselineskip}

(a)\ Assume that $w_0$ is odd. Then the normality of weight $Q$ and the
equality $w_0=w_1+w_2+w_3$ imply that there is exactly one odd $w_i$ for
$1\leq i\leq 3$. For simplicity, assume that $w_1$ is odd and $w_2$ and
$w_3$ are even.
\smallskip

(i)\ If $\gcd (w_0,w_1)\geq 2$, then $(x_0:x_1:0:0)$ are singular points.
Since $1=(t^{w_1})^{w_0}=(t^{w_0})^{w_1}=-1$, none of the singularities 
is fixed by $\sigma$. But, if we consider a $\sigma$-conjugate pair, 
it is invariant under $\sigma$. There are two singularities of the 
form $(x_0:x_1:0:0)$,
each of which is of type $A_{w_1-1}$. (As $w_1\mid 2w_0$ and $w_1$ is odd,
we have $\gcd (w_0,w_1)=w_1$.) Totally, there are $2(w_1-1)$ exceptional
divisors arising from these singularities and $w_1-1$ conjugate pairs
contribute to $r=\rank \pic (S)^{\sigma}$.

If $\gcd (w_0,w_1)=1$, then $(x_0:x_1:0:0)$ is not a singularity.
\smallskip

(ii)\ Consider the case where $\gcd (w_0,w_2)\geq 2$ and $w_2$ is even.
$(x_0:0:x_2:0)$ is a singularity. By letting $t=-1$, we see $t^{w_0}=-1$
and $t^{w_2}=1$. Hence $(x_0:0:x_2:0)$ is fixed by $\sigma$ and so are
the exceptional divisors arising from this point. This means that all
exceptional divisors contribute to $r$. The same argument
is valid for the singularities $(x_0:0:0:x_3)$ with even $w_3$.

Therefore it follows from (i) and (ii) that $r=r(Q)+1$ if there is no
odd $w_i$ with $\gcd (w_0,w_i)\geq 2$, and
$$r=r(Q)+1-(w_i-1)=r(Q)-w_i+2$$
if $\gcd (w_0,w_i)\geq 2$ for some odd weight $w_i$.
\vspace{.5\baselineskip}

(b)\ Assume now that $w_0$ is even. Then the normality of weight $Q$ and the
equality $w_0=w_1+w_2+w_3$ imply that there is exactly one even weight $w_i$
for $1\leq i\leq 3$. For simplicity, let $w_1$ be even, and $w_2$ and $w_3$
be odd.
\smallskip

(i)\ Consider the point $(x_0:x_1:0:0)$ with $w_1$ even. Since $\gcd (w_0,w_1)
\geq 2$, it is a singularity. We see $w_1 \mid 2w_0$.

If $w_1\mid w_0$, then $d_1=\gcd (w_0,w_1)=w_1$ and 
$\{ (x_0:x_1:0:0)\}$ consists of two points, each of which is a 
singularity of type $A_{w_1-1}$. There arise $2(w_1-1)$ exceptional 
divisors from them and $w_1-1$ conjugate pairs are fixed by $\sigma$. 
This means that the rank $r$ is the number of exceptional divisors 
minus $w_1-1$.

If $w_1\not\,\mid w_0$, then $\lcm (w_0,w_1)=2w_0$ and $(x_0:x_1:0:0)$ is a
singularity of type $A_{d_1-1}$. This point is fixed by $\sigma$ and so are
the exceptional divisors arising from it.

In summary, the rank $r$ is less than the number of exceptional divisors by
$$-(d_1-1) \biggl( \frac{2w_0}{\lcm (w_0,w_1)}-1 \biggr)=-(d_1-1)
\biggl( \frac{2d_1}{w_1}-1 \biggr).$$
\smallskip

(ii)\ Consider the points $(x_0:0:x_2:0)$ with $w_2$ odd. They are 
singularities if and only if $\gcd (w_0,w_2)\geq 2$. Here 
$w_2 \mid 2w_0$ implies $w_2 \mid w_0$ and $d_2=\gcd (w_0,w_2)=w_2$. 
Such singularities are of type $A_{d_2-1}$. The multiplicity of 
$(x_0:0:x_2:0)$ is $2w_0/\lcm (w_0,w_2)=2$ and they are
$\sigma$-conjugate. Among the $2(w_2-1)$ exceptional divisors, 
$w_2-1$ conjugate pairs are fixed by $\sigma$. Hence the 
rank $r$ is 
less than the number of exceptional divisors by
$$-(d_2-1) \biggl( \frac{2d_2}{w_2}-1 \biggr).$$
The same argument holds for the points $(x_0:0:0:x_3)$ with odd $w_3$.

Therefore the asserted formula of (b) follows from (i) and (ii).
\end{proof}

\begin{ex} For a generic choice of $S_0$, one has $\rank \pic (S_0)^{\sigma} 
=1$ and Theorem \ref{r-and-rq} gives a convenient way to calculate 
invariant $r$. 

(1) For $Q=(7,3,2,2)$, we find that $w_0$ is odd, $r(Q)=9$ and no odd
weight $w_i$ with $\gcd (w_0,w_i)\geq 2$. Hence $r=9+1=10$.

(2) For $Q=(15,10,3,2)$, we find that $w_0$ is odd, $r(Q)=11$ and $\gcd
(15,3)=3$. Hence $r=11-3+2=10$.

(3) For $Q=(8,4,3,1)$, we find that $w_0$ is even, $r(Q)=8$ and $d_1=4$.
Hence $r=8+1-(4-1)(2\cdot 4/4-1) =6$.

(4) For $Q=(10,5,3,2)$, we find that $w_0$ is even, $r(Q)=12$, $d_1=5$
and $d_3=2$. Hence $r=12+1-(5-1)(2\cdot 5/5-1) -(2-1)(2\cdot 2/2-1)=8$.

(5) For $Q=(24,16,5,3)$, we find that $w_0$ is even, $r(Q)=15$, $d_1=8$
and $d_3=3$. Hence $r=15+1-(8-1)(2\cdot 8/16-1) -(3-1)(2\cdot 3/3-1)=14$.
\end{ex}

\begin{ex} \label{no8} Proposition \ref{r-and-a} tells that we can 
calculate $r$ and $a$ by knowing $S^{\sigma}$, but the computation of 
$a$ by finding a $\ZZ$-basis for $\pic (S)^{\sigma}$ is rather 
involved. 
 
Take a look at the $K3$ surface $\#8$ in Yonemura defined by the equation
$$S_0:¥ x_0^2=x_1^4+x_2^6+x_3^{12}\quad \subset \PP^3(6,3,2,1).$$
The involution is defined by $\sigma (x_0)=-x_0$.

We see that $S_0$ is quasi-smooth and the minimal resolution $S$ is a 
$K3$ surface. $S_0$ has four singularities, $P_1$, $P_1^{'}$, $P_2$ and 
$P_2^{'}$ as follows:
\vspace{.5\baselineskip}

\begin{tabular}{|l|l|l|} \hline
Singularity & Type & Exceptional divisors \\ \hline
$P_1 :=(1:1:0:0)$ & $A_{3,2}$ & $E_1+E_2$ \\ 
$P_1^{'} :=(-1:1:0:0)$ & $A_{3,2}$ & $E_1^{'}+E_2^{'}$ \\
$P_2 :=(1:0:1:0)$ & $A_{2,1}$ & $E_3$ \\ 
$P_2^{'} :=(-1:0:1:0)$ & $A_{2,1}$ & $E_3^{'}$  \\ \hline
\end{tabular}
\vspace{.5\baselineskip}

No singularity is fixed by $\sigma$. There is a curve, $C^{'}$, defined 
by $x_0=0$. Its strict transform, $C_7$, on $S$ is of genus $7$ and 
ramified under $\sigma$. We have
$$S^{\sigma}=C_7.$$
Hence $g=7$ and $k=0$. This implies that $r=11-7+0=4$ and $a=11-7-0=4$. 
As $a=4$, the intersection matrix of a $\ZZ$-basis for 
$\pic (S)^{\sigma}$ should have determinant $\pm 2^4$. 

We look for a basis for $\pic (S)^{\sigma}$. An immediate choice for 
a set of four divisors on $S$ fixed by $\sigma$ is $E_1+E_1^{'}$, 
$E_2+E_2^{'}$, $E_3+E_3^{'}$ and $C_7$. But the determinant of their 
intersection matrix is calculated as 
$$\begin{vmatrix}
-4 & 2 & 0 & 0 \\ 
2 & -4 & 0 & 0 \\ 
0 & 0 & -4 & 0 \\ 
0 & 0 & 0 & 12 \end{vmatrix} =-2^6 3^2.$$
Hence they do not form a $\ZZ$-basis for $\pic (S)^{\sigma}$. 

We now consider another set of divisors by replacing $C_7$ with the 
rational curve $D$ defined by $x_3=0$. As a divisor, $D$ is also 
fixed by $\sigma$. As $D$ is a rational curve on a $K3$ surface, 
$D^2=-2$. Since the curve $\{ x_3=0\}$ on $S_0$ passes through 
every singularity,
$$D.(E_1+E_1^{'})=D.(E_3+E_3^{'})=2,\quad D.(E_2+E_2^{'})=0.$$
Hence the determinant of the intersection matrix of these divisors is 
$$\begin{vmatrix} 
-4 & 2 & 0 & 2 \\ 
2 & -4 & 0 & 0 \\ 
0 & 0 & -4 & 2 \\ 
2 & 0 & 2 & -2 
\end{vmatrix} =-2^4.$$ 
This agrees with the above calculation of $a=4$; thus we see that 
$$E_1+E_1^{'},\,\, E_2+E_2^{'},\,\, E_3+E_3^{'}\,\,\mbox{and}\,\,\, D$$ 
form a $\ZZ$-basis for $\pic (S)^{\sigma}$. 
\end{ex}

It is not too difficult to find some subgroup of $\pic (S_0)^{\sigma}$, 
but often difficult to describe $\pic (S_0)^{\sigma}$ completely. 
Above calculations give us a clue to determine $\rank \pic (S_0)^{\sigma}$. 
Let $S$ be the minimal resolution of $S_0$ and let $\EEE$ denote
the subgroup of $\pic (S)$ generated by the exceptional divisors
of the resolution. We have 
$$\pic (S)^{\sigma}\otimes \QQ \cong \pic (S_0)^{\sigma}\otimes \QQ 
\oplus {\EEE}^{\sigma}\otimes \QQ.$$ 
Groups $\EEE$ and $\EEE^{\sigma}$ are easily describable and the rank of
$\pic (S)^{\sigma}$ may be computed from the fixed part $S_0^{\sigma}$
by Proposition \ref{r-and-a} as
$$\rank \pic (S_0)^{\sigma} =\rank \pic (S)^{\sigma} -\rank
{\EEE}^{\sigma}=11+k-g-\rank {\EEE}^{\sigma}.$$
Hence if one finds this number of divisors in $\pic (S_0)^{\sigma}$, 
then they form a
basis for $\pic (S_0)^{\sigma}$ over $\QQ$.

\begin{cor}
{\sl Let $(S,\sigma)$ be one of the $48$ $K3$ surfaces 
(considered in {\em Proposition \ref{prop-BY}}) with involution
$\sigma$ defined by a hypersurface of the form $x_0^2=f(x_1,x_2,x_3)$
where $\sigma$ acts by $\sigma(x_0)=-x_0$. Let
$S^{\sigma}=C_g\cup L_1\cup\ldots\cup L_k$ be the 
decomposition in connected components of $S^{\sigma}$, 
where $C_g$ is a smooth genus $g$ curve
and $L_1,\cdots, L_k$ are rational curves.

Suppose that $f$ is defined by three monomials, so that $S$
is of Delsarte type. Then the Jacobian variety $J(C_g)$ of $C_g$
is also of CM type.}
\end{cor}

\begin{proof} In this case $C_g$ is defined by putting $x_0=0$. So $C_g$
is defined by three monomials and is realized as a Fermat
quotient.
\end{proof} 

In the next section, we discuss another type (non-Borcea type) of 
$K3$ surfaces. Noting the differences from the case $x_0^2=f(x_1,x_2,x_3)$,
we sketch the outline of our algorithm.

\medskip 

\section{Computations of Nikulin's invariants for $K3$ surfaces of 
non-Borcea type}
\label{sect.2.4}

\subsection{Computations of $r$ and $a$}

In this section, we compute $r$ and $a$ for $K3$ surfaces of non-Borcea type, 
namely for a quasi-smooth $K3$ surface $S_0$ in 
$\PP^3(w_0,w_1,w_2,w_3)$ defined by the equation
\begin{equation} \label{non-bo}
x_0^2x_i=f(x_1,x_2,x_3)
\end{equation}
for some $i\ (=1,2,3)$. Let $S$ be the minimal resolution of $S_0$.
Write $C_g$ for the strict transform of the curve defined by $x_0=0$,
which is isomorphic to the curve $f(x_1,x_2,x_3)=0$ in $\PP^2(w_1,w_2,w_3)$.
We assume that $f(x_1,x_2,x_3)=0$ is quasi-smooth (hence smooth) in 
$\PP^2(w_1,w_2,w_3)$ after normalization of the weight. 

Since $Q=(w_0,w_1,w_2,w_3)$ is normalized, every fixed point in $S_0^{\sigma}$ 
must have at least one zero coordinate. There are four cases to consider. 
\medskip

(i)\ The curve $\{ x_0=0\}$ on $S_0$ is fixed by $\sigma$. Since
$f(x_1,x_2,x_3)=0$ is quasi-smooth in $\PP^2(w_1,w_2,w_3)$, its strict 
transform $C_g$ is also fixed by $\sigma$. The genus $g$ can be 
calculated from $\deg f$ and weight $(w_1,w_2,w_3)$ as in the previous section.
\medskip

The rational curves $L_i$ of $S^{\sigma}=C_g\cup L_1\cup\cdots \cup L_k$ 
are obtained by letting $x_1$, $x_2$ or $x_3$ be zero, or from the 
exceptional divisors arising in the desingularization.
\medskip

(ii)\ As in the case of $K3$ surfaces of Borcea type, the one-dimensional 
locus $\{ x_1=0\}$, $\{ x_2=0\}$ or $\{ x_3=0\}$ may be fixed by $\sigma$. 
In addition to them, there may be another one-dimensional locus fixed by 
$\sigma$; it has two zero coordinates, one of which is $x_i$ of (\ref{non-bo}). 

For instance, consider a surface $x_0^2x_1=f(x_1,x_2,x_3)$ with $f(0,0,x_3)=0$. 
If $w_0$ is odd and $w_3$ is even, then the locus $(x_0:0:0:x_3)$
is a line and fixed by $\sigma$. Its strict transform on $S$ is also
fixed by $\sigma$, which gives one $L_i$ of Proposition \ref{r-and-a}.
\medskip

(iii) The rest of the curves in $L_i$'s are obtained from exceptional 
divisors in the resolution $S_0\leftarrow S$. The singularities discussed 
in Lemmas \ref{sing-case1}, \ref{sing-case2} and \ref{sing-case3} also exist 
on surfaces (\ref{non-bo}) and by the same procedure as described there, 
we can find those divisors fixed identically by $\sigma$. (The proof of 
Lemma \ref{sing-case2} needs a little modification; see Lemma \ref{sing-case5}). 
\medskip 

(iv) In addition to the singularities of (iii), we now have a singularity 
$(1:0:0:0)$ on the surface $x_0^2x_i=f(x_1,x_2,x_3)$. The exceptional 
divisors arising from it and fixed by $\sigma$ are determined as follows. 
 
\begin{lem} \label{sing-case4} 
{\sl Let $S_0:\ x_0^2x_i=f(x_1,x_2,x_3)$ be one of the $K3$ surfaces obtained  
in {\rm Theorem \ref{thm-newS}}. Then $P=(1,0,0,0)\in S_0$ 
is a (cyclic quotient) singularity of type $A_{w_0,w_0-1}$. Let $E_1+\cdots + 
E_{w_0-1}$ be the exceptional divisors on $S$ arising from $P\in S_0$. On 
the quotient $S_0/\sigma$, we see that $\sigma (P)$ is a singularity of 
type $A_{2w_0, w_0-1}$ or $A_{w_0, 2(w_0-1)}$. 

{\rm (1)} If $\sigma (P)$ is of type $A_{2w_0, w_0-1}$, then $\sigma$ acts 
on $E_i$ as identity if and only if $i$ is odd. There exist $\displaystyle 
\Bigl[ \frac{w_0}{2} \Bigr]$ such divisors in total, where $[x]$ denotes
the integer part of $x$ as before. 

{\rm (2)} If $\sigma (P)$ is of type $A_{w_0, 2(w_0-1)}$, then $\sigma$ acts 
on $E_i$ as identity if and only if $i$ is even. There exist $\displaystyle 
\Bigl[ \frac{w_0-1}{2} \Bigr]$ such divisors in total. 

The type of singularity at $\sigma (P)$ is determined as follows according 
to the parity of weights: 
\medskip 

\begin{center} 
\begin{tabular}{l|c|c|cc|c} 
 & $w_0$ & $w_i$ & $w_j$ & $w_k$ & Type of $\sigma (P)$ \\ \hline 
$(a)$ & even & even & odd & odd & $A_{2w_0, w_0-1}$ \\ 
$(b)$ & even & odd & odd & odd & $A_{2w_0, w_0-1}$ \\ \hline 
$(c)$ & odd & odd & even & odd & $A_{2w_0, w_0-1}$ \\ 
$(c)^{'}$ & odd & odd & odd & even & $A_{2w_0, w_0-1}$ \\ \hline 
$(d)$ & odd & even & even & odd & $A_{w_0, 2(w_0-1)}$ \\ 
$(d)^{'}$ & odd & even & odd & even & $A_{w_0, 2(w_0-1)}$ \\ \hline 
\end{tabular} 
\end{center} 
} 
\end{lem} 
\medskip 
 
\begin{proof} 
Since $w_0\geq 2$ and $S$ is $K3$, $P=(1,0,0,0)$ is a cyclic quotient 
singularity of type $A_{w_0,w_0-1}$. From the equation 
$x_0^2x_i=f(x_1,x_2,x_3)$, we can choose variables $x_j$ and $x_k$, 
different from $x_0$ and $x_i$, as local parameters around $P$. 
The $\mu_{w_0}$-action at $P$ is then written as 
$$(x_j,x_k)\longmapsto (\zeta^{w_j} x_j, \zeta^{w_k} x_k).$$  
As $\deg f =2w_0+w_i$ and $Q$ is a $K3$ weight, we have $2w_0+w_i=w_0+w_i+ 
w_j+w_k$. This implies $w_0=w_j+w_k$ and, because $Q$ is normalized, 
$\gcd (w_0,w_j)=\gcd (w_0,w_k)=1$. In particular, the congruence $w_j\alpha 
\equiv w_k\pmod{w_0}$ has solution $\alpha \equiv -1\pmod{w_0}$. Now 
$P$ is mapped to $(1,0,0,0)\in S_0/\sigma$ and the group action around 
this point is 
$$(x_j,x_k)\longmapsto (\xi^{w_j} x_j, \xi^{w_k} x_k)$$ 
where $\xi$ ranges over $\mu_{2w_0}$. To find out the type of singularity 
at $\sigma (P)$, we divide the case according to the parity of weights. 
The cases (c) and $(c)^{'}$, (d) and $(d)^{'}$ are essentially the same, so 
we discuss cases (a), (b), (c) and (d). 

In (a) and (b), $w_0$ is even. Since $w_0=w_j+w_k$, both $w_j$ and $w_k$ 
are odd. Hence weight $(2w_0,w_i,w_j,w_k)$ is normalized and the singularity 
$(1,0,0,0)\in S_0/\sigma$ is of type $A_{2w_0, w_0-1}$. As in Lemma 
\ref{sing-case3}, we see that $\sigma$ acts on $E_i$ as identity if and 
only if $i$ is odd, and there are $[w_0/2]$ such divisors. 

In (c), $w_0$ is odd. Since $w_0=w_j+w_k$, $w_j$ and $w_k$ have different 
parity. Because $w_i$ is odd, the weight $(2w_0,w_i,w_j,w_k)$ is normalized. 
Hence the singularity $(1,0,0,0)\in S_0/\sigma$ is of type $A_{2w_0, w_0-1}$. 
As in Lemma \ref{sing-case3}, $\sigma$ acts on $E_i$ as identity if and 
only if $i$ is odd, and there are $[w_0/2]$ such divisors. 

In (d), $w_0$ is odd. Since $w_0=w_j+w_k$, either $w_j$ or $w_k$ is even, 
say $w_j$ is even. Because $w_i$ is even in this case, $(2w_0,w_i,w_j,w_k)$ 
is not yet normalized. By normalization 
$$\PP^3(2w_0,w_i,w_j,w_k)\cong \PP^3 \Big( w_0,\frac{w_i}{2}, 
\frac{w_j}{2},w_k \Big)$$ 
the group action around $(1,0,0,0)\in S_0/\sigma$ is now written as 
$$(x_j,x_k)\longmapsto (\xi^{w_j/2} x_j, \xi^{w_k} x_k)$$ 
where $\xi$ ranges over $\mu_{w_0}$. As $\gcd (w_0,w_j/2)=\gcd (w_0,w_k)=1$ 
and $w_0=w_j+w_k$, we have the congruence 
$$\frac{w_j}{2}2(w_0-1)\equiv w_k\pmod{w_0}.$$ 
This shows that $(1,0,0,0)\in S_0/\sigma$ is of type $A_{w_0, 2(w_0-1)}$ 
(to be more precise, type $A_{w_0, w_0-2}$). As in Lemma \ref{sing-case2}, 
$\sigma$ acts on $E_i$ as identity if and only if $i$ is even, and there are 
$[(w_0-1)/2]$ such divisors. 
\end{proof} 

\begin{lem} \label{sing-case5} 
{\sl Let $S_0:\ x_0^2x_i=f(x_1,x_2,x_3)$ be one of the $K3$ surfaces obtained  
in {\rm Theorem \ref{thm-newS}}. If $w_j$ and $w_k$ are the weight other than 
$w_0$ and $w_i$, then $\gcd (w_j, w_k)=1$. If $d=\gcd (w_i,w_j)\geq 2$, then 
$d$ must be $2$ and $P=(0,x_i,x_j,0)$ with $x_ix_j\neq 0$ is a singularity of 
type $A_{2,1}$ fixed by $\sigma$. There arises one exceptional divisor from 
$P$ and it is not ramified under the quotient $S\lra S/\sigma$. 
} 
\end{lem} 

\begin{proof} 
Let $d=\gcd (w_j,w_k)$. The relation $w_0=w_j+w_k$ implies $d\mid w_0$. 
But this is not possible as the weight $Q$ is normalized unless $d=1$. 

Let $d=\gcd (w_i,w_j)$. As in the proof of Lemma \ref{sing-case2}, $f$ 
contains a term without $w_k$. By the relation $\deg f =2w_0+w_i$, we 
see that $d\mid 2w_0$. Since $Q$ is normalized, $d\mid 2$ and hence $d=2$. 
The rest of the proof is similar to Lemma \ref{sing-case2}. 
\end{proof} 

We combine (ii), (iii) and (iv) to calculate the number of rational curves 
$L_i$. Then Proposition \ref{r-and-a} gives the value for $r$ and $a$. 
%However, we are not able to obtain closed formula for $r$ (similar to
%the formula obtained in Theorem 5.5) for these non-Borcea type $K3$ surfaces.
\medskip

\begin{ex} \label{no60}
Consider the $K3$ surface $\#60$ in Yonemura. Dropping several monomials,
we choose the equation
$$S_0:\ x_0^2x_2+x_1^3+x_1x_2^3+x_3^{18}=0\quad \subset \PP^3(7,6,4,1).$$
The involution is defined by $\sigma (x_0)=-x_0$. We see that $S_0$ is
quasi-smooth and the minimal resolution $S$ is a $K3$ surface. $S_0$ has
three singularities, $P_1$, $P_2$ and $P_3$ as follows:
\vspace{.5\baselineskip}

\begin{tabular}{|l|l|l|} \hline
Singularity & Type & Exceptional divisors \\ \hline
$P_1 :=(1:0:0:0)$ & $A_{7,6}$ & $E_1+E_2+E_3+E_4+E_5+E_6$ \\
$P_2 :=(0:0:1:0)$ & $A_{4,3}$ & $E_7+E_8+E_9$ \\
$P_3 :=(0:-1:1:0)$ & $A_{2,1}$ & $E_{10}$  \\ \hline
\end{tabular}
\vspace{.5\baselineskip}

Every singularity is fixed by $\sigma$, and $E_2$, $E_4$, $E_6$ and
$E_8$ are ramified under $\sigma$ (acting on the minimal resolution $S$).

There are two curves on $S_0$ fixed by $\sigma$, namely $C^{'}$ defined
by $x_0=0$ and $L^{'}$ defined by $x_3=0$. $C^{\prime}$ has genus $3$ and
$L^{'}$ is a projective line. Their strict transforms $C\ (=C_3)$ and
$L$ on $S$ are ramified under $\sigma$. We have
$$S^{\sigma}=C_3\cup E_2\cup E_4\cup E_6\cup E_8\cup L.$$
Hence $g=3$ and $k=5$. This implies $r=13$ and $a=3$.

\end{ex}

\begin{ex} \label{no89}
Consider the $K3$ surface $\#89$ in Yonemura. Dropping several monomials,
we choose the equation
$$S_0:\ x_0^2x_3+x_1^3x_2+x_1x_2^4+x_3^{11}=0\quad \subset \PP^3(5,3,2,1).$$
The involution is defined by $\sigma (x_0)=-x_0$. We see that $S_0$ is
quasi-smooth and the minimal resolution $S$ is a $K3$ surface. $S_0$ has
three singularities, $P_1$, $P_2$ and $P_3$ as follows:
\vspace{.5\baselineskip}

\begin{tabular}{|l|l|l|} \hline
Singularity & Type & Exceptional divisors \\ \hline
$P_1 :=(1:0:0:0)$ & $A_{5,4}$ & $E_1+E_2+E_3+E_4$ \\
$P_2 :=(0:1:0:0)$ & $A_{3,2}$ & $E_5+E_6$ \\
$P_3 :=(0:0:1:0)$ & $A_{2,1}$ & $E_7$  \\ \hline
\end{tabular}
\vspace{.5\baselineskip}

Every singularity is fixed by $\sigma$, and $E_2$, $E_4$ and $E_6$ are
ramified under $\sigma$ (acting on the minimal resolution $S$).

There are two curves on $S_0$ fixed by $\sigma$, namely $C^{'}$ defined
by $x_0=0$ and $L^{'}$ defined by $x_1=x_3=0$. $C^{'}$ has genus $5$ and
$L^{'}$ is a projective line. Their strict transforms $C\ (=C_5)$ and
$L$ on $S$ are ramified under $\sigma$. We have
$$S^{\sigma}=C_5\cup E_2\cup E_4\cup E_6\cup L.$$
Hence $g=5$ and $k=4$. This implies $r=10$ and $a=2$.

\end{ex}

\begin{cor}
{\sl Let $(S,\sigma)$ be one of the $86$ $K3$ surfaces of Delsarte
type with involution $\sigma$ in {\em Theorem \ref{thm-DelsarteS}}. 
Let $C_g$ be the genus $g$
curve in the fixed locus $S^{\sigma}=C_g\cup L_1\cdots\cup L_k$ (where
$L_i$ are rational curves).
Then $C_g$ is of CM type in the sense that its Jacobian variety
$J(C_g)$ is a CM abelian variety of dimension $g$.}
\end{cor}

\begin{proof}
If $\sigma$ acts as $\sigma (x_i)=-x_i$ on $S$, then $C_g$ is defined
by letting $x_i=0$. It is a curve of Delstarte type and hence of CM type.
\end{proof}

\subsection{Realization of Nikulin's invariants} 

We briefly discuss how many Nikulin's triplets $(r,a,\delta)$ 
can be realized by our $K3$ surfaces. To realize as many 
triplets as possible, we introduce more involutions than those 
considered in previous sections. 

First, we summarize the results of previous sections where $\sigma$ 
acts on the variable $x_0$ (with the highest weight among $x_i$'s).  

\begin{thm} \label{triplets1} 
{\sl Let $(S,\sigma)$ be one of the $92$ $K3$ surfaces in {\em Theorem  
\ref{thm-newS}} with involution $\sigma (x_0)=-x_0$. Among the $75$ 
triplets $(r,a,\delta )$ of Nikulin, at least $29$ triplets are 
realized with such $K3$ surfaces. See {\em Tables \ref{table1}} to 
{\em \ref{table-y4}} and {\em \ref{table-y5}} of {\em Section 
\ref{tables}} for the list of defining equations of $S_0$ and 
the values for $(r,a)$; in most cases, $S_0$ is defined by 
$$x_0^2=f(x_1,x_2,x_3)\quad \mbox{ or } \quad x_0^2x_i=f(x_1,x_2,x_3).$$ 
} 
\end{thm}

\begin{rem} 
We say ``at least $29$'' because we computed only the invariants $r$ 
and $a$. If we are to find $\delta$, we should have a $\ZZ$-basis 
for $\pic (S)$ or calculate intersection numbers of divisors on $S$. 
We leave this as a future problem. Once $\delta$ is calculated, the 
number of realizable triplets may increase. 
\end{rem}

When the weight $Q=(w_0,w_1,w_2,w_3)$ is fixed and $\sigma$ 
is defined by $\sigma (x_0)=-x_0$ 
on the surface $x_0^2=f(x_1,x_2,x_3)$ or $x_0^2x_i=f(x_1,x_2,x_3)$, 
no matter what quasi-smooth equation we choose for $f(x_1,x_2,x_3)$, 
the fixed locus $S^{\sigma}$ is the same. Hence changes in equation 
$f(x_1,x_2,x_3)$ do not lead to any new pairs of $r$ and $a$.  

On the other hand, even with the same $S_0$, changing the involution 
$\sigma$ may change the fixed locus $S^{\sigma}$ and thus
also $r$ and $a$ may change. %, which has a possibility that we 
%obtain a different pair of $r$ and $a$. 
We use  this approach by letting $\sigma$ act on some variable 
$x_i$ other than $x_0$.  

\begin{thm} \label{triplets2} 
{\sl Let $(S,\sigma)$ be one of the $92$ $K3$ surfaces in {\em Theorem  
\ref{thm-newS}} having an involution $\sigma$ on a variable $x_i$ 
other than $x_0$. They are listed on {\em Table \ref{table-extra}} 
in {\em Section \ref{tables}} with the values of $r$ and $a$. Compared 
with the triplets obtained in {\em Theorem \ref{triplets1}}, at 
least $14$ more triplets are realized with such $\sigma$ actions. 
Among the $75$ triplets $(r,a,\delta )$ of Nikulin, the total number 
of triplets we realize is at least $40$. The values for $r$ and $a$ 
of such $40$ triplets are as follows: 

$$\begin{array}{|c|cccccccccccccccccccc|} \hline 
r & 1 & 2 & 3 & 4 & 5 & 6 & 7 & 8 & 9 & 10 & 11 & 12 & 13 & 14 & 15 
& 16 & 17 & 18 & 19 & 20 \\ \hline 
a & 1 & 0 & 1 & & & 2 & 3 & 6 & 1 & 0 & 1 & 6 & 3 & 2 & 5 & 2 & 1 & 0 & 1 & 2 \\ 
 & & 2 & & & & 4 & 7 & 8 & 9 & 2 & 9 & 8 & 5 & 4 & 7 & 6 & 3 & 2 & 3 & \\ 
 & & & & & & & & & & 4 & & & 9 & 6 & & & 5 & 4 & & \\ 
 & & & & & & & & & & 6 & & & & & & & & & & \\ 
 & & & & & & & & & & 8 & & & & & & & & & & \\ \hline 
\end{array}$$ 
}   
\end{thm} 

\begin{proof} 
The $\sigma$-fixed locus $S_0^{\sigma}$ can be determined by the 
same method as in previous subsections. 
In particular, the singularities of $S_0$ are independent of the 
choice of $\sigma$. We use Lemmas \ref{sing-case1}, \ref{sing-case2}, 
\ref{sing-case3} and \ref{sing-case4} to find which exceptional 
divisors are fixed by involution $\sigma$. 
\end{proof} 

\begin{rem} 
In some cases (say, Case $\#57$ of Table \ref{table-y2}), the 
$\sigma$-action on $x_i$ $(i\neq 0)$ 
gives the same $S^{\sigma}$ as the action by $\sigma (x_0)=-x_0$. 
Such cases are omitted from Table \ref{table-extra}.  
\end{rem} 

\begin{rem} 
Hisanori Ohashi has communicated us an idea of even more different 
types of involutions on $S_0$. We thank him for the idea of different 
involutions. We plan to work on it in a subsequent paper. 
\end{rem} 

\section{Calabi--Yau threefolds of Borcea--Voisin type}
\label{sect3}

\subsection{Construction of Calabi--Yau threefolds
of Borcea--Voisin type}\label{sect3.1}

In this section, we recall the Borcea--Voisin construction of
Calabi--Yau threefolds.

Let $E$ be an elliptic curve with the standard involution
$\iota$, and let $(S,\sigma)$ be a pair of a $K3$ surface $S$
with involution $\sigma$ acting by $-1$ on $H^{2,0}(S)$.  By
the classification theorem of Nikulin, the isomorphism class of
such pairs $(S,\sigma)$ is determined by a triplet $(r,a,\delta)$
as we discussed in Section 2.

Now we consider the product $E\times S$, and the quotient
threefold $$E\times S/\iota\times\sigma.$$
Obviously, this quotient is singular, having cyclic
quotient singularities.  We resolve singularities to obtain a
smooth crepant resolution, denoted by $X=X(r,a,\delta)$, which
is a Calabi--Yau threefold; we call it a {\it Calabi--Yau threefold
of Borcea--Voisin type}. It is plain that a Calabi--Yau threefold
of Borcea--Voisin type is equipped with the following two
fibrations: the elliptic fibration with constant fiber $E$ induced from 
the projection $E\times S/\iota\times\sigma\to S/\sigma$, and the $K3$ 
fibraton with the constant fiber $S$ induced from the projection 
$E\times S/\iota\times\sigma \to E/\iota$.

\begin{prop}\,\,{\em [Borcea \cite{B94}]} {\sl The Hodge numbers of
the Calabi--Yau threefold $X=X(r,a,\delta)$ of Borcea--Voisin type 
are determined by the given triplet $(r,a,\delta)$ and 
by the data from the fixed locus $S^{\sigma}$: Indeed,
$$h^{1,1}(X)=5+3r-2a=1+r+4(k+1),$$
$$h^{2,1}(X)=65-3r-2a=1+(20-r)+4g$$
where $k,\,g$ are described in {\em Proposition \ref{r-and-a}}.
The Euler characteristic of $X$ is
$$e(X)=2(h^{1,1}(X)-h^{2,1}(X))=2(r-10).$$}
\smallskip

{\em [Voisin \cite{V93}]} {\sl Put $N:=1+k$ and $N^{\prime}:=g$. That is,
$N$ is the number of components, and $N^{\prime}$ is the sum of
genera of components, of $S^{\sigma}$.  Then
$$h^{1,1}(X)=11+5N-N^{\prime},$$
$$h^{2,1}(X)=11+5N^{\prime}-N$$
and the Euler characteristic of $X$ is
$$e(X)=2(h^{1,1}(X)-h^{2,1}(X))=12(N-N^{\prime}).$$
}
\end{prop}

Now we discuss briefly resolution of singularities; detailed
discussions are in Section 6. As above, let $\iota: E\to E$
be the standard involution. The fixed part $E^{\iota}$ consists
of four points $\{P_1,P_2,P_3,P_4\}$.

We consider the generic case (I) when the fixed part
$S^{\sigma}$ of $S$ is given by
$$S^{\sigma}=C_g\cup L_1\cup L_2\cup \cdots \cup L_k$$
where $C_g$ is a smooth curve of genus $g\geq 1$ and $L_i$
is a rational curve for each $i=1,\cdots, k$.

\begin{prop}{\sl
The quotient threefold $E\times S/\iota\times\sigma$ has
singularities along $\{P_i\}\times S^{\sigma}$ $(i=1,2,3,4)$.
Each singular locus is a cyclic quotient singularity by a group
action of order $2$.

By resolving singularities, we obtain a smooth Calabi--Yau
threefold $X$:
$$E\times S/\iota\times\sigma \leftarrow X.$$

The exceptional divisors are four copies of a union of ruled
surfaces
$$S^{\sigma}\times\PP^1:=(C_g\times\PP^1)\cup(L_1\times\PP^1)\cup
\cdots\cup (L_k\times\PP^1).$$
}
\end{prop}

\subsection{Realization of Calabi--Yau threefolds of Borcea--Voisin type 
as hypersurfaces over $\QQ$}\label{sect3.2}

The construction of Calabi--Yau threefolds of Borcea--Voisin type we have
discussed so far are geometric in nature. In order to study arithmetic
of these Calabi--Yau threefolds, we wish to have defining equations   
by, e.g., hypersurfaces or complete intersections defined over 
$\QQ$ in weighted projective spaces. We require that the zero loci of 
these equations define singular Calabi--Yau threefolds and whose
resolution would be birationally equivalent to our Calabi--Yau threefolds 
of Borcea--Voisin type.  In fact, the constructions of such singular
models have been already carried out in Goto--Kloosterman--Yui \cite{GKY10}, 
using the so-called {\it twist maps}.  Now we will briefly recall such 
constructions.  
\smallskip

We start with examples.  Let $\PP^2(k+1,k,1)$ be a weighted projective 
$2$-space with weight$(k+1,k,1)$ of degree $2(k+1)$. Let 
$\PP^3(w_0,w_1,w_2,w_3)$ be a weighted projective $3$-space with weight
$(w_0,w_1,w_2,w_3)$ of degree $d:=\sum_{i=0}^3 w_i$.  A 
{\it twist map} may be defined as follows. Assume
that $\mbox{gcd}(k+1,w_0)=1$. Let 
$$\Phi: \PP^2(k+1,k,1)\times\PP^3(w_0,w_1,w_2,w_3)$$
$$\mapsto \PP^4(kw_0,w_0,(k+1)w_1,(k+1)w_2,(k+1)w_3)$$
into a weighted projective $4$-space with degree $(k+1)d$.
The twist map was defined in Goto--Kloosterman--Yui
\cite{GKY10} explicitly, and is given by
$$\Phi :((y_0:y_1:y_2),(x_0:x_1:x_2:x_3))\mapsto
(y_1(\frac{x_0}{y_0})^{k/k+1}:y_2(\frac{x_0}{y_0})^{k/k+1}:x_1:x_2:x_3).$$
% We will exhibit twist
%maps below for $k=1$ and $2$. 
This will produce many singular Calabi--Yau threefolds defined by
hypersurfaces in weighted projective $4$-spaces.

%Here are some examples.
Take $k=1$.  Then we have an elliptic curve
$$E_2: y_0^2=y_1^4+y_2^4\subset\PP^2(2,1,1).$$
If we take $k=2$, we obtain an elliptic curve
$$E_3: y_0^2=y_1^3+y_2^6\subset\PP^2(3,2,1).$$
Both $E_2$ and $E_3$ have complex multiplication, by
$\ZZ[\sqrt{-1}]$ and $\ZZ[\sqrt{-3}]$, respectively.

\begin{prop} [Borcea \cite{B94}] \label{prop7.1}
{\sl Let $E_2$ and $E_3$ be elliptic curves defined above. 
Let $S_0: x_0^2=f(x_1,x_2,x_3)\subset\PP^3(w_0,w_1,w_2,w_3)$
be one of the $40$ $K3$ surfaces from {\em Tables \ref{table1}} 
and {\em \ref{table2}}. Let $S$ be the minimal resolution of $S_0$. 

{\em (a)} Suppose that $w_0$ is odd, and that $w_0=w_1+w_2+w_3$.
Then there is a twist map
$$\PP^2(2,1,1)\times \PP^3(w_0,w_1,w_2,w_3)\cdots \lra 
\PP^4(w_0,w_0,2w_1,2w_2,2w_3)$$
given by
$$(y_0:y_1:y_2)\times (x_0:x_1:x_2:x_3)\mapsto
(y_1(\frac{x_0}{y_0})^{1/2}: y_2(\frac{x_0}{y_0})^{1/2}:x_1:x_2:x_3).$$
The product $E_2\times S_0$ maps generically $2:1$ to the hypersurface
of degree $4w_0$ of the form
$$z_0^4+z_1^4=f(z_2,z_3,z_4)$$
where $f$ is a homogeneous polynomial over $\QQ$ of degree $4w_0$.
This is a singular model for a Calabi--Yau threefold
$E_2\times S/\iota\times\sigma$ in $\PP^4(w_0,w_0,2w_1,2w_2,2w_3)$.

\smallskip

{\em (b)} Suppose that $w_0$ is even but not divisible by $3$,
and that $w_0=w_1+w_2+w_3$. Then there is a twist map
$$\PP^2(3,2,1)\times \PP^3(w_0,w_1,w_2,w_3)\cdots \lra
\PP^4(2w_0,w_0,3w_1,3w_2,3w_3)$$
given by
$$(y_0:y_1:y_2)\times (x_0:x_1:x_2:x_3)\mapsto
(y_1(\frac{x_0}{y_0})^{2/3}:y_2(\frac{x_0}{y_0})^{1/3}:x_1:x_2:x_3).$$
The product $E_3\times S_0$ maps generically $2:1$ to the hypersurface
of degree $6w_0$ of the form
$$z_0^3+z_1^6=f(z_2,z_3,z_4)$$
where $f$ is a homogeneous polynomial over $\QQ$ of degree $6w_0$.
This is a singular model for a Calabi--Yau threefold
$E_3\times S/\iota\times\sigma$ in $\PP^4(2w_0,w_0,3w_1,3w_2,3w_3)$.
}
\end{prop}

\begin{prob}
When $w_0$ is divisible by $6$, describe the twist map explicitly and
construct hypersurface equations defining the Calabi--Yau 
threefolds modifying twist maps.
\end{prob}

\begin{prop} \label{prop7.2} 
{\sl There are in total $40$ Calabi--Yau threefolds 
corresponding to $K3$ surfaces $S_0$ in {\em Tables \ref{table1}} and 
{\em \ref{table2}} which are
realized by quasi-smooth hypersurfaces over $\QQ$ in weighted 
projective $4$-spaces by the above construction.} 
\end{prop}

\begin{thm}\label{thm7.3}
{\sl 
Suppose that $S$ is the minimal resolution of one of the $45$ $K3$ surfaces of 
{\em Tables \ref{table1}} to {\em \ref{table3}} excluding 
$\#90, \#91, \#93$. Then 
the associated $45$ Calabi--Yau threefolds of Borcea--Voisin type constructed
above are all of CM type.}
\end{thm}

\begin{proof} 
We follow Borcea \cite{B94} Proposition 1.2. 
Let $h$ be the rational Hodge structure of $H^3(E\times S,\QQ)$,
and let $h_E$ and $h_S$ denote, respectively, the Hodge structures 
of $H^1(E,\QQ)$ and $H^2(S,\QQ)$.  Then
$$(H^3(E\times S,\QQ), h)=(H^1(E,\QQ), h_E)\otimes (H^2(S,\QQ), h_S).$$
(See Voisin \cite{V93}, Th\`eor\'eme 11.38.)  Since the
fixed locus of $S^{\sigma}$ are given by a curve $C_g$ ($g\geq 1$) and 
rational curves $L_i\, (i=1,\cdots,k)$, the rational polarized Hodge structure 
$h_X$ of a Calabi--Yau threefold $X=\widetilde{E\times S/\iota\times\sigma}$ 
is given by the rational sub-Hodge structure of $(H^3(E\times S,\QQ),h)$
together with those arising from exceptional divisors associated
to the curve $C_g$ of genus $g\geq 1$ in the fixed locus $S^{\sigma}$. 
%(See Voisin \cite{Voi}).
We get
$$(H^3(X,\QQ),h_X)\simeq(V_{-}\cap H^2(S,\QQ), h_{-})\otimes (H^1(E,\QQ), h_E)$$
$$\oplus (H^1(C_g,\QQ), h_{C_g})$$ 
%(H^{1,0}(C_g,\QQ), h_{C_g})\oplus H^{0,1}(C_g,\QQ),h_{C_g}))$$
where $(V_{-}, h_{-})$ denotes the restricted Hodge structure
on $H^{1,1}(S)$ of the $-1$ eigenspace, 
and $h_{C_g}$ is the Hodge structure of $H^1(C_g,\QQ)$. Note 
that there is an associated Abel--Jacobi map 
$H^{1,0}(C_g)\to H^{2,1}(X)$ and we identify its image 
with $H^{1,0}(C_g)$ in $H^{2,1}(X)$ (see Clemens
and Griffiths \cite{CG72}).
%We know (Borcea \cite{B92}, Proposition 1.2) 
Then $h_X$ is of CM type if and only if
$h_E$ and $h_S$ (more precisely, $h_{-}$ and $h_{C_g}$) are all of CM type. 
Now with our choices of $E$ and $S$, 
$E=E_2$ or $E_3$ is of CM type, and $S$ is of CM type (in particular,
$C_g$ is of CM type (cf. Lemma 5.13 below)) . Therefore,
$h_E$ and $h_S$ ($h_{-}$ and $h_{C_g}$) are all of CM type, and 
hence $X$ is of CM type. (Cf.  Borcea \cite{B92}, Proposition 1.2).
\end{proof}

\begin{rem}
More examples of CM type Calabi--Yau threefolds of Borcea--Voisin
type may be obtained by taking any CM type elliptic
curves $E$. In fact, one notices that
any elliptic curve $E$ can be embedded in $\PP^2(2,1,1)$, with
equation
$$E :x_0^2=x_2(x_1^3+ax_0x_2+bx_2^3)\quad\mbox{with $a,b\in\QQ$}.$$
In particular, elliptic curves with CM by a quadratic field
but with $j$-invariant in $\QQ$ can be realized in this
way.  

%Define a map $(x_0:x_1:x_2)\mapsto (y:x:1)$. Then $E$ is transformed
%to $E^{\prime}: y^2=x^3+ay+b\quad\Leftrightarrow\quad y^2-ay=x^3+b$ in the 
%affine coordinates. The $j$-invariant is computed to be $j=0$. Thus, 
%$E^{\prime}$ is a CM ellpitic curve. 

Rohde \cite{R09} (Example A.1.9) gave four examples of elliptic curves over
$\QQ$ with complex multiplication in the usual projective $2$-space.
\end{rem}

For the additional $41$ pairs $(S,\sigma )$ of $K3$ surfaces with 
involution $\sigma$ of Theorem \ref{thm-DelsarteS} (b) (see Tables 
\ref{table-y2}, \ref{table-y4} and \ref{table-y6}), 
the situation is slightly different from the above cases. 
Calabi--Yau threefolds $X$ are birational to hypersurfaces over 
$\QQ$, but they are not quasi-smooth. More precisely, we have the 
following result. 

\begin{thm} \label{thm7.3} 
{\sl Let $(S,\sigma)$ be (the minimal resolution of) one of the $41$ 
pairs of $K3$ surfaces with 
involution $\sigma$ of {\em Theorem \ref{thm-DelsarteS} (b)}, which is 
not in the list of Borcea. Let $E$ be an elliptic curve over $\QQ$ with 
involution $\iota$ with or without CM. Take the product $E\times S$ 
and consider the quotient threefold $E\times S/\iota\times \sigma$. 
Resolving singularities, we obtain a smooth Calabi--Yau threefold $X$ 
over $\QQ$. Further, $X$ is of CM type if and only if $E$ is of CM type. 

About the realization of $X$ as a hypersurface, the following holds. 

{\em (a)} If $(S,\sigma )$ is one of the $K3$ surfaces listed in {\em 
Tables \ref{table-y2}} and {\em \ref{table-y4}} other than $\#22$ and 
$\#58$, then $S$ is birational to $x_0^2x_i+f(x_1,x_2,x_3)=0$ for some 
$i\neq 0$ and $X$ is birational to a (non-quasi-smooth) hypersurface 
over $\QQ$ defined by 
$$\begin{cases} 
z_{i+1}(z_0^4+z_1^4)+f(z_2,z_3,z_4)=0 & \mbox{ if $w_0$ is odd and $E=E_2$} \\ 
z_{i+1}(z_0^3+z_1^6)+f(z_2,z_3,z_4)=0 & \mbox{ if $w_0$ is even but not 
divisible by $3$ and $E=E_3$.}
\end{cases}$$  

{\em (b)} Let $(S,\sigma )$ be one of the $K3$ surfaces listed in {\em 
Table \ref{table-y6}} other than $\#16$. If we choose $E=E_2$, then 
$X$ is birational to the following (non-quasi-smooth) hypersurface 
over $\QQ$: 
$$\begin{cases} 
(z_0^4+z_1^4)^2+z_2^3+z_3^4+z_4^6=0 & \mbox{ in $\#2$} \\ 
(z_0^4+z_1^4)^2+z_2^3+z_2z_3^3+z_3z_4^4=0 & \mbox{ in $\#52$ } \\ 
z_3(z_0^4+z_1^4)^2+z_2^3+z_2z_4^3+z_3^3z_4=0 & \mbox{ in $\#84$. } 
\end{cases}$$  
}
\end{thm}

\begin{proof} 
Since the singular locus of $E\times S/\iota\times \sigma$ is defined 
over $\QQ$, resolving singularities, we obtain a smooth Calabi--Yau 
threefold $X$ over $\QQ$. Here $X$ is not necessarily of CM type. By the 
same argument as for proof of Theorem \ref{thm7.3}, $X$ is of CM type if 
and only if each component, $E$ and $S$, is of CM type. Since we already 
know that $S$ is of CM type, $X$ is of CM type if and only if $E$ is a 
CM elliptic curve over $\QQ$. 

(a) If we choose an appropriate elliptic curve $E$, then the twist map of 
Proposition \ref{prop7.1} works for $x_0^2=-f(x_1,x_2,x_3)/x_i$. Depending 
on the parity of $w_0$, we obtain the equations as claimed. 

(b) Since the variable associated with the involution $\sigma$ carries 
an odd weight, we may choose $E=E_2$. Then the twist map of Proposition 
\ref{prop7.1} (a) works for $x_0^2=\sqrt{-f(x_1,x_2,x_3)}$ or 
$x_0^2=\sqrt{-f(x_1,x_2,x_3)/x_i}$ and we obtain the equations as 
asserted. 
\end{proof}

\begin{rem} Note that $K3$ surfaces $S_0$ realized by Yonemura in 
weighted projective $3$-spaces are often singular. To have smooth $K3$ 
surfaces, we ought to consider minimal resolutions $S$. The 
involution $\sigma$ is lifted to $S$ and we use $S$ to 
carry out the above construction of Calabi--Yau threefolds $X$. The 
procedure is shown as follows: 
\medskip
$$\begin{array}{ccccc}
E\times S_0 & \longleftarrow & E\times S & & \\
& & \downarrow & & \\
&& E\times S/\iota\times\sigma & \longleftarrow & X
\end{array}
$$
\end{rem}

\begin{rem}
The above constructions work with any elliptic curves, not only  
with $E_2$ and $E_3$. 
\end{rem}

\subsection{Singularities and resolutions on Calabi--Yau threefolds
of Borcea--Voisin type}\label{sect3.3}

Let $S_0$ be a $K3$ surface defined by a weighted hypersurface
$$S_0: x_0^2=f(x_1,x_2,x_3)\subset\PP^3(w_0,w_1,w_2,w_3)$$
of degree $\mbox{deg}(f)=w_0+w_1+w_2+w_3$. The involution $\sigma$ is
given by $\sigma(x_0)=-x_0$. The singularities on $S_0$ are
determined by the weight. Let $S$ be the
minimal resolution of $S_0$. The involution $\sigma$ is
extended to $S$. Let $S^{\sigma}$ be the
fixed part of $S$ by $\sigma$.

Let $E$ be an elliptic curve
$$E_2 : y_0^2=y_1^4+y_2^4\subset\PP^2(2,1,1)$$
or
$$E_3 : y_0^2=y_1^3+y_2^6\subset\PP^2(3,2,1),$$
defined in Section 4.2.
The involution $\iota$ is given by $\iota(y_0)=-y_0$.
The fixed part $E^{\iota}$ consists of four points
$$E^{\iota}=\{P_1,P_2,P_3,P_4\}.$$

We use for $E$ either $E_2$ if $w_0$ is even, or $E_3$ if $w_0$ is odd. Take
the quotient threefold $E\times S/\iota\times\sigma$.  

\begin{rem}
In fact, the above statement is true for any elliptic curve
$$E: y^2=x^3+ax+b\in\PP^2(1,1,1)\quad\mbox{with $a,b\in\QQ$}.$$
$E$ has the involution $y\to -y$ and the fixed points
consists of $4$ points.  Thus, there is no need to
confine our discussions to $E_2$ or $E_3$.  We will
get extra singularities working in weighted projective
spaces, but this is not intrinsic to the Borcea--Voisin
construction.    
\end{rem}

Let $X$ be a smooth resolution of $E\times S/\iota\times\sigma$.
Then the singular loci $\{P_i\}\times S^{\sigma}$ are determined from 
the weight of $S_0$ and the singularity data of the ambient space.
\smallskip

Here are examples.

\begin{ex} \label{ex7-1} 
Let $$S_0: x_0^2=x_1^5+x_2^5+x_3^{10}\subset\PP^3(5,2,2,1).$$
(This is $\#6$ in Yonemura =$\#2$ in Borcea.)
Let $$E_2: y_0^2=y_1^4+y_2^4\subset\PP^2(2,1,1).$$

$\bullet$ Then $S_0$ is a singular $K3$ surface and the singular locus is:
$$\mbox{Sing}(S_0)=\{(0:x_1:x_2:0)\,|\,x_1^5+x_2^5=0\,\}
=\{Q_1,Q_2,Q_3,Q_4,Q_5 \}$$
where every $Q_i$ is a cyclic quotient singularity of type $A_1$.

$\bullet$ Let $C^{'}$ be a curve on $S_0$ defined by $x_0=0$, that is,
$$C^{'}=\{\,x_0=0\}:x_1^5+x_2^5+x_3^{10}=0\subset\PP^2(2,2,1).$$
Since $\PP^2(2,2,1)\simeq \PP^2(1,1,1)$, $C^{'}$ is identified with
$$C^{'}: x_1^5+x_2^5+x_3^5=0\subset\PP^2$$
which is a smooth curve of genus $6$.

$\bullet$ Let $L^{'}$ be a curve on $S$ defined by $x_3=0$, that is,
$$L^{'}=\{\, x_3=0\,\}: x_0^2=x_1^5+x_2^5\subset\PP^2(5,2,2).$$
Since $\PP^2(5,2,2)\simeq \PP^2(5,1,1)$, $L^{'}$ is identified with
$$L^{'}: x_0=x_1^5+x_2^5\subset\PP^2(5,1,1)$$
and hence $L^{'}$ is a rational curve.

$\bullet$ We see that
$$C^{'}\cap L^{'}=\{Q_1,Q_2,Q_2,Q_4,Q_5\}.$$

$\bullet$ Let $S$ be the minimal resolution of $S_0$. The
involution $\sigma$ is lifted to $S$.
Let $C_6$ and $L_1$ be the respective strict transforms of $C^{'}$
and $L^{'}$ to the minimal resolution $S$.
Let $\EE_1,\cdots, \EE_5$ be the exceptional divisors on $S$
arising from singularities $Q_i,\, (i=1,\cdots,5)$, respectively.

Then $C_6$ and $L_1$ are fixed by $\sigma$, but not the exceptional
divisors $\EE_i$ for any $i\in\{1,2,\cdots, 5\}$. Hence
$$S^{\sigma}=C_6\cup L_1$$ 
and we see $g=6$ and $k=1$ (so $r=6, \, a=4$ in Nikulin's notation). 
The resolution picture is given in Figure \ref{eg7-1}, where the curves 
in boldface are fixed by $\sigma$.   
\medskip

\begin{figure} 
\centering 
\includegraphics[width=5cm]{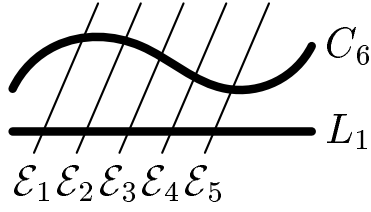}  %
\caption{Exceptional divisors and the fixed locus} 
\label{eg7-1} 
\end{figure}
\medskip

\begin{prop} {\sl Let $X$ be a crepant resolution of the 
quotient threefold $E_2\times S/\iota\times\sigma$ of 
{\rm Example \ref{ex7-1}}. Then $X$ is a Calabi--Yau threefold  
corresponding to the triplet $(6,4,0)$, and its exceptional 
divisors are four copies of the ruled surfaces 
$$(C_6\times\PP^1)\cup(L\times\PP^1).$$

Furthermore, $X$ is of CM type, and has a (quasi-smooth) model 
$$z_0^4+z_1^4=z_2^5+z_3^5+z_4^{10}\subset \PP^4(5,5,4,4,2).$$ 
The Hodge numbers are given by
$$h^{1,1}(X)=15,\,\,h^{2,1}(X)=39.$$}
\end{prop}
\end{ex}

\begin{ex} \label{ex7-3} 
Consider the surface
$$S_0: x_0^2=x_1^3x_3+x_2^7+x_3^{28}\subset\PP^3(14,9,4,1).$$
This is $\#45$ in Yonemua=$\#36$ in Borcea. 
Let $$E_3: y_0^2=y_1^3+y_2^6\subset\PP^2(3,2,1).$$

$\bullet$ The surface $S_0$ is a singular $K3$ surface. There
are two singular points:
$$Q:=(0:1:0:0) \quad\mbox{of type $A_{9,8}$},$$
and
$$R:=(1:0:1:0) \quad\mbox{of type $A_{2,1}$}.$$

$\bullet$ Let $C^{'}$ be the curve on $S_0$ defined by $x_0=0$ 
$$C^{'}=\{x_0=0\} : x_1^3x_3+x_2^7+x_3^{28}=0\subset\PP^2(9,4,1).$$
Then $C^{'}$ is a quasi-smooth curve with singularity $Q$.

$\bullet$ No other curves defined by $x_i=0$ $(i\neq 0)$ are fixed by
the involution $\sigma$.

$\bullet$ Let $S$ be the minimal resolution of $S_0$.
Then $S$ is a smooth $K3$ surface and the involution
$\sigma$ is lifted to $S$. Let $C_6$ be the strict 
transform of $C^{'}$ to $S$; it has genus $6$. 
Let $\EE_1,\cdots, \EE_8$ be exceptional
divisors arising from singularity $Q$.  Let $\EE_9$ be
the exceptional divisor arising from $R$. Then
$\EE_{2i}\, (i=1,2,3,4)$ and $\EE_9$ are fixed by $\sigma$,
but others are not.

Put $L_i:=\EE_{2i}\, (i=1,2,3,4)$ and $L_5:=\EE_9$. Then
$$S^{\sigma}=C_6\cup L_1\cup\cdots\cup L_5.$$
So $g=6$ and $k=5$ (so $r=10, a=0$ in Nikulin's notation). 
The resolution picture is given in Figure \ref{eg7-3}, where the curves 
in boldface are fixed by $\sigma$.  

$\bullet$ The quotient threefold $E_3\times S/\iota\times\sigma$
has singularities along $\{P_i\}\times S^{\sigma}$ where
$E_3^{\iota}=\{P_1,P_2,P_3,P_4\}$.
\medskip

\begin{figure} 
\centering 
\includegraphics[width=5cm]{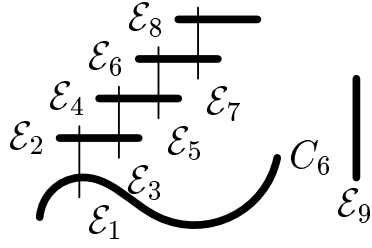}  %
\caption{Exceptional divisors and the fixed locus} 
\label{eg7-3} 
\end{figure} 
\medskip

Summarizing the above, we have

\begin{prop} {\sl A crepant resolution $X$ of the quotient threefold 
$E_3\times S/\iota\times\sigma$ of {\rm Example \ref{ex7-3}} 
is a Calabi--Yau threefold corresponding to the triplet $(10,0,0)$,
and its exceptional divisors are four copies of
$$(C_6\times\PP^1)\cup(L_1\times\PP^1)\cup\cdots\cup(L_5\times\PP^1).$$

Furthermore, $X$ is of CM type, and has a (quasi-smooth) model
$$z_0^3+z_1^6=z_2^3z_4+x_3^7+x_4^{28}\subset\PP^4(28,14,27,12,3).$$
The Hodge numbers are given by
$$h^{1,1}(X)=35,\,\, h^{2,1}(X)=35.$$}
\end{prop}
\end{ex}

\begin{ex} \label{ex7-5} 
Consider the surface
$$S_0: x_0^2=x_1^3+x_2^{10}+x_2^{15}\subset\PP^3(15,10,3,2).$$
This is $\#11$ in Yonemura =$\#18$ in Borcea. Let
$$E_2: y_0^2=y_1^4+y_2^4\subset\PP^2(2,1,1).$$

$\bullet$ $S_0$ is a singular $K3$ surface and singularities are
$$Q_1,Q_2,Q_3=(0:x_1:0:x_3)\quad\mbox{of type $A_{2,1}$,}$$
$$R:=(x_0:x_1:0:0)\quad\mbox{of type $A_{5,4}$,}$$
and
$$T_1,T_2:=(x_0:0:x_2:0)\quad\mbox{of type $A_{3,2}$.}$$

$\bullet$ Let $C^{'}$ be the curve on $S_0$ defined by $x_0=0$:
$$C^{'}=\{x_0=0\}: x_1^3+x_2^{10}+x_3^{15}=0\subset\PP^2(10,3,2).$$
Via the isomorphism $\PP^2(10,3,2)\simeq \PP^2(5,3,1)$, $C^{'}$
is identified with
$$C^{'}: x_1^3+x_2^5+x_3^{15}=0\subset\PP^2(5,3,1).$$

$\bullet$ Let $L^{'}$ be the curve on $S_0$ defined by $x_2=0$:
$$L^{'}=\{x_2=0\}: x_0^2=x_1^3+x_3^{15}\subset\PP^2(15,10,2).$$
Via the isomorphism $\PP^2(15,10,2)\simeq\PP^2(3,1,1)$, $L^{'}$
is identified with
$$L^{'}: x_0=x_1^3+x_3^3\subset\PP^2(3,1,1).$$

$\bullet$ $C^{'}$ has genus $4$ and $L^{'}$ is rational with
$$C^{'}\cap L^{'}=\{Q_1,Q_2, Q_3\}\quad\mbox{and}\quad R\in L^{'}.$$

$\bullet$ Let $S$ be the minimal resolution of $S_0$. The
involution $\sigma$ is lifted to $S$.  Let $C_4$ and $L_1$ be the
strict transforms of $C^{'}$ and $L^{'}$ on $S$,
respectively.  Let
$\EE_i\, (i=1,2,3)$ be the exceptional divisors arising from singularities
$Q_i\, (i=1,2,3)$, $\EE_{4+j}\,(j=0,1,2,3)$ be the
exceptional divisors arising from $R$, and $\EE_{8+t}\,
(t=0,1,2,3)$ be the exceptional divisors arising from singularities
$T_1,T_2$.

$\bullet$ $C, L_1, \EE_5=:L_2$ and $\EE_7=:L_3$ are fixed
by $\sigma$, but all others are not. Hence
$$S^{\sigma}=C_4\cup L_1\cup L_2\cup L_3.$$
So $g=4$ and $k=3$ (so $r=10, a=4$ in Nikulin's notation). 
The resolution picture is given in Figure \ref{eg7-5}, where the curves 
in boldface are fixed by $\sigma$.  

\begin{figure} 
\centering 
\includegraphics[width=7cm]{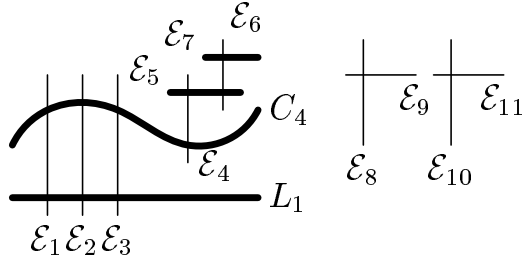}  %
\caption{Exceptional divisors and the fixed locus} 
\label{eg7-5} 
\end{figure} 

$\bullet$ The quotient threefold $E_2\times S/\iota\times\sigma$ has
singularities $\{P_i\}\times S^{\sigma}\,(i=1,2,3,4)$ where
$E_2^{\iota}=\{P_1,P_2,P_3,P_4\}$.

\begin{prop} {\sl A crepant resolution $X$ of the quotient threefold 
$E_2\times S/\iota\times\sigma$ of {\rm Example \ref{ex7-5}} 
is a Calabi--Yau
threefold corresponding to the triplet $(10,4,0)$,
and its exceptional divisors are four copies of ruled
surfaces:
$$(C_4\times\PP^1)\cup(L_1\times\PP^1)\cup(L_2\times\PP^1)
\cup(L_3\times\PP^1).$$

Furthermore, $X$ is of CM type, and has a (quasi-smooth) model
$$z_0^4+z_1^4=z_2^3+z_3^{10}+z_4^{15}\subset\PP^4(15,15,20,6,4).$$
The Hodge numbers are given by
$$h^{1,1}(X)=27,\,\, h^{2,1}(X)=27.$$}
\end{prop}
\end{ex}

\section{Automorphy of Calabi--Yau threefolds 
of Borcea--Voisin type over $\QQ$}\label{sect4}

\subsection{The $L$-series}

Let $X$ be a Calabi--Yau variety defined over $\QQ$ 
of dimension $d$ where $d\leq 3$. 
Hence $X$ is an elliptic curve for $d=1$, a $K3$ surface for $d=2$
and a Calabi--Yau threefold for $d=3$. 
%Let $d$ be its dimension $d=1,2$, or $3$. 

We may assume that $X$ has defining equations with
integer coefficients. A prime $p$ is said to be 
{\it good} if the reduction 
$X_p=X\otimes \FF_p$ is smooth and defines
a Calabi--Yau variety over $\FF_p$. A prime $p$ is said to be {\it bad}
if it is not a good prime.  There are only finitely many bad
primes and we denote by $S$ the product of bad primes. 
Then a Calabi--Yau variety $X$ has an integral model over
$\ZZ[1/S]$.

Put $\bar{X}:=X\otimes_{\QQ}\bar{\QQ}$. We will consider the Galois 
representation associated to the $\ell$-adic \'etale cohomology groups
$H^i_{et}(\bar X, \QQ_{\ell})$ ($0\leq i\leq 2d$) of $X$, 
where $\ell$ is a prime.%, and we put $\bar X:=X\otimes_{\QQ}\bar{\QQ}$.

The absolute Galois group $G_{\QQ}:=\GQ$ acts on $\bar X$.
For each $i,\, 0\leq i\leq 2d$, one has a Galois
representation on the cohomology group $H^i_{et}(\bar X,\QQ_{\ell})$
where $\ell$ is a prime different from $p$. 
This defines a continuous $\ell$-adic representation
$\rho : G_{\QQ}\to GL_r(\QQ_{\ell})$ of some finite rank $r^{\prime}$
where $r^{\prime}=\mbox{dim}_{\QQ_{\ell}} H^i_{et}(\bar X,\QQ_{\ell})=B_i(X)$
(the $i$-th Betti number of $X$). 

For a good prime $p$ the structure of this Galois
representation can be studied by passing to the reduction $X_p=X\otimes
\FF_p$.
The Frobenius morphism $\mbox{Frob}_p$ induces a $\QQ_{\ell}$-linear
map $\rho(\mbox{Frob}_p)$ on $H^i_{et}(\bar X,\QQ_{\ell})$
($i,\, 0\leq i\leq 2d$).  Let
$$P_p^i(X,\rho,t):=\mbox{det}(1-\rho(\mbox{Frob}_p)\,t\,|\, H^i_{et}(\bar
X,\QQ_{\ell}))$$
be the characteristic polynomial of $\rho(\mbox{Frob}_p)$, where $t$ is an
indeterminate.  By the validity of the Weil conjectures,
one knows that
\begin{itemize}
\item 
$P_p^i(X,\rho,t)\in 1+\ZZ[t]$ has degree $B_i(X)$.\\
\item The reciprocal roots of $P_p^i(X,\rho,t)$ are algebraic 
integers with complex absolute value $p^{i/2}$ (the Riemann Hypothesis 
for $X_p$).\\ 
\item The zeta-function of $X_p$ is a rational function of
$t$ over $\QQ$ and is given by
$$\zeta(X_p,t)=\frac{P_p^1(X,\rho,t)P_p^3(X,\rho,t)\cdots P_p^{2d-1}(X,\rho,t)}{P_p^0(X,\rho,t)P_p^2(X,\rho,t)\cdots P_p^{2d}(X,\rho,t)}.$$
\end{itemize}

Now putting all local data together, we can define
the (incomplete) global $L$-series and the (incomplete) 
zeta-function of $X$.
 
\begin{defn}({\rm
For each $i, 0\leq i\leq 2d$, we define the $i$-th (incomplete)
$L$-series by $$L_i(X,s):=L(H^i_{et}(\bar X,\QQ_{\ell}),s)=\prod_{p\not\in S}
\frac{1}{P_p^i(X, \rho, p^{-s})}.$$

The (Hasse--Weil) zeta-function of $X$ is then defined by
$$\zeta(X,s)=\frac{\prod_{i=0}^{d} L_{2i-1}(X,s)}{\prod_{i=1}^{d}
L_{2i}(X,s)}.$$}
%In particular, the middle-dimensional $L$-series 
%$L_d(X,s)$ will simply be denoted by $L(X,s)$, if there is
%no danger of ambiguity.}
\end{defn}

The use of the terminology of ``incomplete'' $L$-series is based
on the fact that it does not include
a few Euler factors corresponding to bad primes. We can also define
Euler factors for primes $p\in S$ to complete the $L$-series
bringing in the {\it Gamma factor} corresponding to the prime
at infinity, and also factors corresponding to bad primes. 

We denote by $\zeta(\QQ,s)=\sum_{n=1}^{\infty}\frac{1}{n^s}=
\prod_{p:prime}\frac{1}{1-p^{-s}}$ the Riemann zeta-function.

\begin{ex} 
We consider an elliptic curve $E$ defined over $\QQ$. 
Let $S$ be a set of bad primes. Then one knows that
the $L$-series of $E$ is given by
$$L_0(E,s)=\zeta(\QQ, s),\,\,L_2(E,s)=\zeta(\QQ, s-1),$$
$$L_1(E,s)=L(H^1(E),s)=\prod_{p\not\in S} P_p^1(E, \rho,p^{-s})^{-1}
=\prod_{p\not\in S}\frac{1}{1-a_pp^{-s}+pp^{-2s}},$$
where $a_p=p+1-\#E(\FF_p)=\mbox{trace}(\rho(\mbox{Frob}_p))$.

The zeta-function of $E$ is then given by
$$\zeta(E,s)=\frac{L_1(E,s)}{\zeta(\QQ,s)\zeta(\QQ,s-1)}.$$

These assertions are true a priori for good primes, but
they can be extended to include bad primes and also
prime at infinity. 

This is a classical result and can be found, for instance, in 
Silverman \cite{Sil94}. 
\end{ex}

\begin{ex}
We consider a $K3$ surface $S$ defined over $\QQ$.
The zeta-function of $S$ is given by
$$\zeta(S,s)=\frac{L_1(S,s)L_3(S,s)}{L_0(S,s)L_2(S,s)L_4(S,s)}
=\frac{1}{L_0(S,s)L_2(S,s)L_4(S,s)}$$
where $L_4(S,s)=L_0(S,s-2)$ by Poincar\'e duality.
%and the zeta-function of $S$ is then given by
%$$\zeta(S,s)=\zeta(\QQ,s)\zeta(\QQ,s-2)L^2(S,s)$$
%where $\zeta(\QQ,s)$ is the Dedekind zeta-function.
The $L_2(S,s)$ factors as a product 
$$L_2(S,s)=L(H^2(S,\Ql), s)=L(NS (S)\otimes\QQ_{\ell},s)L(T(S)\otimes\QQ_{\ell},s)$$
in accordance with the decomposition 
$H^2_{et}(\bar S,\QQ_{\ell})=(NS(S)\oplus T(S))\otimes\Ql$ where
$NS(S)$ is the N\'eron--Severi group spanned by algebraic cycles and 
$T(S)$ is its orthogonal complement, and this decomposition is
Galois invariant. Also this factorization is independent of
the choice of $\ell$. The Tate conjecture
\cite{Ta91}(Theorem 5.6) further asserts that %$NS(S)\simeq \ZZ^{\rho(S)}$
$$NS(S)\otimes_{\QQ}\Ql=H^2_{et}(\bar{S},\Ql)^G$$
where $G_{\QQ}$ denotes the Galois group 
$\mbox{Gal}(\bar{\QQ}/\QQ)$. (We follow the proof in
\cite{Ta91}, due to D. Ramakrishnan. It rests on the facts: (1)
the existence of an abelian variety $A$ and the absolute Hodge cycle
on $S\times A$ inducing an injection $H^2_{et}(S,\Ql)\hookrightarrow
H^2_{et}(A,\Ql)$ (Deligne \cite{D72}), (2) the theorem  
of Faltings that the Tate conjecture is true for $A$, and (3)
the theorem of Lefschetz that rational classes of type $(1,1)$
are algebraic.) Therefore, the Picard number $\rho(S)$ of $S$ 
is equal to the dimension of the $G$-invariant subspace
$H^2_{et}(\bar S,\Ql)^G$.
With the validity of the Tate conjecture
%that all $\rho(S)$ algebraic cycles in $NS(S)$ are defined over $\QQ$, 
the zeta-function of $S$ takes the form
$$\zeta(S,s)=[\zeta(\QQ,s)\zeta(\QQ,s-2)\zeta(\QQ,s-1)^{\rho(S)}
L(T(S)\otimes\QQ_{\ell},s)]^{-1}.$$ 

Now $NS(\bar{S})\neq NS(S)$ in general. In that case,  
not all algebraic cycles in $NS(\bar{S})$ are defined over $\QQ$,
let $\LL$ be the smallest algebraic number field over which
all $\rho(S)$ algebraic cycles are defined. Let $\zeta(\LL,s)$
denote the Dedekind zeta-function of $\LL$, that is,
$$\zeta(\LL,s)=\sum_{I\subseteq{\mathcal{O}}_{\LL}} \frac{1}{N_{\LL/\QQ}(I)^s}
=\prod_{P\subseteq{\mathcal{O}_{\LL}}}\frac{1}{1-N_{\LL/\QQ}(P)^{-s}}$$
where $\mathcal{O}_{\LL}$ is the ring of integers of $\LL$,
$I$ (resp. $P$) is an ideal (resp. prime ideal) of $\mathcal{O}_{\LL}$,
and $N(I)$ (resp. $N(P)$) denotes the norm.
Then $\zeta(\LL,s)$ is a product over the Artin $L$-functions of the
irreducible complex representations of the Galois group, whereas only
some need to occur in $NS(\bar{S})$ and the multiplicity of an 
irreducible representation in $NS(\bar{S})_{\CC}$ depends on the
geometry.  
Then the zeta-function of $S$ is of the form
$$\zeta(S,s)=[\zeta(\QQ,s)\zeta(\QQ,s-2)\zeta(\LL,s-1)^{t}
L(\rho, s)L(T(S)\otimes\QQ_{\ell},s)]^{-1},$$
where 
the exponent $t$ is some integer $1\leq t\leq\rho(\bar{S})$, which
is rather difficult to determine explicitly, and 
$L(\rho, s)$ is the Artin $L$-series of the irreducible complex 
representation.  %of dimension $\rho(\bar{S})-t$. 
(In general, the automorphy of the Artin $L$-function is still an open problem.) 

For example, let $S$ be a $K3$ surface with $NS(\bar{S})_{\QQ}\cong\QQ^2$
so $\rho(\bar{S})=2$. Let $\LL$ be a quadratic extension of $\QQ$ so
that $NS(S_{\LL})_{\QQ}\cong \QQ$, and the Galois group acts trivially on it,
but acts by a non-trivial character on a complementary one-dimensional
subspace. Then 
$$L(NS(\bar{S}),s)=\zeta(\LL, s-1)=\zeta(\QQ,s-1) L(\LL,s-1)$$
where $L(\LL,s)$ is the Dirichlet $L$-function of $\LL$.% and hence
%no exponent $2=\rho(\bar{S})$.
\end{ex} 

\begin{ex}
We now consider a Calabi--Yau threefold $X$ over $\QQ$.
The zeta-function of $X$ is given by
$$\zeta(X,s)=\frac{L_1(X,s)L_3(X,s)L_5(X,s)}{L_0(X,s)L_2(X,s)L_4(X,s)L_6(X,s)}$$
$$=\frac{L_3(X,s)}{L_0(X,s)L_2(X,s)L_4(X,s)L_6(X,s)}$$
where $L_6(X,s)=L_0(X,s-3),\, L_4(X,s)=L_2(X,s-1)$ by Poincar\'e duality.
The zeta-function of $X$ is of the form
$$\zeta(X,s)=\frac{L_3(X,s)}{\zeta(\QQ,s)\zeta(\QQ,s-3)L_2(X,s)L_2(X,s-1)}
.$$
\end{ex}

\begin{conj}(Langlands reciprocity conjecture \cite{L78}) {\sl Let $X$ be a
Calabi--Yau variety defined over $\QQ$. Then the zeta-function $\zeta(X,s)$ is 
automorphic.}
\end{conj}

\begin{rem} For our Calabi--Yau varieties over $\QQ$,
we know the form of their zeta-functions, and the Riemann 
zeta-function $\zeta(s)=\zeta(\QQ,s)$ (and its translates) 
are trivially automorphic as they correspond to the identity 
representation. The automorphy question for our $K3$ surfaces
and our Calabi--Yau varieties over $\QQ$ is then for the automorphy 
of the $L$-series $L_i(X,s)$ for each $i$, $0\leq i\leq \mbox{dim}(X)$. 

Are there any automorphic forms 
(representations) such that $L_i(X,s)$ for each $i$ ($0\leq i\leq 
\mbox{dim}(X)$) are determined by the $L$-series of such automorphic
objects?  
\end{rem}

First we will discuss some examples in support of 
the Langlands reciprocity conjecture. 

\subsection{Elliptic curves over $\QQ$}

For dimension $1$ Calabi--Yau varieties over $\QQ$, we have
the well-known celebrated results of Wiles \cite{Wiles95},
Taylor--Wiles \cite{TW95}.

\begin{thm} {\sl Let $E$ be an elliptic curve defined over $\QQ$. Then
there exists a normalized weight $2$ new form $f_E$ of level $N_E$
which is an eigenvector of the Hecke operators, such that
$$L_1(E,s)=L(f_E,s).$$
Here $N_E$ is the conductor of $E$ and $f_E$ has the
$q$-expansion ($q=e^{2\pi i z},\,z=x+iy$ with $y>0$)
$$f_E=q+a_2q^2+\cdots+a_pq^p+\cdots$$
where $a_p$ is the same as defined in Example 5.1.

In terms of Galois representations, let 
$\rho_{E,\ell}$ be an $\ell$-adic
representation of $\GQ$ on the $\ell$-adic
Tate module $T_{\ell}(E)$ of $E/\QQ$. Then $\rho_{E,\ell}$ is
modular for some $\ell$. That is, there exists
a cusp form $f_E$ and a representation $\rho_{f_E}$ of $\GQ$ 
such that $\rho_{E,\ell}=\rho_{f_E}$.
We will denote this representation simply by $\rho_E$.}
\end{thm}

\begin{rem}
If $E$ is an elliptic curve over $\QQ$ with CM by
an imaginary quadratic field $K/\QQ$, then $L_1(E,s)$
is equal to a Hecke $L$-series of $K$ with a
suitable Grossencharacter of $K$. (Deuring \cite{D53}.)
\end{rem}

\subsection{$K3$ surfaces over $\QQ$ of CM type}

For dimension $2$ Calabi--Yau varieties, namely, $K3$ surfaces,
our results on automorphy is formulated as follows.
We can establish the automorphy of $K3$ surfaces over $\QQ$ of CM type.
This generalizes the result of Shioda and Inose \cite{SI77}
for singular $K3$ surfaces, and also the results of Livn\'e--Sch\"utt--Yui
\cite{LSY10} for certain $K3$ surfaces with non-symplectic group actions.

\begin{thm} \label{thm4.5}
{\sl Let $(S,\sigma)$ be one of the $86$ pairs $(S,\sigma)$ of $K3$
surfaces in {\em Theorem \ref{thm-DelsarteS}}. Then $(S,\sigma)$ 
is defined over $\QQ$, and
there exists a quadruple $(\rho, \KK, \iota, \chi)$ with the
following properties:
\begin{enumerate}
\item $\rho$ is an (Artin) Galois representation of $\GQ$, and 
the degree of $\rho$ is $\rho(\bar S)$ (the geometric Picard number of $S$),
\item $\KK$ is a CM abelian extension of $\QQ$,  
\item $\iota:\KK\ra\CC$ is an embedding,
\item $\chi$ is a Hecke character of $\KK$ of $\infty$ type
$z\ra \iota(z)^2$,
\item $\dim{\rho}+[\KK:\QQ]=22$ where 
$[\KK:\QQ]=\dim\, T(\bar S)_{\QQ}$.
\end{enumerate}
such that the zeta-function $\zeta(S,s)$ and the 
$L_2(S,s)$ of $S/\QQ$ are given by
\[ \zeta(S/\QQ,s)=[\zeta(\QQ, s)\zeta(\QQ, s-2)L_2(S,s)]^{-1} \]
where
\[L_2(S,s)=L(\rho,s-1)L(\chi,s). \] 
}
\end{thm}

\begin{proof}
We follow an argument similar to the one in \cite{LSY10},
to which we refer for
further details. Since $S$ is a surface defined over $\QQ$, its
$\Ql$-cohomology is $1$-dimensional in dimensions $1$ and $4$,
which contribute the factors $\zeta(s)$ and $\zeta(s-2)$
respectively.  Since $S$ is a $K3$ surface the first and
the third cohomology groups vanish, giving no contribution
to the $L$-function. The second cohomology group is a direct
sum $N_S\oplus T_S$ of the algebraic part, spanned by
the subspace $N_S$ of algebraic cycles and its orthogonal
complement $T_S$, called the space of transcendental cycles.
This direct sum decomposition is Galois invariant. Since
$N_S$ is the $\Ql$-span of the image by the cycle map of 
the N\'eron-Severi group $NS(\bar S)$ (with scalars extended to $\Ql$), 
the Galois group acts on $NS(\bar S)$ through a finite quotient. Hence
it acts on $N_S$ by the Tate twist $\rho(1)$ of the
corresponding Artin representation $\rho$.
                                                              
To obtain the last factor we first consider the cohomology
with complex coefficients. Since the defining equation for
$S$ uses only $4$ monomials, it is a Delsarte surface. Hence
it is a quotient of a
surface in $\PP^3$ with a (homogeneous) diagonal equation
$\sum_{i=1,\dots,4}a_i w_i^r=0$ by some diagonal action of
roots of unity. Moreover in our case the monomials in
the (diagonal) equation for $S$ have coefficients $1$, which
implies that the $a_i$'s can also be taken to be all
$1$\/'s. Weil's calculation (see \cite{LSY10}, Section 6)
gives that over an appropriate cyclotomic field the
Galois representation on the transcendental cycles is
a sum of $1$-dimensional representations coming from Jacobi
sums, of infinity type as in the statement of Theorem \ref{thm4.5},
which the absolute Galois group permutes transitively.
The Theorem \label{thm4.5} follows.  (For detailed discussion on
Jacobi sums, the reader is referred to Appendex in the section $7$
below, or Gouv\^ea and Yui \cite{GY95}.)
\end{proof}

\begin{cor} {\sl
Let $(S,\sigma)$ be as in Theorem 5.6. Then 
$S$ has CM by a cyclotomic field $\KK=\QQ(\zeta_t)$ for some $t$.
(Here $\zeta_t$ denotes a primitive $t$-th root of unity.)} 
\end{cor} 

\begin{proof} This follows from Theorem 5.6. $S$ is realized by a 
finite quotient of some 
Fermat surface of some degree, however this finite group may
have a rather large order and it requires more work to determine
its precise form.  The field $\KK$ that corresponds to
the transcendental cycles is isomorphic to $T(S)\otimes\QQ\simeq \QQ(\zeta_t)$
for some $t$. Moreover, it is generated by Jacobi sum Grossencharacters
of $\QQ(\zeta_t)$. This is because the Galois representation
defined by $T(S)$ is a sum of $1$-dimensional representations
induced from the Jacobi sum Grossencharacters corresponding
to the unique character $\ba$ with $\Vert\ba\Vert=0$.
(See \cite{LSY10}).   
\end{proof}

In general, the automorphy of the Artin $L$-function is still a
conjecture. However, in our cases, we have the following
result. 

\begin{cor} {\sl Let $(S,\sigma)$ be as in Theorem 5.6. 
Then the Artin $L$-function $L(\rho,s)$ is automorphic.}
\end{cor}

\begin{proof}
We know that $S$ is dominated by some Fermat surface 
$${\mathcal{F}}_m : x_0^m+x_1^m+x_2^m+x_3^m=0\subset\PP^3$$ 
of degree $m$. Here we
review Shioda's treatment (cf. Shioda \cite{Sh86}, or Gouv\^ea--Yui
\cite{GY95}). The
cohomology group $H^2({\mathcal{F}}_m,\Ql)$ is the direct sum of
one-dimensional spaces. More precisely, let
$$H^2({\mathcal{F}}_m,\Ql)=\oplus_{\alpha\in\{0\}\cup\fA_m} 
V(\alpha),\qquad\mbox{dim}\, V(\alpha)=1$$
where
$$\fA_m:=\{\ba=(a_0,a_1,a_2,a_3)\in (\ZZ/m\ZZ)^4\,|\, 
a_i\neq 0, \sum_{i=0}^3 a_i=0\in\ZZ/m\ZZ\}.$$
This implies that the N\'eron--Severi group $NS({\mathcal{F}}_m)$ and the
group of transcendental cycles $T({\mathcal{F}}_m)$ of ${\mathcal{F}}_m$ are
also described as direct sums of one-dimensional spaces:
$$NS({\mathcal{F}}_m)\otimes \Ql
=\oplus_{\alpha\in\{0\}\cup\fB_m} V(\alpha)$$
and
$$T({\mathcal{F}}_m)\otimes \Ql
=\oplus_{\alpha\in\fC_m} V(\alpha)$$
where
$$\fB_m:=\{\ba=(a_0,a_1,a_2,a_3)\in \fA_m\,|\,
\sum_{i=0}^3\langle{\frac{ta_i}{m}}\rangle=2\quad\mbox{for all $t$ such that
$(t,m)=1$}\},$$
and $$\fC_m:=\fA_m\setminus\fB_m.$$

Since $S$ is realized as a Fermat quotient of some Fermat surface 
${\mathcal{F}}_m$ by a finite group, say, $H$, the irreducible Galois 
representation $\rho$ is also induced from one-dimensional subspaces 
belonging to $NS({\mathcal{F}}_m)$, which are invariant under the action 
of $H$, and hence $L(\rho,s)$ is automorphic. (Indeed, working through all 
the examples, we are able to show that Artin $L$-functions are indeed 
automorphic.) 
\end{proof}

\subsection{Calabi--Yau threefolds over $\QQ$ of Borcea--Voisin
type}

For dimension $3$ Calabi--Yau varieties over $\QQ$ of Borcea--Voisin
type, our automorphy results are formulated in the following
theorems. 

\begin{thm} {\sl  Let $(S, \sigma)$ be one of the $86$ pairs of 
$K3$ surfaces with involution given in {\em Theorem 
\ref{thm-DelsarteS}}.  Let $E$ be an elliptic curve over $\QQ$ 
with involution $\iota$.  
Let $X$ be a crepant resolution of the quotient threefold 
$E\times S/\iota\times\sigma$ with a model defined over
$\QQ$.  Then $X$ is automorphic.}
\end{thm}

We reformulate the above assertion in more concrete fashion as follows.

\begin{thm}
{\sl Let $(S, \sigma)$ be (the minimal resolution of) one of the $86$ 
$K3$ surfaces with
involution $\sigma$ listed in {\em Theorem \ref{thm-DelsarteS}}. Then
$S$ is defined over $\QQ$, and is of Delsarte type, 
and hence $S$ is of CM type.
Let $E$ be an elliptic curve $E_2$ or $E_3$ with
involution $\iota$ (or any elliptic curve with complex
multiplication). Consider the quotient threefold
$E\times S/\iota\times\sigma$, and let $X$ be its
crepant resolution.  Then $X$ is a Calabi--Yau threefold and
has a model defined over $\QQ$.

Furthermore, the following assertions hold: 
\begin{itemize}
\item  $X$ is of CM type,
\item The $L$-series $L_2(X,s)$ and $L_3(X,s)$ are
automorphic. 
\item  The zeta-function $\zeta(X,s)$ is automorphic, and hence
$X$ is automorphic.
\end{itemize}}
\end{thm}

Since all elliptic curves $E$ defined over $\QQ$ are modular,
without the assumption that $E$ is of CM type, we have the 
following more general results.

\begin{thm}
{\sl Let $(S,\sigma)$ be one of the $86$ pairs of $K3$ surfaces with
involution $\sigma$ defined over $\QQ$ in {\em Theorem \ref{thm-DelsarteS}}. 
Let $(E,\iota)$
be an elliptic curve defined over $\QQ$. Let $X$ be a crepant resolution
of the quotient threefold $E\times S/\iota\times\sigma$, which has 
a model defined over $\QQ$. 

Then the following assertions hold:
\begin{itemize}
\item[(a)] Let $E$ be an elliptic curve $\QQ$.
Then there is the automorphic representation $\rho_E$. Equivalently, 
there is a cusp form $f_E$ of weight $2$ associated to $\rho_E$.
%obtained from the $2$-dimensional automorphic
%representation in $\fA(GL_2(\bA_{\QQ}))$.
\item[(b)] Take $S$ to be a $K3$ surface of CM type (cf. {\em Theorem \ref{thm2.5}}). 
%Let $S$ be one of the $K3$ surface $S$ of CM type in Theorem 2.5. Then 
There is an Artin representation
$\rho$ of an algebraic extension $\KK$ over $\QQ$ where we put $m:=[\KK:\QQ]$. 
Put $G_{\QQ}=\GQ$ and $G_{\KK}=\GK$.
Let $\rho_{G_{\KK}}$ be the compatible system of $1$-dimensional
$\ell$-adic representatopn of $G_{\KK}$.  Then the $m$-dimensional 
Galois representation
associated to the group of transcendental cycles $T(S)^{\sigma}$ 
is given by $\mbox{ind}_{G_{\KK}}^{G_{\QQ}}\rho_{G_{\KK}}.$
We denote by $f_{T(S)}$ the ``fictitious'' automorphic form
associated to this representation (though we cannot
write it down explicitly).
%is the submotive $T(S)$ of $H^2(S,\Ql)$.  of a $K3$ surfrace $S$
%of CM type, there is the automorphic form obtained from the
%automorphic induction $I(\chi)$ in
%$\fA(GL_{d}(\bA_{\QQ}))$ for $d=22-\rho(S)\in\NN$.
\item[(c)] %The Calabi--Yau threefold $X$ is given by 
%$\MM_X=E\otimes \MM_S$ and is defined over $\QQ$. 
Let $X$ be a Calabi--Yau threefold over $\QQ$ of Borcea--Voisin
type. Then there is the $2m$-dimensional Galois representation
$$\pi:=\rho_E\otimes \mbox{ind}_{G_{\KK}}^{G_{\QQ}}\rho_{G_{\KK}}
\simeq \mbox{ind}_{G_{\KK}}^{G_{\QQ}}(\rho_{G_{\KK}}\otimes 
\mbox{Res}_{G_{\KK}}^{G_{\QQ}} \rho_E).$$
%\otimesAssociated to $H^3(X,\Ql)$ of a Calabi--Yau threefold
%of Borcea--Voisin type, there is the automorphic
%form $f_E\otimes I(\chi)$ in $\fA(GL_{2r}(\bA_{\QQ)}))$.
$\pi$ is an automorphic cuspidal irreducible
representation of $GL(2,\KK)$, and
$$L(\pi,s)=L(\mbox{ind}_{G_{\KK}}^{G_{\QQ}} \pi, s)
=L(\pi_E\otimes \mbox{ind}_{\GK}^{\GQ}\rho,s)$$
$$=L(f_E\otimes f_{T(S)},s).$$
\item[(d)] The Galois representation associated to the twisted sectors
(the exceptional divisors) is automorphic.
\end{itemize}

Therefore, $X$ is automorphic. }
\end{thm}

\subsection{Preparation for proof of automorphy for Calabi--Yau
threefolds of Borcea--Voisin type}

We need to compute the cohomology groups
of our Calabi--Yau threefolds of Borcea--Voisin type. 
For this, first we compute the cohomology groups of the
product $E\times S$.

\begin{lem} {\sl The K\"unneth formula for the
product $E\times S$ gives:
$$H^i(E\times S,\Ql)=\oplus_{p+q=i} H^p(E,\Ql)\otimes H^q(S,\Ql)$$
for $0\leq i\leq 6$.
Then for each $i,\,0\leq i\leq 3$, we obtain

\begin{itemize}
\item $H^0(E\times S,\Ql)=\Ql.$ \\
\item $H^1(E\times S,\Ql)=H^1(E,\Ql)\otimes \Ql.$ \\ 
\item $H^2(E\times S,\Ql)=\Ql\otimes H^2(S,\Ql)\oplus \Ql\otimes\Ql.$ \\ 
\item $H^3(E\times S,\Ql)=H^1(E,\Ql)\otimes H^2(S,\Ql).$\\
\end{itemize}

The higher cohomologies for $i=4,5,6$ can be determined by 
Poincar\'e duality.} 
\end{lem}

\begin{proof} These follow from the definition of $E$ and $S$
and the K\"unneth formula. In fact, we have
$$H^0(E\times S,\Ql)=H^0(E,\Ql)\otimes H^0(S,\Ql)=\Ql.$$
$$H^1(E\times S,\Ql)=H^1(E,\Ql)\otimes H^0(S,\Ql)\oplus
H^0(E,\Ql)\otimes H^1(S,\Ql)$$
$$=H^1(E,\Ql)\otimes \Ql.$$
$$H^2(E\times S,\Ql)=H^0(E,\Ql)\otimes H^2(S,\Ql)\oplus
H^1(E,\Ql)\otimes H^1(S,\Ql)$$
$$\oplus H^2(E,\Ql)\otimes H^0(S,\Ql)
=\Ql\otimes H^2(S,\Ql)\oplus \Ql\otimes\Ql.$$
$$H^3(E\times S,\Ql)=H^0(E,\Ql)\otimes H^3(S,\Ql)\oplus
H^1(E,\Ql)\otimes H^2(S,\Ql)$$
$$\oplus H^2(E,\Ql)\otimes H^1(S,\Ql)
\oplus H^3(E,\Ql)\otimes H^0(S,\Ql)$$
$$=H^1(E,\Ql)\otimes H^2(S,\Ql).$$
\end{proof}

\medskip

In order to determine the cohomology groups for
the Calabi--Yau threefolds $E\times S/\iota\times\sigma$
of Borcea--Voisin type, we need to compute 
the cohomology groups of ``non-twisted'' sector and the 
cohomology groups arising from singularities of ``twisted'' sector.  
For this we will use the orbifold Dolbeault cohomology theory developed by 
Chen and Ruan \cite{CR04}. 
%cohomology groups. For orbifold cohomology theory and
%calculations of orbifold cohomologies, the reader is
%referred to, e.g., Chen and Ruan \cite{CR04}. 
We will give a brief description of orbifold cohomology formulas 
relevant to our calculations.  

\begin{defn}
{\rm 
Let $X_0=E\times S/\iota\times\sigma$ be a singular Calabi--Yau
threefold of Borcea--Voisin type.  Then the cohomology group of 
$X_0$ is given by
$$H_{\mbox{orb}}^{p,q}({\bar X}_0)=H^{p,q}(E\times S)
\oplus%_{h=\iota\times\sigma, \,h\neq 1}
H^{p-1,q-1}((E\times S)^{\iota\times\sigma})$$
for $0\leq p,q\leq 3$ with $p+q=3$, where $\bar{X}_0=X_0\otimes\CC$.

The twisted sectors are the cohomology groups that correspond to $h\neq 1$ 
(the second term), and the non-twisted sector corresponds to $h=1$ 
(the first term).}
\end{defn}

First we calculate the non-twisted sector of the cohomology.

\begin{lem} {\sl Let $X_0:=E\times S/\iota\times\sigma$ be a singular
Calabi--Yau threefold over $\CC$ of Borcea--Viosin type. 
Then we have

\begin{itemize}
\item $H^{1,0}(X_0)=H^{1,0}(E)^{\iota}\otimes \CC=0.$ \\
\item $H^{2,0}(X_0)=\CC\otimes H^{2,0}(S)^{\sigma}=0.$ \\
\item $H^{3,0}(X_0)=\CC.$ \\
\item $H^{1,1}(X_0)=\CC\otimes H^{1,1}(S)^{\sigma=1}\oplus \CC. $\\
\item $H^{2,1}(X_0)=\CC\oplus H^{1,1}(S)^{\sigma=-1}\otimes\CC$. \\
\end{itemize}

Therefore, the Hodge numbers of the singular
Calabi--Yau threefold $X_0$ are given by
$$h^1(X_0)=0,\, h^{2,0}(X_0)=0,\, h^{3,0}(X_0)=1$$
and
$$h^{1,1}(X_0)=1+r,\, h^{2,1}(X_0)=1+(20-r),$$
where $r=\mbox{rank} NS(S)^{\sigma}$.}
\end{lem}

\begin{proof}
Let $\omega_E$ and $\omega_S$ be the non-trivial holomorphic
$1$-form and $2$-form on $E$ and $S$, respectively.
Then $\omega_E\wedge\omega_S$ descends to a holomorphic
$3$-form on $X$. Now the involution
$\iota: E\to E$ acts on $\omega_E$ non-symplectically,
and the involution $\sigma: S\to S$ acts
on $\omega_S$ non-symplectically.  
This gives
$$H^{0}(X_0,\Omega_{X_0}^3)=\CC$$
so that $h^{3,0}(X_0)=1$, indeed. $H^{3,0}(X_0)$ is spanned
by $\omega_E\times \omega_S$.
Then the K\"unneth formula gives that
$$H^{0}(X_0,\Omega_{X_0}^1)=H^{0}(E,\Omega_{E}^1)^{\iota}\otimes \CC=0$$
so that $h^1(X_0)=0.$  Also we have 
$$H^{0}(X_0,\Omega_{X_0}^2)=\CC\otimes H^{0}(S,\Omega_S^2)^{\sigma}=0$$
so that $h^{2,0}(X_0)=0.$
Hence $X_0$ is indeed a (singular) Calabi--Yau threefold.
Now we compute $H^{1,1}(X_0)=H^2(X_0)$.
$$H^{1,1}(X_0)=H^{0,0}(E)^{\iota}\otimes H^{1,1}(S)^{\sigma=1}
\oplus H^{1,1}(E)^{\iota}\otimes H^{0,0}(S)^{\sigma}
=\CC\otimes H^{1,1}(S)^{\sigma=1}\oplus \CC$$
so that $h^{1,1}(X_0)=1+r$.  
Now note that
$$H^{3,0}(X_0)=(H^{1,0}(E)\otimes 
H^{2,0}(S)^{\iota\times\sigma}),$$
so that $h^{3,0}(X_0)=1$. 
By the Kunneth formula, we get
%$$=(H^{1,0}(E,\CC)\otimes T(S)\otimes\CC)^{\iota\times\sigma},$$
%and that
$$H^{2,1}(X_0)=(H^{1,0}(E)\otimes H^{1,1}(S)
\oplus H^{0,1}(E)\otimes H^{2,0}(S))^{\iota\times\sigma}.$$
Let $H^{1,1}(S,\CC)^{\sigma=-1}$ denote the $-1$ eigenspace for
the action of $\sigma$ on $H^{1,1}(S)$. Since $\iota$ 
induces
$-1$ on $H^1(E)$ and $\sigma$ acts by $-1$ on $H^{2,0}(S)$,
this gives that
$$H^{2,1}(X)=\CC\otimes H^2(S)^{\sigma=-1}\oplus \CC$$
%=\CC\otimes T(S)^{\sigma=-1}\oplus\CC,$$
and hence we have $h^{2,1}(X_0)=1+(20-r)$.
\end{proof}

Now we pass onto a smooth resolution $X$ of $X_0$. We need to
calculate the cohomology groups of $X$, in particular, the
twisted sectors.

\begin{lem} 
{\sl Let $X=\widetilde{E\times S/\iota\times\sigma}$ be a smooth Calabi--Yau 
orbifold over $\CC$ of Borcea--Voisin type. Let $S^{\sigma}$ be the 
fixed locus of $S$ of $\sigma$.
Then the twisted sectors consist of $4$ copies of $S^{\sigma}$, and
$$h^{0,0}(S^{\sigma})=k+1,\,\,h^{1,0}(S^{\sigma})=g.$$
Therefore,
$$h^{1,0}(X)=h^{2,0}(X)=0,\,h^{3,0}(X)=1,$$
$$h^{1,1}(X)=1+r+4(k+1),\, h^{2,1}(X)=1+(20-r)+4g.$$} 
\end{lem}

\begin{proof} The Hodge numbers $h^{1,0}, h^{2,0}$ and $h^{3,0}$ of $X$ 
are the same as those for the singular $X_0$. For $h^{1,1}(X)$ and $h^{2,1}(X)$
we need to bring in resolutions of singularities.
The twisted sectors consist of $4$ copies of
$S^{\sigma}$. We know that $S^{\sigma}=C_g\cup L_1\cup\cdots\cup L_k$
for $(r,a,\delta)\neq (10,10,0),(10,8,0)$ and $S^{\sigma}=C_1\cup \tilde C_1$
for $(r,a,\delta)=(10,8,0)$. 
Therefore,
$h^{1,1}(X)=1+r+4(k+1)$.
For $h^{2,1}(X)$, the contribution from the twisted sectors is
the $4$ copies of $\PP^1\times C_g$, and hence $h^{2,1}(X)=1+(20-r)+4g.$ 
\end{proof}

\begin{rem}
Voisin \cite{V93} gave more geometrical computations for
the Hodge numbers $h^{1,1}(X)$ and $h^{2,1}(X)$. We will recall
briefly her calculations. Recall that the fixed locus
of $\iota$ on $E$ consists of four points $\{P_i,\, i=1,\cdots,4\}$,
and that the fixed
locus of $S$ under the action of $\sigma$ is
$S^{\sigma}=C_g\cup L_1\cup\cdots\cup L_k$
where $C_g$ is a genus $g$ curve and $L_i\,(i=1,\cdots,k)$ are
rational curves. Let $N$ be the number of components in $S^{\sigma}$,
that is, $N=k+1$, and let $N^{\prime}$ be the sum of genera of the
components, that is, $N^{\prime}=g$. The fixed point locus of the 
action $\iota\times\sigma$ on $E\times S$ consists of $4N$ curves 
$\{P_i\}\times C_g,\, \{P_i\}\times L_j$. We blow up $E\times S$
along these $4N$ curves to obtain a smooth Calabi--Yau threefold
$X$ with exceptional divisors arising from the $4N$ curves.

Now compute $h^{1,1}(X)$. First the exceptional divisors 
give
$4N$ classes in $H^{1,1}(X)$. From the quotient 
surface $S/\sigma$
we get $h^0=1=h^4,\,h^1=h^3=0,\, h^{2,0}=h^{0,2}=0$ and
$h^{1,1}=10+N-N^{\prime}$. Then by the Kunneth formula,
$$H^{1,1}(X)=\CC-\mbox{span of $4N$ exceptional divisors}\oplus
H^{1,1}(S/\sigma)\oplus H^{1,1}(E).$$ 
Hence
$$h^{1,1}(X)=4N+10+N-N^{\prime}+1=11-5N-N^{\prime}.$$

For $h^{2,1}(X)$, we first note that the $4N$ curves give
rise to the classes $H^{1,0}(C_g)\oplus_{j=1}^k H^{1,0}(L_j)$
in $H^{2,1}(X)$. Next, let $H^2(S)^{-}$ denote the $-1$ eigenspace
for the action of $\sigma$ on $H^2(S)$. Then again by the
Kunneth formula, we obtain
$$H^{2,1}(X)\simeq H^{1,0}(C_g)\oplus_{j=1}^k H^{1,0}(L_j)\oplus
H^{1,1}(S)^{-1}\oplus H^{2,0}(S)$$
and this shows that
$$h^{2,1}(X)=4N^{\prime}+h^{1,1}(S)^{-}+1=4N^{\prime}+10+N^{\prime}-N+1
=11+5N^{\prime}-N.$$
\end{rem}

We now compute the $L$-series of our Calabi--Yau
threefolds of Borcea--Voisin type. 

\begin{thm}
{\sl Let $X$ be a Calabi--Yau threefold of Borcea--Voisin type,
$X=\widetilde{E\times S/\iota\times\sigma}$.
The Betti numbers of $X$ are given by
$$B_0(X)=1, B_1(X)=1, B_2(X)=h^{1,1}(X)=1+r+4(k+1),$$
$$ B_3(X)=2(1+h^{2,1}(X))=2(1+(20-r)+4g).$$

The $\ell$-adic \'etale cohomological $L$-series $L_i(X,s)$ ($0\leq i\leq 6$) 
can be computed as follows: 

\begin{itemize}
\item $L_0(X,s)=\zeta(\QQ, s).$ \\
\item $L_1(X,s)=1.$ \\
\item $L_2(X,s)=\zeta(\QQ, s-1)^{h^{1,1}(X)}$
provided that all algebraic cycles in $NS(S)^{\sigma}$
%and all the $k$ rational curves $L_i$ in $S^{\sigma}$ are
are defined over $\QQ$.
%where $t$ denotes the number of algebraic cycles in $NS({S})$
%defined over $\QQ$ and $L(\rho^{\prime},s)$ is the Artin 
%$L$-function of dimension $\rho(NS(\bar{S})-t$. 
%that are defined over $\QQ$ and  provided that all
%algebraic cycles in $NS(S)^{\sigma}$ 
%and all rational curves $L_i$
%($i=1,\cdots,k$) (which are the rational curves
%in $S^{\sigma}$) are defined over $\QQ$.
Otherwise, let $t<r$ be the number of algebraic
cycles in $NS(S)^{\sigma}$ that are defined over $\QQ$
and let $F$ be the smallest field of definition for all
$\rho(\bar{S})-t$ algebraic cycles in $NS(\bar{S})^{\sigma}
\setminus NS(S)^{\sigma}$. Also suppose that all $4$ points 
in $E^{\iota}$ are defined over $\QQ$.  Then
$L_2(X,s)=\zeta(\QQ, s-1)^{1+t+4(k+1)} L(\rho^{\prime},s)$, 
where $\rho^{\prime}$ is an irreducible representation
of dimension $r-t$ and $L(\rho^{\prime},s)$ is its
Artin $L$-function.

(Without knowing the field of definitions of algebraic cycles and 
$4$ fixed points in $E^{\iota}$ explicitly, it is very difficult to 
write down an explicit formula for the $L$-function $L_2(X,s)$.)\\
\item $L_3(X,s)%=L(E\otimes T(S)^{\sigma},s)L(E\otimes NS(S)^{\sigma}, s)
%L((E\times C_g,s))$$
=L(E\otimes\chi, s)L(E\otimes\rho,s)L(J(C_g),s-1)^4.$ \\
\end{itemize} 

The higher cohomologies are determined by Poincar\'e duality.}
\end{thm}

\begin{proof} For the calculation of $L$-series, we ought to
pass onto \'etale cohomology groups.  
Obivously, $$L_0(X,s)=\zeta(\QQ,s),\quad\mbox{and}\quad L_1(X,s)=1.$$ 
For $L_2(X,s)$, note that
$$h^{1,1}(X)=1+r+4(k+1).$$
So if all the $r$ algebraic cycles in 
$NS(\bar S)^{\sigma}$% and the
%rational curves in $S^{\sigma}$ 
are defined over $\QQ$, then we have 
$$L_2(X,s)=L_2(H^2(X,\Ql),s)=\zeta(\QQ,s-1)^{h^{1,1}(X)}.$$
Otherwise, $t$ algebraic cycles are defined over $\QQ$
so the Galois group acts trivially on $1+t+4(k+1)$
algebraic cycles so that the exponent is $1+t+4(k+1)$.
But the Galois group acts non-trivially on the $r-t$-dimensional
subspace of algebraic cycles and this gives rise to an 
irreducible representation of dimension $r-t$ and
the Artin $L$-function.
Therefore,
$$L_2(X,s)=\zeta(\QQ,s-1)^{1+t+4(k+1)} L(\rho^{\prime},s).$$ 

Finally for $L_3(X,s)$, note that
$$H^3(X,\Ql)=H^1(E,\Ql)\otimes H^2(S,\Ql).$$
Under the action of $\iota$, $H^1(E,\Ql)$ is the direct sum of
two eigenspaces: $$H^1(E,\Ql)=H^1(E,\Ql)^{\iota=1}
\oplus H^1(E,\Ql)^{\iota=-1}
=H^1(E,\Ql)^{\iota=-1}.$$ 
Similarly, under the action of $\sigma$, $H^2(S,\Ql)$ is the direct sum 
of two eigenspaces:
$$H^2(S,\Ql)=H^2(S,\Ql)^{\sigma=1}\oplus H^2(S,\Ql)^{\sigma=-1}$$
%=(NS(\bar{S})^{\sigma=1}\otimes \Ql) NS(\bar{S})^{\sigma=-1}\otimes\Ql)$$
%$$ \oplus (T(S)^{\sigma=-1}\otimes T(S)^{\sigma=1}\otimes \Ql)
%=((NS(S)^{\sigma=1}\oplus NS(S)^{\sigma=-1})\otimes\Ql) \oplus 
%T(S)^{\sigma=-1}\otimes\Ql.$$
$$=(H^{1,1}(\bar{S})^{\sigma=1}\otimes\Ql)\oplus(H^{1,1}(\bar{S})^{\sigma=-1}
\otimes\Ql)=(NS(S)^{\sigma=1}\otimes\Ql)\oplus(T(S)^{\sigma=-1}\otimes\Ql).$$
%(The eigenvalues of the Frobenius on $NS(S)^{[\sigma=1]}$ over
%$\FF_p$ is $p$; while on $NS(S)^{[\sigma=-1]}$ is $-p$.) %Note
%also that $T(S)^{\sigma=1}=\emptyset$.)
For the twisted sector, singularities occur along 
$S^{\sigma}$ and for each singularity, its smooth 
resolution is the
sum of $4$ copies of the ruled surface $\PP^1\times C_g$.
So we have
$$L_3(X,s)=L(H^3(X,\Ql),s)$$
$$=L((H^1(E,\Ql)\otimes H^2(S,\Ql))^{\iota\times\sigma}, s)
\times L(H^3(\PP^1\otimes J(C_g),\Ql),s)^4$$ 
$$=L((H^1(E,\Ql)^{\iota=-1}\otimes H^{1,1}(S)^{\sigma=-1})\otimes\Ql,s)$$
$$\times L((H^1(E,\Ql)^{\iota=1}\otimes H^{1,1}(S)^{\sigma=1})\otimes\Ql, s)$$
$$\times L(H^3(\PP^1\otimes J(C_g),\Ql),s)^4$$
$$=L(\rho_E\otimes \chi, s) L(\rho_E\otimes\rho,s) L(J(C_g),s-1)^4.$$
\end{proof}

\subsection{Proof of automorphy of Calabi--Yau threefolds of
Borcea--Voisin type}

Finally we can give proofs for Theorems 5.7, 5.8 and 5.9 on
the automorphy of Calabi--Yau threefolds over $\QQ$ of
Borcea--Voisin type. 

\begin{defn} {\rm (a) We will denote the orthogonal complement
of $NS(S)^{\sigma}$ in $H^{1,1}(S)^{\sigma=-1}$ by $T(S)^{\sigma=-1}$. Its
$\ell$-adic realization $T(S)^{\sigma=-1}\otimes\Ql\subset
 H^2(S,\Ql)$ is called the {\it $K3$ motive}  and denoted by $\MM_S$.
This is the unique motive with $h^{0,2}(\MM_S)=1$.

(Note that $H^{1,1}(S)^{\sigma=1}$ gives rise to motives
$\MM_A$, which are all algebraic in the sense that
$h^{0,2}(\MM_A)=0$. )

(b) We will call the submotive 
$H^1(E,\Ql)^{\iota=-1}\otimes (T(S)^{\sigma=-1}\otimes\Ql)$
of $H^3(X,\Ql)$ the {\it Calabi--Yau motive} of $X$, and
denote by $\MM_X$.}
\end{defn}

\begin{proof} Here we will give proof for Theorem 5.8.
\begin{itemize}
\item $X$ is of CM type by Theorem 4.5. \\
\item $L_2(X,s)$ is automorphic by Proposition 5.14. \\ 
For $L_3(X,s)$, we need to show that the $L$-series associated
to the exceptional divisor arising from
the singular loci $\{P_i\}\times C_g$\, $(i=1,2,3,4)$ 
is automorphic. The exceptional
divisor is given by the $4$ copies of the ruled surface
$\PP^1\times C_g$. Now $C_g$ is a component
of $S^{\sigma}$ where $S$ is a finite quotient of a Fermat
or diagonal surface, hence $C_g$ is again
expressed in terms of a diagonal or quasi-diagonal curve in
weighted projective $2$-space (see Corollary 2.14 and 
Corollary 3.5).  
Hence the Jacobian variety $J(C_g)$ of $C_g$ is also of
CM type, and hence $L(\PP^1\otimes J(C_g),s)=L(J(C_g), s-1)$ is automorphic.\\ 
\item $\zeta(X,s)$ is automorphic, as the factors $L_i(X,s)$
($0\leq i\leq 6)$ are all automorphic.\\   
\end{itemize}
\end{proof}

\begin{proof} Now we will prove Theorem 5.9. Here $E$ can be
any elliptic curve, but $S$ is of CM type. The resulting Calabi--Yau 
threefolds are not necessarily of CM type. 

\begin{itemize}
\item An elliptic curve factor $E$ is modular by the results of Wiles et al.
So there is an automorphic representation $\rho_E$ of dimension $2$ associated
to $H^1(E,\Ql)$.
\item The assertion of (b) is proved in Theorem 5.6. 
\item We know that the $G_{\KK}$-Galois representation $\rho_{G_{\KK}}$ on 
$T(S)$ is a direct sum of $1$-dimensional representations (coming from
Jacobi sums), which the Galois group $\mbox{Gal}(\KK/\QQ)$ permutes 
transitively.
This induces an $m$-dimensional irreducible Galois representation
$\mbox{ind}^{G_{\QQ}}_{G_{\KK}}\rho_{G_{\KK}}$. Hence we obtain
the $2m$-dimensional Galois representation
$\pi:=\rho_E\otimes \mbox{ind}^{G_{\QQ}}_{G_{\KK}}\rho_{G_{\KK}}$, which
is isomorphic to $\mbox{ind}^{G_{\QQ}}_{G_{\KK}}(\rho_{G_{\KK}}\otimes
\mbox{res}^{G_{\QQ}}_{G_{\KK}} \rho_E).$
Hence
$$L(\pi,s)=L(\rho_E\otimes\mbox{ind}^{G_{\QQ}}_{G_{\KK}} \rho_{G_{\KK}}, s)
=L(f_E\otimes f_{T(S)},s).$$

In terms of the local $p$-factors, the above $L$-series is given as follows.

The Euler $p$-factor of the $L$-series can be written
as follows, for good prime $p$.
Let $L_{E,p}(s)$ be the $p$-factor of $L(E,s)$. Then
$$L_{E,p}(s)=(1-\alpha_1p^{-s})(1-\alpha_2p^{-s})$$
where $\alpha_1,\,\alpha_2$ are conjugate algebraic integers
with complex absolute value $p^{1/2}$.

Let $L_{T(S),p}(s)$ be the $p$-factor of $L(T(S),s)$. Then
$$L_{T(S),p}(s)=\prod_{i=1}^t (1-\beta_i p^{-s})$$
where $\beta_i$ are algebraic integers with complex
absolute value $p$ such that $\beta_i/p$ is not a root of unity.

Now for $X=\widetilde{E\times S/\iota\times\sigma}$, let
$L_{\pi,p}(s)$ be the $p$-factor of the $L(\pi,s)$. Then
$$L_{\pi,p}(s)=\prod_{i=1}^t (1-\alpha_1\beta_ip^{-s})(1-\alpha_2\beta_ip^{-s}).
$$
\item The $L$-series $L(E\otimes \chi, s)=L(\rho_E\otimes \chi, s)$ is 
automorphic, as $E$ (or $\rho_E$) is automorphic, and 
$\chi$ is automorphic since it is induced by a 
$GL_1$-representation
of some cyclotomic field over $\QQ$. 

Similarly, the $L$-series
$L((E\otimes H^{1,1}(S)^{\sigma=1})\otimes\Ql,s)$ is automorphic
as $E$ is automorphic, and the representation on 
$H^{1,1}(S)^{\sigma=1}\otimes \Ql$ is
also induced by a $GL_1$-representation of some 
cyclotomic field over $\QQ$. 
\item The $L$-series of $\PP^1\times J(C_g)$ is automorphic as that of
$J(C_g)$ is automorphic. 
\item The $L$-series $L_3(X,s)$ is automorphic as each component 
is automorphic. 
\end{itemize}
\end{proof}

We will give a representation theoretic proof (involving base change
and automorphic induction) for the automorphy results of our 
Calabi--Yau threefolds
of Borcea--Voisin type in the appendix. 

\begin{rem}
In motivic formulation, our automorphy results for $K3$ surfaces
and our Calabi--Yau threefolds may be reformulated as follows:
{\it The $L$-series of the $K3$-motive is automorphic, and 
the $L$-series
of the Calabi--Yau motive is automorphic. 
}
\smallskip

For our $K3$ surfaces $S$ over $\QQ$, the $K3$ motive 
$T(S)^{\sigma}\otimes\Ql$ is a submotive of
$H^2(S,\Ql)$, and the  $L$-series $L_2(S,s)$ factors as 
$$L_2(S,s)=L(NS(S)^{\sigma}\otimes\Ql, s) 
L(T(S)^{\sigma}\otimes\Ql,s).$$  
The automorphy of $L_2(S,s)$ then boils down to
the automorphy of each $L$-factor. But we know
that both factors are automorphic by Theorem 5.6 and its corollaries.

For Calabi--Yau threefolds $X$ of
Borcea--Voisin type, the Calabi--Yau motive 
$H^1(E,\Ql)\otimes (T(S)^{\sigma}\otimes\Ql)$ is a submotive
of $H^3(X,\Ql)$, which appears as a factor of $L_3(X,s)$. 
The automorphy of $L_3(X,s)$ again boils down to the 
automorphy of the $L$-series of the Calabi--Yau motive. 
Indeed, the other factors of $L_3(X,s)$ are expressed in terms 
of the $L$-series of the tensor product of 
$\rho_E$ and the motives $\MM_A$ of $K3$ surfaces with 
$h^{2,0}(\MM_A)=0$. These motives are all automorphic as they
are induced from $GL_1$-representations of some cyclotomic fields
over $\QQ$.  
\end{rem}

\section{Mirror symmetry for Calabi--Yau
threefolds of Borcea--Voisin type}

\subsection{Mirror symmetry for $K3$ surfaces}

There are several versions of mirror symmetry for $K3$ surfaces:
\begin{itemize}
\item Arnold's strange duality. This version is discussed
by Dolgachev, Arnold and by others in relation to
singularity theory. It is formulated for lattice polarized
$K3$ surfaces as follows:
A pair of lattice polarized $K3$ surfaces $(S, S^{\vee})$ is said
to be a {\it mirror} pair if
$$\pic (S)^{\perp}_{H^2(S,\ZZ)}=U\oplus \pic (S^{\vee})$$
as lattices.  In terms of the Picard numbers,
$$22-\rho(S)=2+\rho(S^{\vee})\quad\Leftrightarrow\quad
\rho(S^{\vee})=20-\rho(S).$$
(See, for instance, Dolgachev \cite{Dol96}.) 
\item Berglund--H\"ubsch--Krawitz mirror symmetry (Berglund--H\"ubsch 
\cite{BH93} and Krawitz \cite{K10}). 
This version of mirror symmetry is for finite quotients of hypersurfaces
in weighted projective $3$-spaces. Mirror symmetry for these
$K3$ surfaces are addressed in the articles by Artebani--Boissi\`ere and Sarti 
\cite{ABS} and Comparin--Lyons--Priddis--Suggs \cite{CLPS13}.
It is formulated as follows: Let $W$ be a quasihomogeneous invertible 
polynomial together with a group $G$ of diagonal 
automorphisms. (Here an ``invertible''
polynomial means that it has the same number of monomials as variables.
Thus their zero loci define Delsarte surfaces. )
Let $Y_W$ be the hypersurface $\{W=0\}$ in a weighted projective $3$-space,
then the orbifold $Y_W/G$ defines a $K3$ surface. Now define the
polynomial $W^T$ by transposing the exponent matrix of $W$. Then $W^T$
is again invertible and let $G^T$ be the dual group of $G$. Then
the orbifold $Y_{W^T}/G^T$ is again a $K3$ surface. The Berglund--H\"ubsch--Krawitz
mirror symmetry is that $Y_W/G$ and $Y_{W^T}/G^T$ form a mirror
pair of $K3$ surfaces. 
\end{itemize}

These two versions of mirror symmetry for $K3$ surfaces are shown
to coincide for certain $K3$ surfaces in \cite{CLPS13}.

Now we consider mirror symmetry for pairs $(S,\sigma)$ of
$K3$ surfaces with involution $\sigma$ classified by
Nikulin in terms of triplets $(r,a,\delta)$.
Let $(S, \sigma)$ be a pair of $K3$ surfaces with involution $\sigma$
corresponding to a triplet $(r,a,\delta)$.  Then the mirror pair
$(S^{\vee},\sigma^{\vee})$ corresponds to the 
triplet $(20-r, a, \delta)$.

For the Nikulin pyramid given in Section 2, the mirror is
placed at the vertical line $r=10$, corresponding to the
symmetry $(r,a,\delta)\leftrightarrow (20-r,a,\delta)$.
It should be remarked that mirrors do not exist for
the points located at the utmost right outerlayer of the pyramid,
(the so-called the ``pale region''), that is, $(r,a,\delta)$ is 
one of the following triplets
$(20,2,1)$, $(19,3,1)$, $(18,4,1)$, $(18,4,0)$, $(17,5,1)$, 
$(16,6,1)$, $(15,7,1)$, $(14,8,1)$, $(13,9,1)$ 
and $(12,10,1)$. 
However, one particular triplet $(14,6,0)$ is not in this
region, but does not have a mirror partner.
 
\subsection{Mirror symmetry of Calabi--Yau threefolds of
Borcea--Voisin type}

Now we consider our Calabi--Yau threefolds of Borcea--Voisin type
obtained as crepant resolutions of quotient threefolds
$E\times S/\iota\times\sigma$.  Mirror symmetry for these
Calabi--Yau threefolds has been discussed by Voisin \cite{V93}
and also by Borcea \cite{B94}. 

\begin{thm} {\sl  Given a Calabi--Yau threefold of Borcea--Voisin type $X=X(r,a,\delta)
=\widetilde{E\times S/\iota\times\sigma}$, there is a
mirror family of Calabi--Yau threefolds
$X^{\vee}=X(20-r,a,\delta)
=\widetilde{E\times S^{\vee}/\iota\times\sigma^{\vee}}$ 
such that
$$e(X^{\vee})=-e(X).$$
Mirror symmetry for Calabi--Yau threefolds $X$ is
purely determined by mirror symmetry for the $K3$ components $S$.

Borcea's formulation of mirror symmetry is:
$$h^{1,1}(X^{\vee})=5+3(20-r)-2a=65-3r-2a=h^{2,1}(X),$$
$$h^{2,1}(X^{\vee})=65-3(20-r)-2a=5+3r-2a=h^{1,1}(X)$$
and
$$e(X^{\vee})=-12(r-10)=-e(X).$$
That is, mirror symmetry interchanges $r$ by $20-r$.
\smallskip

Voisin's formulation of mirror symmetry is given as follows:
Recall that the fixed part $S^{\sigma}$ of $S$ under $\sigma$
ia a disjoint union of a genus-$g$ curve and $k$ rational
curves on $S$.  Put
$$N:=1+k=\mbox{the number of components of $S^{\sigma}$},$$
and
$$N^{\prime}:=\mbox{the sum of genera of components of $S^{\sigma}$}.$$
Then
$$h^{1,1}(X)=11+5N-N^{\prime},$$
$$h^{2,1}(X)=11+5N^{\prime}-N$$
and
$$e(X)=12(N-N^{\prime}).$$
Mirror symmetry interchanges $N$ and $N^{\prime}$.
$$h^{1,1}(X^{\vee})=11+5N^{\prime}-N,$$
$$h^{2,1}(X^{\vee})=11+5N-N^{\prime}$$
and
$$e(X^{\vee})=12(N^{\prime}-N)=-e(X).$$
}
\end{thm}

\begin{rem}
The mirror symmetry in the above theorem is merely a 
numerical check for
the topological mirror symmetry that the Hodge numbers of $X(r,a,\delta)$ 
and $X(20-r,a, \delta)$ are indeed ``mirrored''. Mirror symmetry for 
$X=X(r,a,\delta)$ indeed comes from the mirror symmetry
of the $K3$ surface component.
Also the mirror of $X(r,a,\delta)$ occurs in a family, so mirror symmetry
does not relate one Calabi--Yau threefold to another Calabi--Yau
threefold, rather mirror symmetry deals with families. 
\end{rem}

\begin{rem}
For the Calabi--Yau threefolds corresponding to 
the $11$ $K3$ surfaces corresponding to the triplets $(r,a,\delta)$ located 
at the utmost right outerlayer of the Nikulin's pyramid (called the 
``pale region'' by Borcea) plus the triplet $(14,6,0)$, mirror partners 
do not exist.

Rohde \cite{R09} and Garbagnati--van Geemen \cite{GvG10} considered
those Calabi--Yau threefolds of Borcea--Voisin type whose
$K3$ surface components have only rational curves in their fixed loci 
by non-symplectic involution (i.e., no curves with higher genera).
A reason for not having mirror 
partners is the non-existence of boundary points in the
complex structure moduli space of the Calabi--Yau threefold where the
variation of Hodge structures on $H^3$ has maximal unipotent
monodromy, and hence there is no way of defining mirror maps.  
\end{rem}

\subsection{Mirror pairs of $K3$ surfaces}

Now we consider the $95$ $K3$ surfaces in the
list of Reid and Yonemura.  Belcastro \cite{Bel02} determined
the Picard lattices for these $95$ $K3$ surfaces, and showed
that the set of these $95$ $K3$ surfaces are not closed under mirror 
symmetry. 

We can fish out those $K3$ surfaces with involution $\sigma$ which
are closed under mirror symmetry.

\begin{lem} [Belcastro \cite{Bel02}] 
{\sl The set of the $95$ $K3$ surfaces of
Reid and Yonemura is not closed under mirror symmetry.
Among them, the $57$ $K3$ surfaces have mirror partners within the list.}
\end{lem}

\begin{lem}
{\sl All $57$ $K3$ surfaces $S$ have non-symplectic involutions $\sigma$ 
acting as $-1$ on $H^{2,0}(S)$, and their mirror partners $S^{\vee}$ also 
have non-symplectic involutions $\sigma^{\vee}$ 
acting as $-1$ on 
$H^{2,0}(S^{\vee})$.}
\end{lem}

\begin{proof} We tabulate the $57$ $K3$ surfaces with
involutions, and their mirror partners in Table \ref{table9} and
Table \ref{table10}.
\end{proof}

\subsection{Examples of mirror pairs of Calabi--Yau threefolds
of Borcea--Voisin type}

\begin{ex}
Let $E=E_2$ be the elliptic curve with involution $\iota$
as in Section 4.3, and let $S_0$ be the $K3$ surface, $\# 14$ 
in Yonemura and $\# 26$ in Borcea, 
given by
$$S_0: x_0^2=x_1^3+x_2^7+x_3^{42}\subset\PP^3(21,14,6,1)$$
of degree $42$ and involution $\sigma(x_0)=-x_0$. Let $S$ be the 
minimal resolution of $S_0$. $S$ has Nikulin's triplet $(10,0,0)$. 
Thus, $S$ is its own mirror. Recall from Example 6.4 that
the fixed locus $S^{\sigma}$ is
$S^{\sigma}=C_6\cup L_1\cup\cdots\cup L_5$.  
Also $S$ is of CM type. This is because $S$ is dominated 
by the Fermat surface of degree $42$. Hence the
field $\KK$ corresponding to the transcendental cycles $T(S)$ of $S$ is
the cyclotomic field $\QQ(\zeta_{42})$ with $[\KK:\QQ]=\varphi(42)=12$. 
(Here $\zeta_{42}$ is a primitive $42$-th root of unity and 
$\varphi$ is the Euler phi-function.) Note that
%$12=22-r=22-10$, so that $T(S)\cong T(S)^{\sigma}$.
$10=22-12=r$, so that $NS(S)\cong NS(S)^{\sigma}$. (Or
equivalently, $12=22-r=22-10$ so that 
$T(S)\cong T(S)^{\sigma}$.) 
The $K3$-motive is automorphic and hence $S$ is automorphic by Theorem 4.5.

The Calabi--Yau threefold $X=\widetilde{E_2\times S/\iota \times \sigma}$ has a
birational model defined over $\QQ$
$$X: z_0^4+z_1^4=z_2^3+z_3^7+z_4^{42}\subset\PP^4(21,21,28,12,2)$$
of degree $84$. Since $E$ and $S$ are both of CM type, so is $X$.
The Hodge numbers and the Euler characteristic are
$$h^{1,1}(X)=35=1+10+4(1+5),\, h^{2,1}(X)=35=1+10+4\cdot 6,\, e(X)=0$$
so that $X$ is its own topological mirror.

Obviously $L_i(X,s)$ and $L_{6-i}(X,s)$ for $i=0,1,2$ are all automorphic.
To show the automorphy of $L_3(X,s)$, we have only to show the
automorphy of the $L$-series corresponding to the Calabi--Yau 
motive
$H^1(E,\Ql)\otimes T(S)^{\sigma}\otimes\Ql$. 
The Galois representation associated to the Calabi--Yau motive has dimension $24$,
and is given by the tensor product of the $2$-dimensional Galois
representation associated to $H^1(E,\Ql)$ and the $12$-dimensional irreducible
Galois representation associated to $T(S)^{\sigma}\otimes\Ql$ induced from a 
Jacobi sum Grossencharacter of $\KK=\QQ(\zeta_{42})$.  Hence it is
automorphic.
We repeat the argument in motivic formulation. 
The Calabi--Yau motive $\MM_X$ has dimension 
$\varphi(84)=24$, 
and the Jacobi sum Grossencharacter of $K=\QQ(\zeta_{84})$ gives
$GL_1$-representations and its automorphic induction gives rise
to the $GL_{24}$ irreducible representation for $\MM_X$ over
$\QQ$, and hence it is modular (automorphic).
(Compare the Calabi--Yau motive $\MM_X$ with the $\Omega$-motive
constructed by Schimmrigk in \cite{S13}.)
\end{ex}

\begin{ex}
Let $E=E_2$ be the elliptic curve with involution $\iota$
as in Section 4.3 and let $S_0$ be the $K3$ surface, $\#40$ 
in Yonemura and $=\#5$ in Borcea, 
given by
$$S_0: x_0^2=x_1^3x_2+x_1^3x_3^2+x_2^7-x_3^{14}\subset\PP^2(7,4,2,1)$$
of degree $14$ with involution $\sigma(x_0)=-x_0$. Its minimal resolution $S$ 
has Nikulin's 
triplet $(7,3,0)$. By Theorem \ref{thm-DelsarteS}, we may 
remove the monomial $x_1^3x_2^2$ from the defining equation, we get
$$S_0: x_0^2=x_1^3x_2+x_2^7-x_3^{14}$$
making $S_0$ of CM type. This is a weighted hypersurface of degree $14$,
and $\mbox{lcm}(3,2,14)=42$, and hence $S_0$ is dominated by
the Fermat surface of degree $42$ (cf. \cite{GKY10},
Corollary 8.1).
The field $\KK$ corresponding to $T(S)$ is the cyclotomic field $\QQ(\zeta_{42})$ 
of degree $\varphi(42)=12$, and we obtain the induced
Galois representation of dimension $12$. Thus, the $K3$-motive is
automorphic, and hence $S$ is automorphic. 

In this case,  $r=7\neq 10=22-12$ so
$NS(S)^{\sigma}\not\cong NS(S)$. (Or equivalently, 
$22-r=22-7=15\neq 12=\varphi(42)$ so $T(S)^{\sigma}\not\cong T(S)$.)
%In fact, the anti-invariant algebraic cycles in 
%$NS(S)$ under $\sigma$ has rank $5$. 

The Calabi--Yau threefold $X$ has a birational model defined  
over $\QQ$:
$$X: z_0^4+z_1^4=z_2^3z_3+z_3^7-z_4^{14}
\subset\PP^4(7,7,8,4,2)$$
of degree $28$, and $\mbox{lcm}(4,3,14)=84$. (See \cite{GKY10}, Theorem 9.2.) 
Since $E$ and $S$ are of CM type, so is $X$.
The Hodge numbers and the Euler characteristic are
$$h^{1,1}(X)=20=1+7+4(2+1),\, h^{2,1}(X)=38=1+(20-7)+4\cdot 6,\,\, e(X)=-36.$$
Now we apply Theorem 5.9. We pass from $\QQ(\zeta_{42})$ to
$\QQ(\zeta_{84})$ to take $H^1(E_2)$ into account.
The Calabi--Yau motive $\MM_X$ has dimension $24=\varphi(84)$. 
Indeed, the Jacobi sum Grossencharacter of 
$\KK=\QQ(\zeta_{84})$ gives rise to %the $GL_1$ representations of $\mbox{Gal}(\overline{\QQ}/\KK)$, 
%and its automorphic induction gives rise to 
the $GL_{24}$ irreducible automorphic 
cuspidal representation for the Calabi--Yau motive $\MM_X$ 
over $\QQ$. Hence the Calabi--Yau motive $\MM_X$ is automorphic. 
Hence $L_3(X,s)$ is automorphic,
and consequently $X$ is automorphic.
\medskip

To find a mirror family of Calabi--Yau threefolds, we first
look for a mirror $S^{\vee}$ of $K3$ surface $S$. 
We may take for $S^{\vee}$ the K3 surface $\#47$ in Yonemura 
=$\#24$ in Borcea. $S^{\vee}$ is a
$K3$ surface defined by
$$S^{\vee}: x_0^2=x_1^3+x_1x_2^7+x_2^9x_3^2+x_3^{14}
\subset\PP^3(21,14,4,3)$$
of degree $42$. It has a non-symplectic involution
$\sigma^{\vee}$ that sends $x_0$ to $-x_0$. 
The pair $(S^{\vee},\sigma^{\vee})$ corresponds to the 
triplet $(13,3,0)$.
By Theorem 4.3, we may remove the monomial $x_2^9x_3^2$ from
the defining equation, which makes $S^{\vee}$ to be of CM type.
So $S^{\vee}: x_0^2=x_1^3+x_1x_2^7+x_3^{14}$ is a weighted 
hypersurface of degree $28$. 
Since $\mbox{lcm}(2,7,4)=28$, the field $\KK$ corresponding 
to $T(S^{\vee})$ is the cyclotomic field $\QQ(\zeta_{28})$ 
of degree $\varphi(28)=12$, and we obtain the induced
Galois representation of dimension $12$. Thus, the 
$K3$-motive is automorphic, and hence $S^{\vee}$ is 
automorphic.
In this case, $22-r=22-13=9\neq 12=\varphi(28)$ 
so $T(S^{\vee})^{\sigma}\not\cong T(S^{\vee})$.
(Or equivalently, $r=13\neq 10=22-12$ so that 
$NS(S^{\vee})^{\sigma}\not\cong NS(S^{\vee})$.) 
\medskip

A candidate for mirror family for $X$ might be a 
deformation of  
$\widetilde{E_2\times S^{\vee}/\iota\times\sigma^{\vee}}$.
One member of this mirror family denoted by, $X^{\vee}$,
may be chosen to have a birational model defined over $\QQ$ by
the following equation:
$$X^{\vee}: z_0^4+z_1^4=z_2^3+z_2z_3^7+z_4^{14}
\subset\PP^4(21,21,28,8,6)$$
of degree $84$.  The Hodge numbers and the Euler characteristic
of $X^{\vee}$ are
$$h^{1,1}(X^{\vee})=38=1+13+4(5+1),\, h^{2,1}(X^{\vee})=20=1+(20-13)+4\cdot 3,\,\, 
e(X^{\vee})=36.$$

We pass from $\QQ(\zeta_{28})$ to $\QQ(\zeta_{56})$ to take $H^1(E_2)$ into
account. Then the Calabi--Yau motive $\MM_{X^{\vee}}$ has 
dimension
$24=\varphi(56)$. Again, by Theorem 5.9, the Jacobi sum Grossencharacter of 
$\KK=\QQ(\zeta_{28})$ gives rise to a $GL_1$ representations 
for $\mbox{Gal}(\overline{\QQ}/\KK)$ and its automorphic induction yields 
the $GL_{24}$ irreducible cuspidal automorphic representation for the Calabi--Yau motive
$\MM_{X^{\vee}}$ over $\QQ$. Hence the Calabi--Yau motive 
$\MM_{X^{\vee}}$ is automorphic.
Consequently, we conclude that $L_3(X^{\vee},s)$ is automorphic, and
hence the automorphy of $X^{\vee}$.

For these two examples, it happens that $L(\MM_X,\rho,s)=L(\MM_{X^{\vee}},\rho^{\vee},s)$.
That is the $L$-series of the Calabi--Yau motives of $X$ and $X^{\vee}$ 
coinside.
\end{ex}

\subsection{Automorphy and mirror symmetry for Calabi--Yau threefolds
of Borcea--Voisin type}

Mirror symmetry for Calabi--Yau threefolds is not the
correspondence for one Calabi--Yau threefold to another,
rather it is a correspondence between families.  At the
moment, we do not know how to compute the zeta-functions
and $L$-series of a deformation family of mirror Calabi--Yau
threefolds. So we will consider one particular member of
this mirror family and compare the $L$-series of
the Calabi--Yau motives.

\begin{thm} {\sl Let $(S,\sigma)$ be one of the $57$ $K3$
surfaces in Lemma 6.2 with involution $\sigma$, which are 
closed under mirror symmetry. Then $S$ is of CM type. Let 
$X=X(r,a,\delta)$ be a Calabi--Yau threefold corresponding 
to a triplet $(r,a,\delta)$ as in Section 4.1, so 
$X=\widetilde{E\times S/\iota\times\sigma}$. Then a mirror
family of Calabi--Yau threefolds exists and corresponds to
a triplet $(20-r,a,\delta)$, and may be obtained as a 
deformation of a crepant resolution of the quotient 
$E\times S^{\vee}/\iota\times\sigma^{\vee}$, where
$\sigma^{\vee}$ is a non-symplectic involution on $S^{\vee}$.
Then a special member $X^{\vee}$ of this mirror family has
the following properties:
\begin{itemize}
\item[(a)] $X^{\vee}$ has a model defined over $\QQ$ provided
that $E$ is defined over $\QQ$.
\item[(b)] $X^{\vee}$ is of CM type if and only if $E$ is of CM type.
\item[(c)] If $X^{\vee}$ is of CM type, then $X^{\vee}$ is automorphic.
\end{itemize}}
\end{thm}

\begin{obs}{\sl Under the situation of the above theorem, we have

\begin{itemize} 
\item[(d)] If the $K3$ motives of $S$ and $S^{\vee}$ are
isomorphic (in the sense that they correspond to
the same Jacobi sum Grossencharacter), then they have the
same $L$-series. Furthermore, the Calabi--Yau motives of 
$X$ and $X^{\vee}$ are invariant under mirror symmetry. 
\end{itemize}}
\end{obs}

The two examples 6.4 and 6.5 are in support of this observation.
It appears that when the original Calabi--Yau threefold and a
member of its mirror Calabi--Yau threefolds are both 
of CM type and are realized as finite quotients of the same
Fermat or quasi-diagonal hypersurface, then the Calabi--Yau
motives are the same and hence are invariant under mirror
symmetry. 

\subsection{Berglund--H\"ubsch--Krawitz mirror symmetry 
for Calabi--Yau threefolds}

Here are other examples of Calabi--Yau threefolds of CM type due
to Kelly \cite{K13}. For the computations of zeta-functions and
$L$-series, we use the method developed in Goto-Kloosterman-Yui
\cite{GKY10}.

Consider the polynomials
$$F_A: x_0^8+x_1^8+x_2^4+x_3^3+x_4^6=0$$
$$F_{A^{\prime}}: x_0^8+x_1^8+x_2^4+x_3^3+x_3x_4^4=0.$$
Both are hypersurfaces of degree $24$ in the weighted projective 
$4$-space $\PP^4(3,3,6,8,4)$. %Both have the same degree $24$. 
Let $\zeta=\zeta_{24}$ be a primitive $24$-th
root of unity.  Both $F_A$ and $F_{A^{\prime}}$ are covered by the Fermat 
hypersurface of degree $24$ (see Theorem 9.2 in \cite{GKY10}), and hence $F_A$ and
$F_{A^{\prime}}$ are both of CM type. 

Let $J_{F_A}=\mbox{Aut}(F_A)\cap \CC^*$. Then $J_{F_A}$ is generated by 
$(\zeta^3,\zeta^3,\zeta^6, \zeta^8,\zeta^4)\in(\CC^*)^5$.
Define the group $SL(F_A):=\{\,(\lambda_0,\lambda_1,\cdots,\lambda_4)\in
\mbox{Aut}(F_A)\,|\,\prod_{j=0}^4 \lambda_j=1\,\}$. Fix a group $G$
so that $J_{F_A}\subseteq G\subseteq SL(F_A)$. Put $\tilde G:=G/J_{F_A}$. Define
$Z_{A,G}:=X_{F_A}/\tilde G$. Then $Z_{A,G}$ is a Calabi--Yau threefold (orbifold). 

For our $F_A$ and $F_A^{\prime}$, choose $G$ and $G^{\prime}$ to be the same
group given by
$$G=G^{\prime}=<(\zeta^3,\zeta^3,\zeta^6,\zeta^8,\zeta^4),(\zeta^{18},1,\zeta^6,1,1),
(1,1,\zeta^{12},1,\zeta^{12})>.$$   
Then $Z_{A,G}$ and $Z_{A^{\prime},G^{\prime}}$ are Calabi--Yau threefolds
which are in the same family of hypersurfaces in $\PP^4(3,3,6,8,4)/\tilde G$.
Since both are realized as finite quotients of the Fermat hypersurface
of degree $24$, both $Z_{A,G}$ and $Z_{A^{\prime},G^{\prime}}$ are of CM type.

The Hodge numbers are given by:
$$h^{1,1}(Z_{A,G})=7,\,\, h^{2,1}(Z_{A,G})=55$$
and
$$h^{1,1}(Z_{A^{\prime},G^{\prime}})=55, \,\, 
h^{2,1}(Z_{A^{\prime}, G^{\prime}})=7.$$ 

Now recall the construction of the Berglund--H\"ubsch--Krawitz 
mirrors of these Calabi--Yau threefolds.
Let
$$F_{A^T}=F_A: x_0^8+x_1^8+x_2^4+x_3^3+Y_4^6=0\subset\PP^4(3,3,6,8,4)$$
$$F_{(A^{\prime})^T}: x_0^8+x_1^8+x_2^4+x_2^3x_3+x_4^4=0\subset\PP^4(1,1,2,2,2).$$
Then $F_{A^T}$ is a hypersurface of degree $24$ in the weighted projective
$4$-space $\PP^4(3,3,8,6,4)$ but $F_{(A^{\prime})^T}$ is a hypersurface
of degree $8$ in the weighted projective $4$-space $\PP^4(1,1,2,2,2)$.
The groups $J_{F_A},\, J_{F_{(A^{\prime})^T}},\, G^T,\, 
(G^{\prime})^T$ are computed:
$$J_{F_A}=<(\zeta^3,\zeta^3,\zeta^6,\zeta^8,\zeta^4)>;
\,\,J_{F_{(A^{\prime})^T}}
=<(\zeta^3,\zeta^3,\zeta^6,\zeta^6,\zeta^6)>;$$
$$G^T=J_{F_A};\,\, (G^{\prime})^T
=<(\zeta^3,\zeta^3,\zeta^6,\zeta^6,\zeta^6),(1,1,1,\zeta^{12},
\zeta^{12})>.$$
Then taking the quotients, we obtain Calabi--Yau orbifolds
$Z_{A^T,G^T}$ and $Z_{(A^{\prime})^T,(G^{\prime})^T}$ which are the
topological mirrors of $Z_{A,T}$ and $Z_{A^{\prime}, G^{\prime}}$
respectively.
$$h^{1,1}(Z_{A^T,G^T})=55,\,\,h^{2,1}(Z_{A^T,G^T})=7$$
and
$$h^{1,1}(Z_{(A^{\prime})^T,(G^{\prime})^T})=7,\,\,
h^{2.1}(Z_{(A^{\prime})^T,(G^{\prime})^T})=55.$$

The Berglund--H\"ubsch--Krawitz mirror symmetry is that
$Z_{A,G}$ and $Z_{A^T,G^T}$ are mirror partners in the sense of
interchanging Hodge numbers. Similarly, 
$Z_{A^{\prime},G^{\prime}}$
and $Z_{(A^{\prime})^T,(G^{\prime})^T}$ are mirror pairs. 
However, the latter two do not live in the same weighted 
projective $4$-spaces.

\begin{thm} [Kelly\cite{K13}] {\sl Let $Z_{A,G}$ and
$Z_{A^{\prime},G^{\prime}}$ be the Calabi--Yau orbifolds 
constructed above. Let $Z_{A^T, G^T}$ and $Z_{(A^{\prime})^T, 
(G^{\prime})^T}$ be Berglund--H\"ubsch--Krawitz mirrors, 
respectively. If $G=G^{\prime}$,
then $Z_{A^T,G^T}$ and $Z_{(A^{\prime})^T,(G^{\prime})^T}$ 
are birational.}
\end{thm}

\begin{prop} {\sl Both $Z_{A,G}$ and 
$Z_{A^{\prime},G^{\prime}}$ are of CM type and hence 
automorphic.  The mirrors $Z_{A^T,G^T}$ and 
$Z_{(A^{\prime})^T,(G^{\prime})^T}$ 
are again of CM type and hence automorphic.
The $L$-series of the Calabi--Yau motives of $Z_{A,G}$
and $Z_{A^T,G^T}$ are invariant under the mirror symmetry.
Similar assertions hold for $Z_{A^{\prime},G^{\prime}}$ 
and $Z_{(A^{\prime})^T,(G^{\prime})^T}$.}
\end{prop}

\begin{proof} We have only to show the last claim. Since the
Calabi--Yau motives of Calabi--Yau threefolds $Z_{A,G}$ and
$Z_{A^T,G^T}$ come from the unique Fermat motive associated to
the weight of the same Fermat hypersurface, the 
Calabi--Yau motive is invariant under the mirror symmetry. 
For $Z_{A^{\prime},G^{\prime}}$ and 
$Z_{(A^{\prime})^T,(G^{\prime})^T}$,
they do not sit in the same family of hypersurfaces, 
but they are birational.  The Calabi--Yau motives are 
left invariant under birational map.
The $L$-series of the Calabi--Yau motives are invariant
under mirror symmetry.
For details about Fermat motives, see Appendix in the
section $7$ below or Goto--Kloosterman--Yui \cite{GKY10}, 
and Kadir--Yui \cite{KY08}.
\end{proof}
 
\begin{rem}
Rohde \cite{R09} (see Appendix A, page 209) constructed
many examples of Calabi--Yau threefolds of CM type 
(CMCY $3$-folds), by Borcea--Voisin construction. The 
automorphy of his CMCY $3$-folds
should follow by studying Galois representations associated
to them. This is left to the reader for exercise.
\end{rem}

\section{Appendix: 
Base change and automorphic induction, and
Rankin--Selberg $L$-series of convolution}\label{appendix} 

\subsection{Base change and automorphic induction maps}

For the proof of automorphy of our Calabi--Yau threefolds via representation
theory, we need the three ingredients, (the existence of) base change 
and automorphic induction maps for solvable extensions over $\QQ$, and 
the Rankin--Selberg $L$-series of convolution.

In this subsection, we will explain the result of Arthur and Clozel
\cite{AC90} on base change and automorphic induction proved for cyclic
extensions of prime degree over $\QQ$, and their generalization by
Rajan \cite{Raj02} (see also Murty [M93]) to solvable extensions over $\QQ$.

\begin{defn} {\rm Let $k$ be a number field with the 
ring $\OO_k$ of integers.  Let $K$ be a Galois extension 
of $k$ with the ring of integers $\OO_K$ and Galois 
group $G=\mbox{Gal}(K/k)$.  %Assume that $G$ is abelian.
If $\rho$ is an irreducible (finite-dimensional) representation
of $G$, we can associate to it a Dirichlet series with
Euler product, called the Artin $L$-series $L(s,\rho,K/k)$ 
as follows.
Let $v$ be a (finite) place of $\OO_k$, $p_v$ the
associated prime ideal in $\OO_k$, $q_v$ the cardinality
of the residue field $\OO_k/p_v$, and $\Phi_v$ the
conjugacy class of Frobenius elements attached to $p_v$, for
$v$ unramified in the extension $K/k$. Let $\fS$ be the finite
set of (finite) places ramified in $K/k$. The Artin
$L$-series is defined by
$$L(s,\rho,K/k)=\prod_{v\not\in \fS}
\frac{1}{\mbox{det}(1_{\rho}-q_v^{-s}\rho(\Phi_v))}.$$
The definition of the Artin $L$-series
can be extended to arbitrary representations of $G$ by
additivity:
$$L(s,\rho_1\oplus\rho_2, K/k)=L(s,\rho_1, K/k)L(s,\rho_2, K/k).$$}
\end{defn}

Let $\bA_k$ be the adele ring of $k$
and $\fA(GL_n(\bA_k))$ be the set of automorphic
representations of $GL_n(\bA_k)$ for some $n$.

The Langlands philosophy predicts that an Artin $L$-series should
be equal to an $L$-series associated to some automorphic
form (e.g., cusp form) on $GL_n$. More concretely,
for each $\rho$, the Langlands reciprocity conjecture states that
there exists an automorphic representation $\pi(\rho)
\in \fA(GL_n(\bA_K))$ ($n=\mbox{deg}(\rho)$) such that
$$L(s,\rho, K/k)=L(s,\pi(\rho)).$$

We assert that the Artin $L$-functions of the Calabi--Yau 
threefolds of Borcea--Voisin type which are of CM type
are indeed automorphic.
 
Now we need to introduce the notion of ``base change''
and ``automorphic induction''.

\begin{lem} {\sl
Let $H$ be a subgroup of $G$, and let $K^H$ be the fixed subfield
of $K$ by $H$. Let $\psi$ be an Artin representation of
$\mbox{Gal}(K/K^H)=H$. Let $L(s,\psi, K/K^H)$ be the Artin $L$-series
of the extension $K/K^H$. Then the Artin $L$-series
is invariant under induction, that is, if $\mbox{Ind}^H_G$ is the induced
representation, then
$$L(s,\mbox{Ind}^H_G\psi, K/K^H)=L(s,\psi, K/K^H).$$
}
\end{lem}

When $L(s,\rho, K/k)=L(s,\pi(\rho))$, then
$L(s,\rho|_H, K/K^H)=L(s,\rho\otimes\mbox{Ind}_H^G {\bf 1}, K/k)$.
But $\mbox{Ind}_H^G{\bf 1}=\mbox{reg}_H$ is nothing but the
permutation representation on the cosets of $H$ in $G$.
Let $\pi\in\fA(GL_n(\bA_k))$. For each unramified $\pi_v$,
let $A_v\in GL_n(\CC)$ be a semi-simple conjugacy class defined
by the representation $\pi$.
If $v$ is unramified in $K$, define
$$L_v(s, B(\pi))=\mbox{det}(1-A_v\otimes\mbox{reg}_H(\sigma_v)Nv^{-s})^{-1}$$
where $\sigma_v$ is the Artin symbol of $v$.

\begin{conj}{\sl
\begin{itemize}
\item [(a)] (Base change)  There exists a base change map
$$B : \fA(GL_n(\bA_K))\to \fA(GL_n(\bA(K^H)))$$
and the Artin $L$-series $L(s,B(\pi), K/K^H)$ such that
its $v$-factor coincides with $L_v(s, B(\pi))$ defined above. \\
\item [(b)] (Automorphic induction)
Now let $\psi$ be a representation of $H$. Then
there exists an automorphic induction map
$$I: \fA(GL_n(\bA_{K^H})) \to \fA(GL_{nr}({\bA_k}))$$
such that for $I(\pi)\in\fA(GL_{nr}(\bA_k))$, 
$$L(s,I(\pi))=L(s, \mbox{Ind}^G_H, K/k).$$
Here $n=\mbox{deg}(\psi)$, and $r=[G:H]$.
\end{itemize}}
\end{conj}

We now recall a theorem of Arthur and Clozel \cite{AC90} on
the existence of base change and automorphic induction maps
for $GL_n$, when $K/k$ is a cyclic extension of prime degree,
and representations are automorphic cuspidal representations.

\begin{thm}
[Arthur--Clozel] {\sl
Suppose that $K/k$ is a cyclic extension of prime degree $\ell$.
Let $\pi$ and $\Pi$ denote cuspidal unitary automorphic representations
of $GL_n(\bA_{k})$ and $GL_n(\bA_{K})$, respectively.
Then

\begin{itemize}
\item the base change lift of $\pi$, denoted by $B(\pi)$, exists,
and it is an automorphic representation in
$\fA(GL_n(\bA_K))$, \\
\item the automorphic induction $I(\Pi)$ of $\Pi$
exists, and it is an automorphic
representation in $\fA(GL_{n\ell}(\bA_k))$.
\end{itemize}}
\end{thm}

\subsection{Rankin--Selberg $L$-series of convoluton}

We can reformulate the Arthur--Clozel theorem in terms
of the $L$-series.
In this subsection, we will consider Rankin--Selberg $L$-series
of convolution. These $L$-series are needed from the fact that
the eigenvalues of the Frobenius morphism of our Calabi--Yau threefolds
of Borcea--Voisin type are given by tensor products of
eigenvalues of those of the components.
We need to consider Rankin--Selberg $L$-series of convolution.

Let $\pi$ and $\pi^{\prime}$ be two cuspidal, unitary
automorphic representations
of $GL_n(\bA_k)$ and $GL_m(\bA_k)$, respectively.  Let $\mathfrak{S}$ be a
finite set of primes of $k$ such that $\pi$ and $\pi^{\prime}$ are
unramified outside $\mathfrak{S}$.  Let $L(s,\pi\otimes\pi^{\prime})$
be the Rankin--Selberg $L$-series of convolution.  Then
the result of Arthur and
Clozel mentioned above is formulated in terms of the
Rankin--Selberg $L$-series as follows:

\begin{lem} {\sl
Let $K/k$ be cyclic extension of prime degree,  and suppose
that $\pi\in \fA(GL_n(\bA_k))$ and $\Pi\in\fA(GL_m(\bA_K))$
are cuspidal unitary automorphic representations, respectively. 
Then the Rankin--Selberg $L$-series satisfies the
formal identity:
$$L(s,B(\pi)\otimes\Pi)=L(s,\pi\otimes I(\Pi)).$$
}
\end{lem}

\subsection{Generalizations of base change and automorphic
induction to solvable extensions over $\QQ$}

Arthur and Clozel's results are proved for cyclic extensions of
prime degree over $\QQ$. For our application, we need 
base change and automorphic induction results for
abelian extensions (e.g., cyclotomic fields) over $\QQ$. 
In fact, the existence of base
change and automorphic induction is established for solvable
extensions over $\QQ$ by Rajan \cite{Raj02}, see also 
Murty \cite{M93}.

\subsection{Weighted Jacobi sums and Fermat motives}

We recall now the definition of weighted Jacobi sums and weighted
Fermat motives from Gouv\^ea--Yui \cite{GY95}.

We consider a weighted Fermat hypersurface of dimension $n+1$, degree
$m$ and a weight ${\bold{w}}=(w_0,w_1,\cdots,w_{n+1})$ defined by
$$x_0^{m_0}+x_1^{m_1}+\cdots+x_{n+1}^{m_{n+1}}=0\subset\PP^n({\bold{w}})$$
where $m_iw_i=m$ for every $i,\, 0\leq i\leq n+1$.

If ${\bf{w}}=(1,1,\cdots,1)$, this is nothing but the Fermat
hypersurface of dimension $n+1$ and degree $m$.

\begin{defn}
{\rm (a) Let $\KK=\QQ(\zeta_m)$ be the $m$--th cyclotomic field
over $\QQ$, ${\OO}_{\KK}$ the ring of integers of $\KK$.  Let ${\fp}\in
\mbox{Spec}({\OO}_{\KK})$. For every $x\in {\OO}_{\KK}$
relatively prime to ${\fp}$, let
$\chi_{\fp}(x\ \mbox{mod}\, {\fp})=(\frac{x}{\fp})$
be the $m$--th power residue symbol on $\KK$.  If $x\equiv 0
\pmod{\fp}$, we put $\chi_{\fp}(x\ \mbox{mod}\, {\fp})=0$.
Let $(w_0,w_1,w_2,\cdots, w_{n+1})$ be a weight.  Define the set
\begin{equation*}
\begin{split}
& {\fA}_d(w_0,w_1,\cdots, w_{n+1}) \\
&:=\biggl\{ {\ba}=(a_0,a_1,\cdots,a_{n+1})\,|\, a_i\in (w_i\ZZ/m\ZZ),
a_i\neq 0, \sum_{i=0}^{n+1} a_i=0\biggr\}.
\end{split}
\end{equation*}
For each $\ba=(a_0,a_1,\cdots,a_{n+1})\in\fA_d(w_0,w_1,\cdots, w_{n+1})$, 
the {\it weighted Jacobi sum} is defined by
\begin{equation*}
j_{\fp}(\ba)=j_{\fp}(a_0,a_1,\cdots, a_{n+1})=(-1)^n
\sum\chi_{\fp}(v_1)^{a_1}\chi_{\fp}(v_2)^{a_2}\cdots\chi_{\fp}(v_{n+1})^{a_{n+1}}\end{equation*}
where the sum is taken over $(v_1,v_2,\cdots, v_{n+1})\in
({\OO}_{\KK}/{\fp})^{\times} \times \cdots \times
({\OO}_{\KK}/{\fp})^{\times}$ subject to the linear
relation $1+v_1+v_2+\cdots + v_{n+1}\equiv 0 \pmod{\fp}$.

Weighted Jacobi sums are elements of ${\OO}_{\KK}$ with complex absolute
value equal to $q^{n/2}$ where $q=\mid \mbox{Norm}\,{\fp} \mid \equiv 1
\pmod{m}$.

(b) The Galois group $\Gal(\KK/\QQ)\simeq(\ZZ/m\ZZ)^{\times}$ acts
on weighted Jacobi sums, by multiplication by $t\in(\ZZ/m\ZZ)^{\times}$
on each component of $\ba$. Let $A$ denote the
$(\ZZ/m\ZZ)^{\times}$-orbit of $\ba$.  For a weighted Jacobi sum
$j_{\fp}(\ba)$, the $(\ZZ/m\ZZ)^{\times}$-orbit of $j_{\fp}(\ba)$
is called the weighted {\it Fermat motive}, and denoted 
by $\MM_A$.

To each $\ba=(a_0,a_1,\cdots,a_{n+1})\in\fA_d(w_0,w_1,\cdots,w_{n+1})$,
define the {\it length} of ${\ba}$ to be
$$\Vert \ba\Vert:=\left(\frac{1}{m}\sum_{i=0}^{n+1}a_i\right) -1.$$

Via cohomological realizations of these motives, we can compute
the numerical characters of $\MM_A$.

$\bullet$ The $i$-th Betti number is 
$$B_i(\MM_A):=\mbox{dim}_{\Ql} H^i(\overline{\MM_A},\Ql)=\begin{cases}
\#A\quad&\mbox{if $i=n$} \\
1\quad&\mbox{if $i$ is even and $A=[0]$}\\
0\quad&\mbox{otherwise.}
\end{cases}$$

$\bullet$ The $(i,j)$-th Hodge number is
$$h^{i,j}(\MM_A):=\mbox{dim}_{\CC} H^j(\overline{\MM_A},\Omega^i)=
\begin{cases}
\#\{\ba\in A\,|\, ||\ba||=i\}\quad &\mbox{if $i+j=n$}\\
1\quad&\mbox{ if $A=[0]$}\\
0\quad&\mbox{otherwise}\end{cases}
$$
where we put $\overline{\MM_A}:=\MM_A\otimes {\CC}.$}
\end{defn}

For the Fermat hypersurface of dimension $n+1$ and degree $m$, we
simply write $\fA_n$ for $\fA(1,1,\cdots)$.

\begin{lem}
{\sl {\em (a)} Let $S$ be a $K3$ surface of degree $d$ in a weighted projective 
$3$-space $\PP^3(w_0,w_1,w_2,w_3)$ and suppose that $S$ is
dominated by a Fermat surface (so $S$ is of CM type). Then there is the 
unique motive
$\MM_{\bf w}$ associated to the weight ${\bf w}=(w_0,w_1,w_2,w_3)$ such that
$h^{0,2}(\MM_{\bf w})=1$ and $B_2(\MM_{\bf w})=\varphi(d)$.
For all other motives $h^{0,2}(\MM_A)=0$.

{\em (b)} Let $X$ be a Calabi--Yau threefold of degree $d$ in a weighted
projective $4$-space $\PP^4(w_0,w_1,w_2,w_3,w_4)$, and suppose that
$X$ is dominated by a Fermat threefold (so $X$ is of CM type). Then 
there is the unique motive $\MM_{\bf w}$ assocaited to the weight
${\bf w}=(w_0,w_1,\cdots,w_4)$ such that $h^{0,3}(\MM_{\bf w})=1$
and $B_3(\MM_{\bf w})=\varphi(d)$.
For all other motives, $h^{0,3}(\MM_A)=0$.  

Here $\varphi$ denotes the Euler $\varphi$-function.} 
\end{lem}

\begin{prop} {\sl Under the situation of Lemma 7.5,
the following assertions hold.
 
(a) The Fermat motive $\MM_{\bf w}$ associated to the weight contains
the $K3$ motive $\MM_S$ as a submotive.

(b) The Fermat motive $\MM_{\bf w}$ associated to the weight contains 
the Calabi--Yau motive $\MM_X$ as a submotive.} 
\end{prop}

\begin{proof}  $S$ is realized as the quotient of a
Fermat surface ${\mathcal{F}}_m$ by some finite subgroup $G$
of the automorphism group of ${\mathcal{F}}_m$. That is,
$S$ is birationally equivalent to ${\mathcal{F}}_m/G$. Furthermore, 
the transcendental part of $H^2(S)$ can be identified with the 
transcendental part of $H^2({\mathcal{F}}_m)$ that is invariant
under $G$. We know that $H^2({\mathcal{F}}_m)$ is a dirct
sum of one-dimensional subspaces. The Fermat motive associated to the weight is
the unique motive of Hodge type $(0,2)$, and hence its $G$-invariant
transcendental part must contain the $K3$ motive, the unique
motive $\MM_S$ of $S$ of Hodge type $(0,2)$. 

Similarly for the Calabi--Yau motive $\MM_X$, it 
corresponds to the tensor product $E\otimes T(S)^{\sigma}$
and it must be contained in the $G$-invariant part of the Fermat
motive associated to the weight, which is the
unqiue motive of $X$ of Hodge type $(0,3)$. 
\end{proof}

(Compare the Fermat motive associated to the weight to the
$\Omega$ motive defined by Schimmrigk in \cite{S13}.)

\begin{prop}
{\sl Let $(S,\sigma)$ be one of the $86$ $K3$ surfaces with
involution $\sigma$ defined in {\em Theorem \ref{thm-DelsarteS}} 
by a hypersurface over $\QQ$ in a weighted projective $3$-space
$\PP^3(w_0,w_1,w_2,w_3)$.  Then $S$ is of CM type.
The $L$-series of $S$ is determined by the
Jacobi sum Grossencharacter of some cyclotomic field
$\KK:=\QQ(\zeta_d)$ over $\QQ$.  

(a) Let ${\bf w}=(w_0,w_1,w_2,w_3)$ be the weight defining $S$, 
and let $\MM_{\bf w}$ be the unique motive associated to ${\bf w}$.
Let $j_{\fp}(\bf{w})$ be the Jacobi sum
associated to it.  Then $j_{\fp}(\bf{w})$ is an algebraic
integer in $\OO_{\KK}$ with absolute value 
$\mid\mbox{Norm}\,{\fp}\mid$.
The motive $\MM_{\bf w}$ associated to ${\bf w}$ is transcendental
and  corresponds to the single $(\ZZ/d\ZZ)^{\times}$-orbit of 
$j_{\fp}(\bf{w})$. 
Therefore the Galois representation associated to $\MM_{\bf w}$ is 
induced by a $GL_1$ automorphic representation of $\KK=\QQ(\zeta_d)$,
and it is irreducible over $\QQ$ of dimension $\varphi(d)$.
Consequently, the Galois representation of $\MM_S$ is the automorphic
induction of the $GL_1$ Grossencharacter representation of $\KK$.

In other words, $\MM_{\bf w}$ is automorphic,  that is, 
$L(\MM_{\bf w}, s)$ is determined by an automorphic representation over
$\QQ$.
 
(b) Let $\MM_A$ be a motive associated to $S$ other
than $\MM_{\bf w}$. Then $\MM_A$
is automorphic, that is, $L(\MM_A,s)$ is the Artin $L$-function
determined by an automorphic representation over $\QQ$.}
\end{prop}

\begin{proof} Our $K3$ surface $S$ is defined by a
hypersurface of degree $d$ over $\QQ$ in a weighted projective $3$-space. 
$S$ is of CM type.  The characteristic polynomial of the Frobenius of the 
motive $\MM_{\bf w}$ has reciprocal roots $j_{\fp}(\bf w)$ 
and its Galois 
conjugates, that is, the $(\ZZ/d\ZZ)^{\times}$-orbit of $j_{\fp}(\bf w)$. 
We know that $j_{\fp}(\bf w)$ and its Galois conjugates are elements of 
the cyclotomic field $\KK=\QQ(\zeta_d)$.  The restriction  
$\mbox{Gal}(\bar{\QQ}/\KK)$ is a sum of $GL_1$-dimensional representations 
corresponding to Jacobi sum Grossencharacters. %So $\MM_{\bf w}$ is 
%$G_1$-modular over $\KK$. Then the Arthur--Clozel 

For (a), the automorphic induction process yields the automorphic 
representation 
$I(\MM_{\bf w})$ in $\fA(GL_{\varphi(d)}(\QQ))$, which is irreducible over 
$\QQ$ of dimension $\varphi(d)$.  Therefore, $\MM_{\bf w}$ is automorphic.

For (b), a similar argument establishes the automorphy of 
$\MM_A$.
\end{proof}

Now we consider Calabi--Yau threefolds of Borcea--Voisin type,
and establish their automorphy.

\begin{lem} {\sl Let $\KK=\QQ(\zeta_d)$ be the $d$-th cyclotomic
field over $\QQ$.
Let $\psi$ be a Jacobi sum Grossencharacter of $\KK$.
Let $\phi\in \fA(GL_2(\bA_{\QQ}))$ be an automorphic
representation.  Then there is the base change representation
$B_{\KK/\QQ}(\phi)\in \mathfrak{A}(GL_2(\bA_{\KK}))$.

Furthermore, there exists the automorphic induction
$$I(B_{\KK/\QQ}(\phi))\otimes \psi\in 
\fA(GL_{2\varphi(d)}(\bA_{\QQ})).$$

Finally, the Rankin--Selberg $L$-series is given by
$$L(s, \psi\otimes B_{\KK/\QQ}(\phi))
=L(s, I(B_{\KK/\QQ}(\phi))\otimes \psi).$$}
\end{lem}

\begin{prop} {\sl Let $d_1, d_2\in\NN$.
Let $\KK_1=\QQ(\zeta_{d_1})$ and $\KK_2=\QQ(\zeta_{d_2})$
be $d_1$-th and $d_2$-th cyclotomic fields over $\QQ$.
Let $\psi_1$ and $\psi_2$ be Jacobi sum Grossencharacters
of $\KK_1$ and $\KK_2$, respectively. Then they are 
automorphic forms in $\fA(GL_1(\bA_{\KK_1}))$
and $\fA(GL_1(\bA_{\KK_2}))$, respectively.
Consider the induced automorphic representation
$\psi_1\otimes \psi_2$.

\begin{itemize}
\item[(a)] If $\KK_1\nsupseteq \KK_2$, then $\psi_1\otimes \psi_2$
corresponds to an automorphic representation in 
$\fA(GL_1(\bA_{\KK_1\KK_2}))$,
and
$$L(s,\psi_1\otimes\psi_2)=L(s,\psi_1)L(s,\psi_2).$$
\item[(b)] If $\KK_1=\KK_2$, then $\psi_1\otimes\psi_2$
corresponds to the induced automorphic representation
$I(\psi_1\otimes\psi_2)$
in $\fA(GL_2(\bA_{\KK_1}))$, and
$$L(s,\psi_1\otimes\psi_2)=L(s,I(\psi_1\otimes\psi_2)).$$
\item[(c)] If $\KK_1\supset \KK_2$ but $\KK_1\neq \KK_2$, then
$\psi_2$ corresponds to a representation of a subgroup, $H$,
of $\mbox{Gal}(\KK_1/\QQ):=G$, and $\psi_1\otimes \psi_2$
corresponds to the induced representation in
$\fA(GL_1(\bA_{\KK_1}))$,
and
$$L(s,\psi_1\otimes\psi_2)=L(s, \psi_1\otimes\mbox{Ind}^G_H
\psi_2).$$
\end{itemize}
}
\end{prop}
                  
To prove these results, we apply the base change and
automorphic induction method of Arthur and Clozel (and
Rajan) to our situation. Also, confer Murty \cite{M93}.  

\section{Tables} \label{tables}

Some clarifications might be in order how to read the tables.

\begin{itemize} 
\item In the Table \ref{table1}, \ref{table2} and \ref{table3}, 
we use two numbering systems, one from Borcea $B\#$ and the 
other from Yonemura $Y\#$. We matched up the numbers in two lists.
%\smallskip
\item We use two sets of notations for variables, one is
$x_0,x_1,\cdots, x_n$, and the other is $x,y,z,w,\cdots$.  
In relevant tables, we indicated identification of these two sets 
of variables.
%\smallskip
\item The equations in Tables \ref{table1}, \ref{table2} and 
\ref{table3} are taken from Borcea's paper~\cite{B94}. ``Terms removed'' 
indicates the terms we may remove from the equations (or 
deformation) in \cite{B94} to create those of Fermat or Delsarte type. 
%\smallskip
\item The equations in Table \ref{table-y2} are 
taken from Yonemura's paper~\cite{Yo89}. ``Terms removed'' indicates the 
terms we may remove to specialize the equations into Delsarte type. 
In Case \#1, there is no choice of equations of the form 
$x_0^2=f(x_1,x_2,x_3)$ or $x_0^2x_i=f(x_1,x_2,x_3)$. 
%\smallskip 
\item In Table \ref{table-y4}, the equations are taken from Yonemura's 
paper~\cite{Yo89}. In order to make the equations into Delsarte type, we 
slightly generalize the original equations and then remove some terms, if necessary. 
In other words, we first add a few terms to the original equation and then 
remove several terms to make the equation into a form of Delsarte type. This 
procedure is indicated as ``terms changed/removed.'' 

For instance, in the case $\#17$, we first add a term $x^2y$ 
and then remove $x^3$, $xw^5$ and $yw^5$. In effect, this procedure interchanges 
$x^3$ with $x^2y$, and remove $xw^5$ and $yw^5$. It results in a new equation 
$x^2y+y^3+z^5+zw^6$. 
%\smallskip
\item In Table \ref{table-y6}, the equations are taken from Yonemura's 
paper~\cite{Yo89}. For the $K3$ surfaces on this table, there is no way to 
define the involution $\sigma (x)=-x$ by using equations of Delsarte 
type. We therefore define an involution on some other variable. 
This alternative involution is indicated in the column ``involution.'' 
Nikulin's invariants $r$ and $a$ for such $(S,\sigma )$ are calculated 
in Table \ref{table-extra}. 
%\smallskip
\item The $K3$ surfaces in Table \ref{table-y5} are taken from Yonemura's 
paper~\cite{Yo89}. They have involution $\sigma (x)=-x$, but no matter what 
terms we add or remove from the equations, we cannot transform them into 
quasi-smooth equations of Delsarte type (cf. Remark \ref{rmk-table8}).  
%\smallskip
\item The $K3$ surfaces in Table \ref{table-y1} are also taken from 
Yonemura's paper~\cite{Yo89}. For them, we do not know how to define 
a non-symplectic involution on the $K3$ surfaces 
by preserving their quasi-smoothness (even if we allow more than four 
monomials in the equations). 
%\smallskip 
\item Table \ref{table-extra} lists $K3$ surfaces with involution 
at $x_1$, $x_2$ or $x_3$ (i.e., not at the variable $x_0$ of highest 
weight). The variable we choose is indicated under the 
column $\sigma (x_i)=-x_i$. For some $K3$ surfaces, we consider two 
involutions. 
\end{itemize} 
\bigskip 

\begin{table}[ht]
\caption{$K3$ weights in Borcea's list with odd $w_0$}
\label{table1}
$$
\begin{array}{r|r|l|l|r|r|r}\hline \hline
Y\#& B\# & (w_0,w_1,w_2,w_3) & f(x_1,x_2,x_3)=f(y,z,w) & r & a & \mbox{ 
\begin{tabular}{l} 
terms removed from \\   
equations of [B] 
\end{tabular}} \\ \hline
5 & 1  & (3,1,1,1) & y^6+z^6+w^6  & 1 & 1 & \\ \hline
6 & 2  & (5,2,2,1) & y^5+z^5+w^{10} & 6 & 4 & \\ \hline
42 & 3  & (5,3,1,1) & y^3z+z^{10}+w^{10} & 3 & 1 & y^3w  \\ \hline
32 & 4  & (7,3,2,2) & y^4z+z^7+w^7 & 10 & 6  & y^4w \\ \hline
40 & 5  & (7,4,2,1) & y^3z+z^7+w^{14} & 7 & 3 & y^3w^2 \\ \hline
33 & 6  & (9,4,3,2) & y^4w+z^6+w^9 & 10 & 6 &y^3z^2 \\ \hline
39 & 7  & (9,5,3,1) & y^3z+z^6+w^{18} & 7 & 3 &y^3w^3 \\ \hline
12 & 8  & (9,6,2,1) & y^3+z^9+w^{18} & 6 & 2  & \\ \hline
75 & 9  & (11,5,4,2) & y^4w+z^5w+w^{11} & 13 & 5 &y^2z^3 \\ \hline
78 & 10 & (11,6,4,1) & y^3z+yz^4+w^{22} & 10 & 2 &y^3w^4, z^5w^2 \\ \hline
82 & 11 & (11,7,3,1) & y^3w+yz^5+w^{22} & 9 & 1 & z^7w \\ \hline
76 & 12 & (13,6,5,2) & y^4w+yz^4+w^{13} & 14 & 4 & z^4w^3\\ \hline
77 & 13 & (13,7,5,1) & y^3z+z^5w+w^{26} & 11 & 1 &y^3w^5 \\ \hline
81 & 14 & (13,8,3,2) & y^3w+yz^6+w^{13} & 13 & 3 &z^8w \\ \hline
29 & 15 & (15,6,5,4) & y^5+z^6+yw^6 & 12 & 6 & z^2w^5 \\ \hline
34 & 16 & (15,7,6,2) & y^4w+z^5+w^{15} & 14 & 4 & \\ \hline
38 & 17 & (15,8,6,1) & y^3z+z^5+w^{30} & 11 & 1 & y^3w^6\\ \hline
11 & 18 & (15,10,3,2) & y^3+z^{10}+w^{15} & 10 & 4 & \\ \hline
50 & 19 & (15,10,4,1) & y^3+yz^5+w^{30} & 9 & 1 & z^7w^2 \\ \hline 
90 & 20 & (17,7,6,4) & y^4z+y^2w^5+z^5w+zw^7 & 17 & 3 & 
\mbox{no Delsarte form} \\ \hline
93 & 21 & (17,10,4,3) & y^3z+yz^6+yw^8+z^7w^2+zw^{10} & 16 & 2 & 
\mbox{no Delsarte form} \\ \hline
91 & 22 & (19,8,6,5) & y^4z+yz^5+yw^6+z^3w^4 & 18 & 2 & 
\mbox{no Delsarte form} \\ \hline
92 & 23 & (19,11,5,3) & y^3z+yw^9+z^7w & 17 & 1 & zw^{11} \\ \hline
47 & 24 & (21,14,4,3) & y^3+yz^7+w^{14} & 13 & 3 & z^9w^2 \\ \hline
49 & 25 & (21,14,5,2) & y^3+z^8w+w^{21} & 14 & 2 & \\ \hline
14 & 26 & (21,14,6,1) & y^3+z^7+w^{42} & 10 & 0 & \\ \hline 
73 & 27 & (25,10,8,7) & y^5+yz^5+zw^6 & 19 & 1 & \\ \hline
83 & 28 & (27,18,5,4) & y^3+yw^9+z^{10}w & 17 & 1 & z^2w^{11}\\ \hline
46 & 29 & (33,22,6,5) & y^3+z^{11}+zw^{12} & 18 & 0 & \\ \hline
\hline
\end{array}$$
\end{table}

\bigskip

\begin{table}[ht]
\caption{$K3$ weights in Borcea's list with even $w_0$}
\label{table2}
$$
\begin{array}{r|r|l|l|r|r|r}\hline \hline
Y\# & B\# & (w_0,w_1,w_2,w_3) & f(x_1,x_2,x_3)=f(y,z,w) & r & a & \mbox{ 
\begin{tabular}{l} 
terms removed from \\   
equations of [B] 
\end{tabular}} \\ \hline
7 & 30 & (4,2,1,1) & y^4+z^8+w^8 & 2 & 2 & \\ \hline
37 & 31 & (8,4,3,1) & y^4+yz^4+w^{16} & 6 & 4 & z^5w \\ \hline
44 & 32 & (8,5,2,1) & y^3w+z^8+w^{16} & 6 & 2 & y^2z^3 \\ \hline
36 & 33 & (10,5,3,2) & y^4+yz^5+w^{10} & 8 & 6 & z^6w \\ \hline
9 & 34 & (10,5,4,1) & y^4+z^5+w^{20} & 6 & 4 & \\ \hline
35 & 35 & (14,7,4,3) & y^4+z^7+yw^7 & 10 & 6 & zw^8\\ \hline
45 & 36 & (14,9,4,1) & y^3w+z^7+w^{28} & 10 & 0 & \\ \hline
74 & 37 & (16,7,5,4) & y^4w+yz^5+w^8 & 14 & 4 & z^4w^3 \\ \hline
79 & 38 & (16,9,5,2) & y^3z+z^6w+w^{16} & 14 & 2 & y^2w^7 \\ \hline
30 & 39 & (20,8,7,5) & y^5+z^5w+w^8 & 14 & 4 & \\ \hline
80 & 40 & (22,13,5,4) & y^3z+z^8w+w^{11} & 18 & 0 & \\ \hline
\hline
\end{array}$$
\end{table}

\bigskip

\begin{table}[ht]
\caption{$K3$ weights in Borcea's list with $w_0$ divisible by 6}
\label{table3}
$$
\begin{array}{r|r|l|l|r|r|r}\hline \hline
Y\# & B\# & (w_0,w_1,w_2,w_3) & f(x_1,x_2,x_3)=f(y,z,w) & r & a & \mbox{ 
\begin{tabular}{l} 
terms removed from \\   
equations of [B] 
\end{tabular}} \\ \hline
8 & 41 & (6,3,2,1) & y^4+z^6+w^{12} & 4 & 4 & \\ \hline
10 & 42 & (6,4,1,1) & y^3+z^{12}+w^{12} & 2 & 0 & \\ \hline
31 & 43 & (12,5,4,3) & y^4z+z^6+w^8 & 10 & 6 & y^3w^3 \\ \hline
41 & 44 & (12,7,3,2) & y^3z+z^8+w^{12} & 10 & 4 & y^2w^5 \\ \hline
13 & 45 & (12,8,3,1) & y^3+z^8+w^{24} & 6 & 2 & \\ \hline
43 & 46 & (18,11,4,3) & y^3w+z^9+w^{12} & 14 & 2 & \\ \hline
51 & 47 & (18,12,5,1) & y^3+z^7w+w^{36} & 10 & 0 & \\ \hline
48 & 48 & (24,16,5,3) & y^3+z^9w+w^{16} & 14 & 2 & \\ \hline
\hline
\end{array}$$
\end{table}

\bigskip

\begin{table}[ht]
\caption{Delsarte-type $K3$ surfaces with involutions $\sigma (x)=-x$,
NOT in Borcea's list, after removal of several terms}
\label{table-y2}
$$\begin{array}{r|l|l|r|r|r}\hline \hline
Y\# & (w_0,w_1,w_2,w_3) & F(x_0,x_1,x_2,x_3)=F(x,y,z,w) & r & a & \mbox{ 
\begin{tabular}{l} 
terms removed from \\   
equations of [Y] 
\end{tabular}} \\ \hline
1 & (1,1,1,1) & x^4+y^4+z^4+w^4 & 8 & 8 & \\ \hline
19 & (3,2,2,1) & x^2y+y^4+z^4+w^8 & 10 & 6 & x^2w^2 \mbox{ and } x^2z\\ \hline
20 & (9,8,6,1) & x^2z+y^3+z^4+w^{24} & 10 & 6 & x^2w^6\\ \hline
21 & (2,1,1,1) & x^2y+y^5+z^5+w^5 & 6 & 4 & x^2z, x^2w\\ \hline
22 & (6,5,3,1) & x^2z+y^3+z^5+w^{15} & 10 & 4 & x^2w^3\\ \hline
23 & (5,3,2,2) & x^2z+y^4+z^6+w^6 & 12 & 6 & x^2w\\ \hline
24 & (5,4,2,1) & x^2z+y^3+z^6+w^{12} & 10 & 4 & x^2w^2\\ \hline
25 & (4,3,1,1) & x^2z+y^3+z^9+w^9 & 6 & 2 & x^2w \\ \hline
26 & (9,5,4,2) & x^2w+y^4+z^5+w^{10} & 14 & 4 &  \\ \hline
27 & (11,8,3,2) & x^2w+y^3+z^8+w^{12} & 14 & 2 &\\ \hline
28 & (10,7,3,1) & x^2w+y^3+z^7+w^{21} & 11 & 1 & \\ \hline
55 & (7,6,5,2) & x^2y+y^3w+z^4+w^{10} & 14 & 4 & x^2w^3\\ \hline
56 & (11,8,6,5) & x^2y+y^3z+z^5+w^6 & 19 & 1 & \\ \hline 
57 & (9,6,5,4) & x^2y+y^4+z^4w+w^6 & 18 & 2 & xz^3 \\ \hline
58 & (6,5,4,1) & x^2z+y^3w+z^4+w^{16} & 14 & 2 & x^2w^4, xy^2 \\ \hline 
59 & (8,7,5,1) & x^2z+y^3+z^4w+w^{21} & 14 & 2 & x^2w^5\\ \hline
60 & (7,6,4,1) & x^2z+y^3+yz^3+w^{18} & 13 & 3 & x^2w^4, z^4w^2\\ \hline
61 & (11,7,6,4) & x^2z+y^4+z^4w+w^7 & 18 & 2 & \\ \hline 
62 & (8,5,4,3) & x^2z+y^4+yw^5+z^5 & 14 & 4 & xw^4, z^2w^4 \\ \hline
63 & (4,3,2,1) & x^2z+y^3w+z^5+w^{10} & 10 & 4 & x^2w^2, xy^2, y^2z^2 \\ \hline
64 & (10,7,4,3) & x^2z+y^3w+z^6+w^8 & 18 & 0 & xy^2 \\ \hline 
65 & (14,11,5,3) & x^2z+y^3+z^6w+w^{11} & 18 & 0 & \\ \hline 
66 & (3,2,1,1) & x^2z+y^3w+z^7+w^7 & 7 & 3 & x^2w, xy^2, y^3z \\ \hline
67 & (9,7,3,2) & x^2z+y^3+yw^7+z^7 & 13 & 3 & xw^6, zw^9 \\ \hline 
68 & (13,10,4,3) & x^2z+y^3+yz^5+w^{10} & 17 & 1 & z^6w^2\\ \hline 
69 & (7,4,3,2) & x^2w+y^4+yz^4+w^8 & 14 & 4 & xz^3, z^4w^2 \\ \hline
70 & (8,5,3,2) & x^2w+y^3z+z^6+w^9 & 14 & 2 & xy^2, y^2w^4 \\ \hline
71 & (7,4,3,1) & x^2w+y^3z+z^5+w^{15} & 11 & 1 & xy^2, y^3w^3 \\ \hline
72 & (7,5,2,1) & x^2w+y^3+yz^5+w^{15} & 9 & 1 & xz^4, z^7w \\ \hline
86 & (9,7,5,4) & x^2y+y^3w+z^5+zw^5 & 19 & 1 & xw^4 \\ \hline
87 & (5,4,3,1) & x^2z+y^3w+yz^3+w^{13} & 13 & 3 & x^2w^3, xy^2, z^4w \\ \hline
88 & (11,9,5,2) & x^2z+y^3+yw^9+z^5w & 17 & 1 & xw^8, zw^{11} \\ \hline
89 & (5,3,2,1) & x^2w+y^3z+yz^4+w^{11} & 10 & 2 & xy^2, xz^3, y^3w^2, z^5w\\ \hline 
\hline
\end{array}$$
\end{table} 

\begin{table}[ht]
\caption{$K3$ surfaces with involution $\sigma (x)=-x$, NOT in Borcea's list, 
after change and/or removal of several terms}
\label{table-y4}
$$\begin{array}{r|l|l|r|r|r}\hline \hline
Y\# & (w_0,w_1,w_2,w_3) & F(x_0,x_1,x_2,x_3)=F(x,y,z,w) & r & a & \mbox{ 
\begin{tabular}{l} 
terms changed/removed \\   
 from equations of [Y] 
\end{tabular}} \\ \hline
3 & (2,2,1,1) & x^2y+y^3+z^6+w^6 & 7 & 7 & x^3\rightarrow x^2y \\ \hline
4 & (4,4,3,1) & x^2y+y^3+z^4+w^{12} & 7 & 7 & x^3\rightarrow x^2y \\ \hline
17 & (5,5,3,2) & x^2y+y^3+z^5+zw^6 & 12 & 6 & xw^5, yw^5, x^3\rightarrow x^2y \\ \hline
18 & (3,3,2,1) & x^2y+y^3+z^4w+w^9 & 10 & 6 & xz^3, yz^3, x^3\rightarrow x^2y \\ \hline
\hline
\end{array}$$
\end{table}

\begin{table}[ht]
\caption{$K3$ surfaces with a different kind of involution}
\label{table-y6}
$$\begin{array}{r|l|l|r|r|r|r}\hline \hline
Y\# & (w_0,w_1,w_2,w_3) & F(x_0,x_1,x_2,x_3)=F(x,y,z,w) & r & a & \mbox{ 
\begin{tabular}{l} 
terms \\  
removed 
\end{tabular}} & \mbox{involution}  \\ \hline
2 & (4,3,3,2) & x^3+y^4+z^4+w^6 & 10 & 8 & \mbox{none} & y\to -y \\ \hline
16 & (8,7,6,3) & x^3+y^3w+z^4+w^8 & 14 & 6 & \mbox{none} & z\to -z \\ \hline
52 & (12,9,8,7) & x^3+y^4+xz^3+zw^4 & 19 & 3 & \mbox{none} & y\to -y \\ \hline
84 & (9,7,6,5) & x^3+xz^3+y^3z+yw^4 & 20 & 2 & z^2w^3 & w\to -w \\ \hline
\hline
\end{array}$$
\end{table}

\bigskip 

\begin{table}[ht]
\caption{$K3$ surfaces with involution $\sigma (x)=-x$, but not realized as
quasi-smooth hypersurfaces in 4 monomials}
\label{table-y5}
$$\begin{array}{r|l|l|l|l}\hline \hline
Y\# & (w_0,w_1,w_2,w_3) & F(x_0,x_1,x_2,x_3)=F(x,y,z,w) & r & a \\ \hline 
85 & (5,4,3,2) & x^2y+x^2w^2+y^3w+y^2z^2+yw^5+z^4w+w^7 & 15 & 5 \\ \hline 
90 & (17,7,6,4) & x^2+y^4z+y^2w^5+z^5w+zw^7 & 17 & 3 \\ \hline 
91 & (19,8,6,5) & x^2+y^4z+yz^5+yw^6+z^3w^4 & 18 & 2 \\ \hline 
93 & (17,10,4,3) & x^2+y^3z+yz^6+yw^8+z^7w^2+zw^{10} & 16 & 2 \\ \hline 
94 & (7,5,4,3) & x^2y+y^3z+y^2w^3+z^4w+zw^5 & 18 & 2 \\ \hline
95 & (7,5,3,2) & x^2z+y^3w+yz^4+yw^6+z^5w+zw^7 & 16 & 2 \\ \hline
\hline
\end{array}$$
\end{table}

\begin{table}[ht]
\caption{$K3$ weights with no obvious involution}
\label{table-y1}
$$\begin{array}{r|l|l}\hline \hline
Y\# & (w_0,w_1,w_2,w_3) & F(x_0,x_1,x_2,x_3)=F(x,y,z,w) \\ \hline
15 & (5,4,3,3) & x^3+y^3z+y^3w+z^5+w^5 \\ \hline
53 & (6,5,4,3) & x^3+y^3w+y^2z^2+xz^3+z^3w^2+w^6 \\ \hline
54 & (7,6,5,3) & x^3+y^3w+yz^3+z^3w^2+w^7 \\ \hline
\hline
\end{array}$$
\end{table}

\begin{table}[ht]
\caption{Nikulin's invariant associated with other types of 
involutions}
\label{table-extra}
$$\begin{array}{r|l|l|c|r|r}\hline \hline
Y\# & (w_0,w_1,w_2,w_3) & F(x_0,x_1,x_2,x_3)=F(x,y,z,w) & 
\sigma (x_i)=-x_i & r & a \\ \hline
2 & (4,3,3,2) & x^3+y^4+z^4+w^6 & y & 10 & 8 \\ \hline 
 & (4,3,3,2) & x^3+y^4+z^4+w^6 & w & 18 & 4 \\ \hline
3 & (2,2,1,1) & x^2y+y^3+z^6+w^6 & z & 10 & 8 \\ \hline 
4 & (4,4,3,1) & x^2y+y^3+z^4+w^{12} & z & 14 & 6 \\ \hline
5 & (3,1,1,1) & x^2+y^6+z^6+w^6 & y & 9 & 9 \\ \hline
6 & (5,2,2,1) & x^2+y^5+z^5+w^{10} & w & 6 & 4 \\ \hline
 &  (5,2,2,1) & x^2+y^5+yz^4+w^{10} & z & 10 & 8 \\ \hline 
7 & (4,2,1,1) & x^2+y^4+z^8+w^8 & y & 10 & 6 \\ \hline
 & (4,2,1,1) & x^2+y^4+z^8+w^8 & w & 10 & 8 \\ \hline
8 & (6,3,2,1) & x^2+y^4+z^6+w^{12} & y & 12 & 6 \\ \hline 
 & (6,3,2,1) & x^2+y^4+z^6+w^{12} & z & 12 & 8 \\ \hline 
9 & (10,5,4,1) & x^2+y^4+z^5+w^{20} & y & 14 & 4 \\ \hline 
 & (10,5,4,1) & x^2+y^4+z^5+w^{20} & w & 14 & 4 \\ \hline  
10 & (6,4,1,1) & x^2+y^3+z^{12}+w^{12} & z & 10 & 8 \\ \hline
12 & (9,6,2,1) & x^2+y^3+z^9+w^{18} & w & 6 & 2 \\ \hline
13 & (12,8,3,1) & x^2+y^3+z^8+w^{24} & z & 14 & 6 \\ \hline
16 & (8,7,6,3) & x^3+y^3w+z^4+w^8 & z & 14 & 6 \\ \hline
17 & (5,5,3,2) & x^2y+y^3+z^5+zw^6 & w & 17 & 5 \\ \hline
18 & (3,3,2,1) & x^2y+y^3+z^4w+w^9 & z & 14 & 6 \\ \hline 
19 & (3,2,2,1) & x^2y+y^4+z^4+w^8 & z & 10 & 8 \\ \hline
23 & (5,3,2,2) & x^2z+y^4+z^6+w^6 & w & 12 & 6 \\ \hline
29 & (15,6,5,4) & x^2+y^5+z^6+yw^6 & w & 18 & 4 \\ \hline 
31 & (12,5,4,3) & x^2+y^4z+z^6+w^8 & y & 18 & 4 \\ \hline
33 & (9,4,3,2) & x^2+y^4w+z^6+w^9 & y & 14 & 6 \\ \hline
 & (9,4,3,2) & x^2+y^4w+z^6+w^9 & z & 10 & 6 \\ \hline
36 & (10,5,3,2) & x^2+y^4+yz^5+w^{10} & w & 16 & 6 \\ \hline
37 & (8,4,3,1) & x^2+y^4+yz^4+w^{16} & z & 10 & 8 \\ \hline
39 & (9,5,3,1) & x^2+y^3z+z^6+w^{18} & w & 15 & 7 \\ \hline
40 & (7,4,2,1) & x^2+y^3z+z^7+w^{14} & w & 7 & 3 \\ \hline
41 & (12,7,3,2) & x^2+y^3z+z^8+w^{12} & w & 15 & 7 \\ \hline
42 & (5,3,1,1) & x^2+y^3z+z^{10}+w^{10} & w & 11 & 9 \\ \hline
44 & (8,5,2,1) & x^2+y^3w+z^8+w^{16} & z & 14 & 6 \\ \hline
52 & (12,9,8,7) & x^3+y^4+xz^3+zw^4 & y & 20 & 2 \\ \hline 
 & (12,9,8,7) & x^3+y^4+xz^3+zw^4 & w & 19 & 3 \\ \hline
75 & (11,5,4,2) & x^2+y^4w+z^5w+w^{11} & y & 13 & 5 \\ \hline
84 & (9,7,6,5) & x^3+xz^3+y^3z+yw^4 & w & 20 & 2 \\ \hline 
\hline
\end{array}$$
\end{table}

\begin{table}[ht]
\caption{$K3$ surfaces and their mirror partners}
\label{table9}
$$\begin{array}{r|l|l|l||r|r|r|l}\hline\hline
S:Y\# & S:B\#  & \mbox{weight} & r & S^{\vee}Y\#& S^{\vee} B\# & 20-r & 
\mbox{weight for mirror} \\ \hline\hline
1 & & (1,1,1,1) & 8 & 56 & & 12 & (11,8,6,5) \\
  & &          &   & 73 & 27& 12 & (25,10,8,7) \\ \hline
4 & & (4,4,3,1) & 10 & 4 & & 10 & (4,4,3,1) \\ \hline
5 & 1& (3,1,1,1) & 1  & 52 & & 19 & (12,9,8,7) \\ \hline
6 & 2 & (5,2,2,1) & 6  & 26 & & 14 & (9,5,4,2) \\
    & &          &  & 34 &16 & 14 & (15,7,6,2) \\
    & &          &  & 76 &12 & 14 & (13,6,5,2) \\ \hline
8 & 41& (6,3,2,1)   & 3 & 64& & 17 & (10,7,4,5) \\ \hline
9 & & (10,5,4,1) & 10 & 9 & & 10 & (10,5,4,1) \\
    &&            &  & 71 & & 10 & (7,4,3,1) \\ \hline
10 &42 & (6,4,1,1) & 2 & 65 & & 18 & (14,11,5,3) \\
     & &          &   & 46 & 29& 18 & (53,22,6,5) \\
     & &          &   & 80 &40 &18 & (22,13,5,4) \\ \hline
11 & 17 & (15,10,3,2) & 12 & 24 & & 8 & (5,4,2,1) \\ \hline
12 & 8& (9,6,2,1)   & 6 & 27 & &14 & (11,8,3,2) \\
     & &          &   & 49 &25 &14 & (21,14,5,2) \\ \hline
13 & 45 &(12,7,3,1) & 8 & 20& & 12 & (9,8,6,1) \\
     &  &       &      & 59 & & 12 & (8,7,5,1) \\ \hline
14 & 26 &(21,14,6,1) & 10 & 14&26 & 10 & (21,14,6,1) \\
     &  &           &  & 28 & & 10 & (10,7,3,1) \\
     &  &           &  & 45 & 36& 10 & (14,9,4,1) \\
     &  &           &  & 51 &47 &10 & (18,12,5,1) \\ \hline
20 & &(9,8,6,1) & 12 & 17 & & 8 &  (12,8,3,1) \\
   &  &           &    & 72& & 8 & (7,5,2,1) \\ \hline
21 && (2,1,1,1) & 2 & 30 &39 &18 & (20,8,7,5) \\
   &  &           &   & 86 & & 18 & (9,7,5,4) \\ \hline 
22 && (6,5,3,1) & 10 & 22 & & 10 & (6,5,3,1) \\ \hline
24 && (5,4,2,1) & 8 & 11 &18 & 12 & (15,10,3,2) \\ \hline
25 && (4,3,1,1) & 8 & 43 &46 & 12 & (18,11,4,3) \\
   &  &           &   & 48 &48 & 12 & (24,16,5,3)\\
   &  &           &   & 88 & & 12 & (11,9.5,2) \\ \hline
26 && (9,5,4,2) & 14 & 6 &2 & 6 & (5,2,2,1) \\ \hline
27 && (11,8,3,2) & 14 & 12&8 & 6 & (9,6,2,1) \\ \hline
28 && (10,7,3,1) & 10 & 14 & 26& 10 & (21,14,6,1) \\
   &&            &    & 28 & &10 & (10,7,3,1) \\
   &&            &    & 45 & 36& 10 & (14,9,4,1) \\
   &&            &    & 51 &47 &10 & (18,12,5,1) \\ \hline
30 &39& (20,8,7,5) & 18 & 21 & & 2 & (2,1,1,1) \\ \hline
32 &4& (7,3,2,2) & 10 & 10 &42 &10 & (7,3,2,2) \\ \hline
34 &16& (15,7,6,2) & 14 & 6 &2 & 6 & (5,2,2,1) \\ \hline
35 & 35& (14,7,4,3) & 16 & 66& & 4 & (3,2,1,1) \\ \hline
37 & 31& (8,4,3,1) & 9 & 58 & & 11 & (6,5,4,1) \\ \hline
38 & 17& (15,8,6,1) & 11 & 50 &19& 9 & (15,10,4,1) \\
   &  &          &    & 82 &11 & 9 & (11,7,3,1) \\ \hline
39 & 7& (9,5,3,1) & 9 & 60 & & 11 & (7,6,4,1) \\ \hline
40 & 5& (7,4,2,1) & 7 & 81 &14& 13 & (13,8,3,2) \\ \hline
\end{array}$$
\end{table}

\begin{table}[ht]
\caption{$K3$ surfaces and their mirror partners (continued)}
\label{table10}
$$\begin{array}{r|l|l|l||r|r|r|l}\hline\hline
S:Y\# & S:B\# & \mbox{weight} & r & S^{\vee}:Y\#& S^{\vee}: B\# & 20-r & \mbox{weight for mirror} \\ \hline\hline
42 &3 & (5,3,1,1) & 3 & 68 && 17 & (13,10,4,3) \\
   &  &         &   & 83 & 28& 17 & (27,18,5,4) \\
   &  &         &   & 92 & 23 &17 & (19,11,5,3) \\ \hline
43 & 46& (18,11,4,3) & 16 & 25& & 4 & (4,3,1,1) \\ \hline
45 & 36& (14,9,4,1) & 10 & 14&26 & 10 & (21,14,6,1) \\
   &   &         &    & 28 & &10 & (10,7,3,1) \\
   &   &         &    & 45 &36& 10 & (14,9,4,1) \\
   &   &         &    & 51 & 47& 10 & (18,12,5,1)\\ \hline
46 & 29& (33,22,6,5) & 18 & 10&42 & 2 & (6,4,1,1) \\ \hline
48 & 48& (24,16,5,3) & 16 & 25 && 4 & (4,3,1,1) \\ \hline
49 & 25& (21,14,5,2) & 14 & 12 &8 & 6 & (9,6,2,1) \\ \hline
50 & 19& (15,10,4,1) & 9 & 38 &17& 11 & (15,8,6,1) \\
   &   &          &   & 77 & 13& 11 & (13,7,5,1) \\ \hline
51 & 47& (18,12,5,1) & 10 & 14 &26& 10 & (21,12,6,1) \\
   &   &          &    & 28 & & 10 & (10,7,3,1) \\
   &   &          &    & 45 & 36& 10 & (14,9,4,1) \\
   &   &          &    & 51 & 47& 10 & (18,12,5,1) \\ \hline
52 & &(12,9,8,7) & 19 & 5 & & 1 & (3,1,1,1) \\ \hline
56 & &(11,8,6,5) & 19 & 1 & & 1 & (1,1,1,1) \\ \hline
58 & &(6,5,4,1) & 11 & 37 &31 & 9 & (8,4,3,1) \\ \hline
59 & &(8,7,5,1) & 12 & 13 & 45& 8 & (12,8,3,1) \\
   & &          &    & 72 & & 8 & (7,5,2,1) \\ \hline
60 & &(7,6,4,1) & 11 & 39 & 7& 9 & (9,5,3.1) \\ \hline
64 & &(10,7,4,3) & 17 & 7 & 30& 3 & (4,2,1,1) \\ \hline
65 & &(14,11,5,3) & 17 & 10 &42 & 3 & (6,4,1,1) \\ \hline
66 & &(3,2,1,1) & 4 & 35 &35 & 16 & (14,7,4,3) \\ \hline
68 & &(13,10,4,3) & 17 & 42&3 & 3 & (5,3,1,1) \\ \hline
71 & &(7,4,3,1) & 10 & 9 &34& 10 & (10,5,4,1) \\
   & &          &    & 71 && 10 & (7,4,3,1) \\ \hline
72 & &(7,5,2,1) & 8 & 20 & & 12 & (9,8,6,1) \\
   & &          &   & 59 && 12 & (8,7,5,1) \\ \hline
73 & 27& (25,10,8,7) & 19 & 1& & 1 & (1,1,1,1) \\ \hline
76 & 12& (13,6,5,2) & 14 & 6 &2& 6 & (5,2,2,1) \\ \hline
77 & 13& (13,7,5,1) & 11 & 50 &19& 9 & (15,10,4,1) \\
   &   &         &    & 82 &11& 9 & (11,7,3,1) \\ \hline
78 & 10& (11,6,4,1) & 10 & 10&42 & 10 & (11,6,4,1) \\ \hline
80 & 40& (22,13,5,4) & 18 & 10 &42& 2 & (6,4,1,1) \\ \hline
81 & 14& (13,8,3,2) & 13 & 40 &5& 7 & (7,4,2,1) \\ \hline
82 & 11& (11,7,3,1) & 9 & 38 &17& 11 & (15,8,6,1) \\
   &    &        &   & 77 &13& 11 & (13,7,5,1) \\ \hline
83 & 28& (27,18,5,4) & 17 & 42&3 & 3 & (5,3,1,1) \\ \hline
86 & &(9,7,5,4) & 18 & 21 && 2 & (2,1,1,1) \\ \hline
87 & &(5,4,3,1) & 10 & 87 && 10 & (1,3,4,5) \\ \hline
92 & 23& (19,11,5,3) & 17 & 42 &3& 3 & (5,3,1,1) \\ \hline
\end{array}$$
\end{table}

\centerline{\bf Acknowledgments}
\medskip

We thank the referees for carefully reading through the
earlier version(s) and having given us constructive
criticism and suggestions for the improvement of the paper.

Y. Goto is supported in part by JSPS grant from Japan (category C, 
21540003 and 24540004). R. Livn\'e is supprted in part by a grant from 
Israel. N. Yui is supported in part by NSERC Discovery grant from
Canada. The paper was completed while N. Yui was a visiting
professor at Kavli IPMU, and at Tsuda College in Japan in 2012,
and subsequently revisions are done at IHES and the Fields
Institute in 2013. She thanks the hospitality of these institutions. 
Y. Goto prepared the manuscript while he visited the Fields Institute 
and Queen's 
University in Canada, Tsuda College in Japan and the National Center 
for Theoretical Sciences of Taiwan. He thanks the hospitality of 
these institutions. 

Finally, we would like to thank  M. Espinosa-Lara, A. Molnar,
A. Rabindranath and Jun Ho Whang for their careful reading
of the earlier version of the manuscript and pointout
numerous typos and inaccuracies. They studied this paper 
as their project under the Fields Undergraduate Research
Program in the summer of 2013.

\end{document}